\documentclass[11pt]{amsart}

\usepackage{amsthm,amsmath,amssymb}
\usepackage{mathrsfs}
\usepackage{bbm}
\usepackage{xcolor}
\usepackage{verbatim}

\usepackage[colorlinks=true,
            citecolor=blue,
            linkcolor=blue,
            urlcolor=blue]{hyperref}

\newtheorem{thm}{Theorem}[section]
\newtheorem{lem}[thm]{Lemma}
\newtheorem{prop}[thm]{Proposition}

\theoremstyle{definition}

\newcommand{\R}{\mathbb{R}}

\newcommand{\h}{\mathrm{h}}

\def\eqdefa{\buildrel\hbox{\footnotesize def}\over =}

\newcommand{\dx}{\,{\rm d}x}
\newcommand{\dt}{\,{\rm d}t}
\newcommand{\nablah}{\nabla_{\mathrm{h}}}

\numberwithin{equation}{section}

\begin{document}

\title[3-D MHD system]{Global Existence for 3D Anisotropic MHD system with Horizontal Dissipation and Small Horizontal Variations}

\author[Q.Lin]{Qiliang Lin}
\address[Q.Lin]%
{Department of Mathematics, Zhejiang  Normal University, Jinhua, 321004, China}
\email{linqiliang$\_$zjnu@163.com }
\author[C. Qian]{Chenyin Qian}
\address[C. Qian]
{Department of Mathematics, Zhejiang  Normal University, Jinhua,
	321004, China}
\email{cyqian@zjnu.edu.cn }
\author[D.ZHou]{Daoyao Zhou}
\address[D.Zhou]%
{Department of Mathematics, Zhejiang  Normal University, Jinhua, 321004, China}
\email{Zhoudaoyao$\_$zjnu@163.com }

\begin{abstract}
This paper establishes the global well-posedness for  the 3D anisotropic MHD system with partial dissipation: $\Delta_\mathrm{h}u$ for velocity and $\partial_1^2b$ for magnetic field, near background field $(0,1,0)$. Crucially, only horizontal components $(u^\mathrm{h}_0,b^\mathrm{h}_0)$ need to be small in $H^2(\R^3)$, while $(u^3_0,b^3_0)$ can be arbitrarily large.
Our analysis develops novel techniques including component-decoupled energies and iterative control of dangerous nonlinearities using the background field structure. This establishes the  global result for anisotropic MHD equations allowing large vertical data, breaking the full-smallness requirement of previous works.
\end{abstract}

\maketitle

\section{Introduction}
The magnetohydrodynamic (MHD) equations are fundamental to modeling plasmas, liquid metals, and astrophysical flows. A central problem in both theory and applications is the stability of a conducting fluid in the presence of a strong, steady background magnetic field---a scenario ubiquitous in fusion devices, solar physics, and planetary dynamos. This stability problem translates directly to studying the global well-posedness of the anisotropic MHD system with only partial dissipation.

Specifically, we consider the 3D incompressible MHD system with anisotropic viscosity and magnetic diffusion:
\begin{align} \label{MHD}
	\left\{\begin{array}{l}
		\displaystyle {\partial_t}u+u\cdot\nabla{u}-\sum_{i=1}^3\mu_i\partial_i^2u =-\nabla {p}+B\cdot\nabla{B},\\
		\displaystyle {\partial_t}B+u\cdot\nabla{B}-\sum_{i=1}^3\nu_i\partial_i^2B =B\cdot\nabla{u},\\
		\displaystyle \operatorname{div}u=\operatorname{div}B=0,\\
		u{|_{t = 0}} =  {u_0},\; B{|_{t = 0}} =  {B_0},
	\end{array}
	\right.
\end{align}
where $u$, $B$, and $p$ denote the velocity, magnetic field, and total pressure, respectively, and $\mu_i, \nu_i > 0$ are anisotropic transport coefficients.

The presence of a background magnetic field $B^*$ introduces Alfv\'{e}n waves and fundamentally alters the dissipation mechanism compared to the Navier-Stokes case. Recent advances have established global well-posedness for partially dissipative systems, such as with dissipation $(\Delta_\mathrm{h}u,\partial_3^2B)$ and $B^*=(1,0,0)$ \cite{24han}, or $(\partial_1^2u,\Delta_\mathrm{h}B)$ with $B^*=(0,1,0)$ \cite{26wu}. However, these results uniformly require the \emph{full} initial perturbation $(u_0, B_0)$ to be small, a severe limitation for realistic flows where vertical motions (such as in buoyancy-driven convection or jet instabilities) can be large.

This raises a fundamental question: can the stabilizing effect of a background field compensate for the lack of full smallness, requiring only the \emph{horizontal} part of the data to be controlled?
{This work provides an affirmative answer.} We establish the global well-posedness of the 3D anisotropic MHD system under a \emph{horizontal smallness condition alone}, allowing the vertical components of the initial velocity and magnetic fields to be arbitrarily large. By exploiting the directional smoothing effects enabled by the background magnetic field and the system's structural anisotropy, we prove that stability is governed primarily by horizontal dynamics. This result breaks the prevailing ``full smallness'' paradigm and addresses a physically natural regime where strong vertical disturbances coexist with controlled horizontal fluctuations.

\subsection{The Model and Main Result}
Denoting $B=b+B^*$, we consider the 3D anisotropic MHD system with background magnetic field $B^*=(0,1,0)$ and dissipation pattern
\begin{align*}
	\mu_1=\mu_2=\nu_1=1,\qquad \mu_3=\nu_2=\nu_3=0,
\end{align*}
so that the system \eqref{MHD} could be reformulated as the perturbation $(u,b)=(u^\mathrm{h},u^3,b^\mathrm{h},b^3)$ which satisfies
\begin{align} \label{eq.MHD}
	\left\{\begin{array}{l}
		\displaystyle {\partial_t}u+u\cdot\nabla{u}-\Delta_\h u =-\nabla{p}+b\cdot\nabla{b}+\partial_2b,\quad x\in\mathbb{R}^3,t>0,\\
		\displaystyle {\partial_t}b+u\cdot\nabla{b}-\partial_1^2b =b\cdot\nabla{u}+\partial_2u,\\
		\displaystyle \operatorname{div}u=\operatorname{div}b=0,\\
		u{|_{t = 0}} =  {u_0},b{|_{t = 0}} =  {b_0}.
	\end{array}
	\right.
\end{align}
The velocity enjoys full horizontal dissipation $\Delta_\mathrm{h}$, while the magnetic field has only $\partial_1^2$ dissipation. The linear terms $\partial_2 b$ and $\partial_2 u$ arise from the background field and provide a weak stabilizing effect.

Our main theorem states:
\begin{thm}\label{main.Th}
	Let $(u_0,b_0)\in H^2(\mathbb{R}^3)$ be divergence-free. There exists an absolute constant $\varepsilon>0$ such that if
	\begin{align}\label{small}
		\|(u^\mathrm{h}_0,b^\mathrm{h}_0)\|_{H^2}^2
		\exp\Bigl(C\|(u^3_0,b^3_0)\|_{H^2}^2\bigl(1+\|(u^3_0,b^3_0)\|_{H^2}^2\bigr)\Bigr)\leqq\varepsilon,
	\end{align}
	then system \eqref{eq.MHD} admits a unique global solution
	\begin{align*}
		(u,b)\in L^\infty(\mathbb{R}^+;H^2(\mathbb{R}^3)),\quad
		(\nabla_\mathrm{h}u,\partial_1b)\in L^2(\mathbb{R}^+;H^2(\mathbb{R}^3)),\quad
		\partial_2b\in L^2(\mathbb{R}^+;H^1(\mathbb{R}^3)).
	\end{align*}
\end{thm}

\noindent \textbf{Key Innovations and Contributions:}
\begin{itemize}
	\item \textbf{Large vertical components.} Condition \eqref{small} allows the vertical components $u^3_0$ and $b^3_0$ to be arbitrarily large in $H^2(\mathbb{R}^3)$; only the horizontal part $(u^\mathrm{h}_0,b^\mathrm{h}_0)$ must be small. As far as we know, this is the first global well‑posedness result for a three‑dimensional anisotropic MHD system with mixed partial dissipation that permits large data in this component‑wise sense.
	\item \textbf{Optimal dissipation pattern.} The chosen dissipation $(\Delta_{\mathrm{h}} u,\; \partial_1^2 b)$ combined with the background field $B^*=(0,1,0)$ is shown to be critical: it generates exactly the weak dissipation needed to control the most dangerous nonlinear terms while still allowing the vertical components to be large.
	\item \textbf{New analytical framework.} We develop a refined energy method that decouples the horizontal and vertical dynamics, treats the pressure terms in a component‑aware fashion, and handles intricate nonlinearities through a combination of anisotropic interpolation, iterative use of the equations, and subtle cancellations.
\end{itemize}


\subsection{Strategy and Innovations}

The proof of Theorem \ref{main.Th} is built upon a novel \emph{component-decoupled energy method} that systematically handles the distinct roles played by horizontal and vertical components. This approach is necessitated by the fundamental tension in our problem: while the horizontal components $(u^\mathrm{h}, b^\mathrm{h})$ are subject to a smallness condition, the vertical components $(u^3, b^3)$ can be arbitrarily large. The strategy unfolds in three main stages, each addressing a core analytical challenge.

\subsubsection{Stage 1: Decoupled Energy Framework}

We construct two separate energy functionals: one for the vertical components and one for the horizontal components, and derive their evolution equations independently. This decoupling is natural due to the anisotropic structure of the dissipation, which acts as $\Delta_\mathrm{h}$ on $u$ but only $\partial_1^2$ on $b$, and the fact that the smallness assumption is imposed solely on horizontal variations. Moreover, the incompressibility conditions $\partial_3 u^3 = -\mathrm{div}_\mathrm{h} u^\mathrm{h}$ and $\partial_3 b^3 = -\mathrm{div}_\mathrm{h} b^\mathrm{h}$ link vertical derivatives to horizontal components, allowing the evolution of vertical quantities to be expressed in terms of horizontal ones.


For instance, a detailed decomposition of the critical term $I_2$ of \eqref{en_H2_v} from the vertical energy estimate reveals:
\begin{align*}
	I_2=&\sum_{i,j=1}^3\int_{\R^3}\partial_i^2b^j\partial_jb^3\partial_i^2u^3+2\partial_ib^j\partial_j\partial_ib^3\partial_i^2u^3	\dx.
\end{align*}
Essentially,   vertical derivatives like $\partial_3^2 u^3$ are expressed via $\partial_3^2 u^3 = -\partial_3\mathrm{div}_\mathrm{h} u^\mathrm{h}$, transferring the regularity requirement to the (small) horizontal components. This \emph{derivative transfer mechanism} is systematically applied to all high-order vertical terms. A meticulous rearrangement using the same method allows us to express $I_2$ as a sum of terms where at most one high-order vertical derivative appears on any single factor. This structure leads to the simplified bound
\begin{align}\label{I2deco}
	I_2\leqq &C\int_{\R^3}|\nablah^2b^3\nablah b^\h \nablah^2 u^3|+|\nablah^2b^\h\nablah b^3 \nablah^2 u^3|+|\partial_3^2b^\h\nablah b^3 \nablah\partial_3 u^\h|\dx\nonumber\\
	&+C\int_{\R^3}|\nablah\nabla b^\h\nabla b^\h \nablah \nabla u^\h|\dx.
\end{align}
Consequently, we find that the vertical energy satisfies
\begin{align*}
	\frac12\frac{\mathrm{d}}{\dt}\|(u^3,b^3)\|_{H^2}^2 + \|(\nabla_\mathrm{h} u^3,\partial_1b^3)\|_{H^2}^2 \lesssim \text{(nonlinear terms)} + \text{(pressure terms)}.
\end{align*}
The terms in the ``nonlinear terms" share a similar structure with \eqref{I2deco}. They can either be controlled by the vertical energy itself or be handled under the smallness condition in the horizontal direction.
\subsubsection{Stage 2: Pressure Analysis}

The pressure term $-\nabla p$ presents a unique difficulty in the decoupled framework. In the full energy estimate, integration by parts combined with $\mathrm{div} u = 0$ eliminates $\int_{\R^3} u\cdot\nabla p \dx$. When estimating components separately, this cancellation no longer occurs. Instead, we have the identity
\begin{align}\label{p_same}
	-(\nabla_\mathrm{h} p  |u^\mathrm{h})_{L^2} = (p  |\mathrm{div}_\mathrm{h} u^\mathrm{h})_{L^2} = -(p  |\partial_3 u^3)_{L^2} = (\partial_3 p | u^3)_{L^2},
\end{align}
revealing that the pressure terms appearing in the horizontal and vertical energy balances are intrinsically linked. Namely, they are essentially the same quantity manifested in different components. It follows that any estimate for the pressure must simultaneously serve both the horizontal and vertical energy budgets. Yet, our smallness assumption applies only to the horizontal components $(u^\h, b^\h)$, preventing a unified control of the pressure term via a single bound. This dichotomy forces us to develop two distinct, component-specific estimates for the pressure (e.g., \eqref{A2(v)} and \eqref{A2(h)}).

To estimate $p$, we solve the pressure equation
\begin{align}\label{ppp}
	-\Delta p = \mathrm{div}\,\mathrm{div}(u\otimes u - b\otimes b),
\end{align}
and obtain explicit representations involving Riesz transforms. The key difficulty arises from terms like
\begin{align}\label{a3}
	\big((-\Delta)^{-1}\partial_3^2(u^3)^2 | \mathrm{div}_\mathrm{h} u^\mathrm{h}\big)_{L^2}.
\end{align}
As mentioned before, we need to obtain an estimate that both ensures a closing estimate for the vertical energy and guarantees a closing estimate for the horizontal energy under the condition \eqref{small}. While the former objective is straightforward, for the latter, the desired ideal estimate we aim for is
\begin{align*}
	\|\nablah u^3\|_{H^2}\|\nablah u^\h\|_{H^2}\|u^\h\|_{H^2} \quad\mathrm{or}\quad \|\nablah u^\h\|_{H^2}^2\|u^\h\|_{H^2}^\frac12\|u^3\|_{H^2}^\frac12.
\end{align*}
Compared with the two expressions above, one observes that \eqref{a3} contains too many vertical component $u^3$. This observation naturally leads us to employ the incompressibility condition to extract the horizontal component $u^\h$, and thus we turn to consider
\begin{align*}
	A_3 = 2\big((-\Delta)^{-1}\partial_3(u^3 \mathrm{div}_\mathrm{h} u^\mathrm{h}) \mid \mathrm{div}_\mathrm{h} u^\mathrm{h}\big)_{L^2}.
\end{align*}
This, however, cannot be simply derived from the following basic one-dimensional inequality alone
\begin{align}\label{L^infty_1D}
	\|f\|_{L^\infty}\leqq \sqrt{2}\|f\|_{L^2}^\frac12\|f'\|_{L^2}^\frac12, \quad\forall\ \  f\in H^1(\mathbb{R}).
\end{align}
We overcome this by employing the boundedness of the Riesz operator $(-\Delta)^{-1}\partial_i\partial_j$ with $i,j=1,2,3$ in $L_\text{h}^p(\R_{x_1}\times\R_{x_2};L_\text{v}^q(\R_{x_3}))$ spaces, where $1<p,q<\infty$, and combining the horizontal Gagliardo--Nirenberg inequality
\begin{align}\label{GNS2d}
	\|f\|_{L^4_\mathrm{h}(\mathbb{R}^2)} \lesssim \|f\|_{L^2_\mathrm{h}(\mathbb{R}^2)}^{1/2}\|\nabla_\mathrm{h} f\|_{L^2_\mathrm{h}(\mathbb{R}^2)}^{1/2},
\end{align}
which allows us to trade integrability for derivatives in the horizontal variables.

We emphasize that the term $A_6=\big((-\Delta)^{-1}\partial_3^2(b^3)^2 | \mathrm{div}_\mathrm{h} u^\mathrm{h}\big)_{L^2}$ must also be addressed with the same method, making our proposed dissipation pattern and background magnetic field essential.

\subsubsection{Stage 3: Nonlinear Term Analysis via Iterative Structure Exploitation}

The most severe analytical challenges reside in controlling high-order nonlinear terms within the estimate for $\|b^\mathrm{h}\|_{H^2}$. Specifically, the following terms create essential difficulties in the derivation of appropriate a priori bounds
\begin{align*}
	\Big(\partial_i^2(u\cdot\nabla b^\h)|\partial_i^2 b^\h\Big)_{L^2}=&\Big(2\partial_{3}u^{3} \partial_{3}^2b^\h|\partial_{3}^2b^\h\Big)_{L^2}+ \text{good terms}\\
	=& \text{ good terms}-2\Big(\partial_{2}u^{2} \partial_{3}^2b^\h|\partial_{3}^2b^\h\Big)_{L^2},\\
	\Big(\partial_i^2(b\cdot\nabla u^\h)|\partial_i^2 b^\h\Big)_{L^2}=&\Big(\partial_{i}^2b^{\mathrm{h}}\cdot\nabla_{\mathrm{h}} u^\h|\partial_{i}^2b^\h\Big)_{L^2}+\text{good terms}\\
	=& \text{ good terms}+\Big(\partial_{3}^2b^2\partial_2u^\h|\partial_{3}^2b^\h\Big)_{L^2}.
\end{align*}
In order to handle these two terms simultaneously, we estimate the following more general form:
\begin{align}\label{W1}
	W_1 \eqdefa \int_{\R^3} \partial_2 u^i  \partial_3^2 b^j \partial_3^2 b^k \dx \quad \mathrm{with}\quad i,j,k=1,2.
\end{align}
It is easy to observe that \eqref{W1} cannot be controlled directly via anisotropic Sobolev inequalities because $b$ possesses only $\partial_1^2$ dissipation, making $\partial_3^2 b^\mathrm{h}$ a ``dangerous'' factor. Our breakthrough comes from a iterative procedure: using the magnetic equation $ \partial_t b + u\cdot\nabla b - \partial_1^2 b = b\cdot\nabla u + \partial_2 u$, we rewrite $\partial_2 u$ as
\begin{align}\label{ha}
	\partial_2 u = \partial_t b + u\cdot\nabla b - \partial_1^2 b - b\cdot\nabla u.
\end{align}
Substituting this into \eqref{W1} transfers one time derivative or one $\partial_1^2$ dissipation onto the critical factors
\begin{align}\label{W11}
	W_1&=\int_{\R^3}({\partial_t}b^i+u\cdot\nabla{b}^i-\partial_1^2b^i -b\cdot\nabla{u}^i)\partial_3^2b^j\partial_3^2b^k\dx \nonumber\\
	&=\int_{\R^3}\frac{\mathrm{d}}{\dt}(b^i\partial_3^2b^j\partial_3^2b^k)-b^i\partial_t\partial_3^2b^j\partial_3^2b^k-b^i\partial_t\partial_3^2b^k\partial_3^2b^j\dx +\text{other terms},
\end{align}
The time derivative $\partial_t\partial_3^2 b^j$ (resp. $\partial_t\partial_3^2 b^k$) is then replaced using the equation again, creating an iterative structure. Combining \eqref{ha} and \eqref{W11}, we observe that
\begin{align}\label{W111}
	W_1&=\text{other terms}+\int_{\R^3} b^i\partial_3^2b^k\partial_3^2(u\cdot\nabla{b}^j-\partial_1^2b^j -b\cdot\nabla{u}^j-\partial_2u^j)\dx\nonumber\\
	&\quad+\int_{\R^3} b^i\partial_3^2b^j\partial_3^2(u\cdot\nabla{b}^k-\partial_1^2b^k -b\cdot\nabla{u}^k-\partial_2u^k)\dx\nonumber\\
	&=\text{good terms}-\int_{\R^3} b^i\partial_3^2b^k b\cdot\nabla\partial_3^2u^j+b^i\partial_3^2b^jb\cdot\nabla\partial_3^2u^k\dx.
\end{align}
the ``good terms'' that emerge after iteration contain the quartic products factors like $\partial_3^2 b^\mathrm{h} \partial_3^2 b^\mathrm{h}$. We introduce a key interpolation lemma
\begin{lem}\label{tool}
	The following estimates hold when the right-hand terms are all bounded. If $i,j,k,l=1,2,3$ and $(m,n)=(1,2)$ or $(2,1)$, we have
	\begin{align} \label{yi}
		\|f\cdot g\|_{L^2}&\lesssim\|\partial_1 f\|_{H^1}^\frac12\|\partial_2 g\|_{L^2}^\frac12\|f\|_{H^1}^\frac12\|g\|_{L^2}^\frac12,\\
		\label{er}\int_{\R^3}\partial_i\partial_jf\cdot\partial_kg\cdot\partial_lh\dx&\lesssim\|\partial_1f\|_{H^2}^\frac12\|\partial_if\|_{H^1}^\frac12\|\partial_kg\|_{H^1}\|\partial_lh\|_{H^1}^\frac12\|h\|_{H^1}^\frac12,\\\label{san}
		\int_{\R^3}f\cdot g\cdot\partial_i\partial_jh\cdot\partial_k\partial_lv\dx&\lesssim\|\partial_2f\|_{H^1}^\frac12\|f\|_{H^1}^\frac12\|\partial_mg\|_{H^1}^\frac12\|g\|_{H^1}^\frac12\nonumber\\
		&\quad \times\|\partial_1h\|_{H^2}^\frac12\|h\|_{H^2}^\frac12\|\partial_nv\|_{H^2}^\frac12\|v\|_{H^2}^\frac12.
	\end{align}
\end{lem}
\begin{proof}
	By using H\"older's inequality and \eqref{L^infty_1D}, we find that
	\begin{align*}
		\|f\cdot g\|_{L^2}&\lesssim\|f\|_{L^\infty_{x_1}L^2_{x_2}L^\infty_{x_3}}\| g\|_{L^2_{x_1}L^\infty_{x_2}L^2_{x_3}}\lesssim\|\partial_1 f\|_{H^1}^\frac12\|\partial_2 g\|_{L^2}^\frac12\|f\|_{H^1}^\frac12\|g\|_{L^2}^\frac12,
	\end{align*}
	and
	\begin{align*}
		\int_{\R^3}\partial_i\partial_jf\cdot\partial_kg\cdot\partial_lh\dx&\lesssim\|\partial_i\partial_jf\|_{L^\infty_{x_1}L^2_{x_2}L^2_{x_3}}\|\partial_kg\|_{L^2_{x_1}L^\infty_{x_2}L^2_{x_3}}\|\partial_lh\|_{L^2_{x_1}L^2_{x_2}L^\infty_{x_3}}\\
		&\lesssim\|\partial_1f\|_{H^2}^\frac12\|\partial_if\|_{H^1}^\frac12\|\partial_kg\|_{H^1}\|\partial_lh\|_{H^1}^\frac12\|h\|_{H^1}^\frac12.
	\end{align*}
	We complete the proof of \eqref{yi} and \eqref{er}. The proof of \eqref{san} follows the same method, and we omit the details here.
\end{proof}
It provides the precise anisotropic estimates needed to distribute derivatives optimally. For example, \eqref{san} allows us to place the ``missing'' derivative $\partial_1$ (the only one with dissipation for $b$) on the most dangerous factors. After the above steps, we encounter the term
\begin{align*}
	W_2\eqdefa -\int_{\R^3} b^i\partial_3^2b^k b\cdot\nabla\partial_3^2u^j+b^i\partial_3^2b^jb\cdot\nabla\partial_3^2u^k\dx,
\end{align*}
which contains the derivative of $u^j$ (resp. $u^k$) involving a third order. By systematically applying derivative rules (see \eqref{S21deco} for details), this gives
\begin{align*}
	W_2=\text{good terms}+\int_{\R^3}b^i\partial_3^2u^j\partial_3^2(b\cdot\nabla b^k)+b^i\partial_3^2u^k\partial_3^2(b\cdot\nabla b^j)\dx.
\end{align*}
And then, we recall
\begin{align}\label{bbbb}
	b\cdot\nabla b=\partial_tu+u\cdot\nabla u-\Delta_\h u+\nabla p-\partial_2b.
\end{align}
Combining the above equation and iterating the argument, we obtain
\begin{align}\label{W2}
	W_2=&\text{good terms}+\int_{\mathbb{R}^3} b^i \partial_3^2 u^j u\cdot\nabla\partial_3^2 u^k  +  b^i \partial_3^2 u^k u\cdot\nabla\partial_3^2 u^j \dx.
\end{align}
Using integration by parts, the second term of \eqref{W2} combines to
\begin{align*}
	-\int_{\mathbb{R}^3} u\cdot\nabla b^i \partial_3^2 u^j \partial_3^2 u^k \dx,
\end{align*}
which is of lower differential order and can be controlled using available bounds. This final cancellation exploits the symmetric structure of the coupled system.

\subsubsection{Innovative Elements}

Our approach introduces several technical elements:

\begin{enumerate}
	\item \textbf{Component-decoupled energy method:} A systematic framework for handling different smallness conditions on different components.
	
	\item \textbf{Pressure term duality:} The recognition that pressure terms in horizontal and vertical balances are dual expressions of the same quantity, requiring compatible but distinct estimates.
	
	\item \textbf{Iterative nonlinear analysis:} A multi-step procedure that repeatedly uses the equations to transfer derivatives, ultimately revealing hidden cancellations.
	
	\item \textbf{Anisotropic interpolation lemma:} Lemma \ref{tool} provides a suite of inequalities specifically tailored to our dissipation pattern, enabling precise derivative allocation.
	
	\item \textbf{Background field as a structural tool:} The linear terms $\partial_2 b$ and $\partial_2 u$ not only provide weak dissipation but also enable the key substitution $\partial_2 u = \partial_t b + \cdots$ that initiates the iterative control of dangerous nonlinearities.
\end{enumerate}

This comprehensive strategy allows us to close the estimates under condition \eqref{small}, establishing global well-posedness with large vertical data. This is a result previously inaccessible to existing methods.

\subsection{Comparison with Previous Works and Originality}

The improvement of our result becomes particularly clear when placed in context with existing literature. Unlike previous studies on anisotropic MHD systems with background magnetic fields, our work achieves a significant relaxation of the smallness requirement while maintaining minimal dissipation assumptions. This advancement not only solves a long-standing open problem in the specific context of \eqref{eq.MHD}, but also provides a blueprint for analyzing component-wise smallness conditions in other coupled PDE systems with anisotropic dissipation.

\subsubsection{Contrast with Fully Dissipative MHD}
The fully dissipative case ($\mu_i,\nu_i>0$ for all $i=1,2,3$) allows for standard energy estimates where all components enjoy dissipation. In such settings, global well-posedness for small data follows from relatively straightforward modifications of Navier-Stokes techniques \cite{72lions,83sermange}. More recent breakthroughs, such as \cite{2023Nacer}, established global existence for large data with small unidirectional derivatives, but still relied on full dissipation.
\subsubsection{Contrast with Partially Dissipative MHD}
For partially dissipative systems (only velocity dissipation or only magnetic diffusion), notable progress was made by Lin et al. \cite{15Lin} and later by several other authors, but all under smallness of the full initial data. Further advances (including 2D settings) have been achieved for MHD systems with only viscosity or magnetic diffusion (see, e.g., \cite{2017abidi,2018cai,18he,19jiang,15Lin2,18pan,14Ren,18tan,17wei,15wu,15Xu,18zhou}).


\subsubsection{Contrast with Mixed Dissipative MHD Requiring Full Smallness}
For MHD systems with mixed partial dissipation, previous results all required the full initial data $(u_0,b_0)$ to be small. For instance:
\begin{itemize}
	\item With dissipation $(\Delta_\mathrm{h}u,\partial_3^2b)$ and background field $B^*=(1,0,0)$, \cite{24han,21wu} proved global well-posedness assuming $\|(u_0,b_0)\|_{H^s(\R^3)}$ is small ($s=2$ or $3$).
	\item With dissipation $(\partial_1^2u,\Delta_\mathrm{h}b)$ and $B^*=(0,1,0)$, \cite{23Lin,25Lin,26wu} established similar results under full smallness conditions.
\end{itemize}
All these works share a common limitation: the vertical components $u^3_0$ and $b^3_0$ must be as small as the horizontal ones. By making energy estimates far more delicate and necessitating the novel component-decoupled approach described above, our Theorem \ref{main.Th} breaks this pattern by allowing $\|(u^3_0,b^3_0)\|_{H^2(\R^3)}$ to be arbitrarily large. This represents a major advancement toward simulating physically realistic regimes in which vertical motions (e.g., convection cells, plumes) can reach substantial amplitudes.

\subsubsection{Why Our Configuration is Critical}
The specific choice of dissipation $(\Delta_\mathrm{h}u,\partial_1^2b)$ with background field $B^*=(0,1,0)$ is not arbitrary; it represents a \emph{critical} configuration that permits the horizontal-only smallness condition. To understand this, consider the alternative pattern $(\Delta_\mathrm{h}u,\partial_3^2b)$ with $B^*=(1,0,0)$ studied in \cite{21wu}. There, the dangerous nonlinear terms in the $b^\mathrm{h}$ estimate take the form
\begin{align*}
	\int_{\R^3}\partial_1u^1\partial_2^2b^1\partial_2^2b^1\dx,\quad \int_{\R^3}\partial_1u^1\partial_2^2b^3\partial_2^2b^3\dx, \quad\mathrm{and}\quad\int_{\R^3}\partial_1u^3\partial_2^2b^1\partial_2^2b^3\dx,
\end{align*}
which involve $\partial_2^2 b^3$. Since $b^3$ lacks dissipation in this configuration, controlling these terms forces $b^3$ to be small. In contrast, for our system \eqref{eq.MHD}, the analogous problematic terms become
\begin{align*}
	\int_{\R^3}\partial_{2}u^{2} \partial_{3}^2b^\h\partial_{3}^2b^\h\dx,\quad\mathrm{and}\quad\int_{\R^3}\partial_{3}^2b^2\partial_2u^\h\partial_{3}^2b^\h \dx,
\end{align*}
which involve only $\partial_3^2 b^\mathrm{h}$. Importantly, $b^\mathrm{h}$ possesses $\partial_1^2$ dissipation, and the factors $\partial_3^2 b^\mathrm{h}$ can be handled through our iterative procedure that leverages the background field structure. This structural difference explains why we can impose smallness only on the horizontal components, whereas other dissipation patterns inherently require full smallness.

\subsubsection{Relation to Navier-Stokes with Anisotropic Dissipation}
When $b\equiv 0$, system \eqref{eq.MHD} reduces to the Navier-Stokes equations with only horizontal dissipation. There are substantial results on the well-posedness (see,  e.g.,\cite{00chemin,02if,20Liu,05paicu,11paicu,10zhang}). It is worth noting that Chemin et al. \cite{07chemin} proved global well-posedness under a horizontal smallness condition. However, the techniques developed for Navier-Stokes fail dramatically in the MHD context due to the strong coupling between $u$ and $b$. The magnetic field introduces additional nonlinearities and wave-type interactions that disrupt the energy-coupling mechanisms available in the fluid-only case. Our component-decoupled method and iterative nonlinear analysis are specifically designed to overcome these MHD-specific challenges.

\subsection*{Outline}
{Section \ref{2}} completes the proof of Theorem~\ref{main.Th} by synthesizing the horizontal and vertical estimates.
{Section \ref{3}} derives the global energy bound for $(u^3,b^3)$.
{Section \ref{4}} establishes the global estimate for $(u^\mathrm{h},b^\mathrm{h})$.
{Appendix A} provides a detailed analysis of the most subtle nonlinear term ($J_{325}$), illustrating the iterative use of the equations and the cancellations that close the estimates.

\noindent\textbf{Notation.} $H^s=H^s(\mathbb{R}^3)$, $L^2=L^2(\mathbb{R}^3)$. We write $a\lesssim b$ if $a\leqq Cb$ for a universal constant $C$. The $L^2$ inner product is denoted $(\cdot\mid\cdot)_{L^2}$, and $\|(f,g)\|=\|f\|+\|g\|$. We set $\nabla_\mathrm{h}=(\partial_1,\partial_2)$, $\mathrm{div}_\mathrm{h}=\partial_1+\partial_2$, $\Delta_\mathrm{h}=\partial_1^2+\partial_2^2$.

\section{Proof of the Main Theorem}\label{2}

We begin by establishing the two core a priori estimates for the horizontal and vertical components, which form the foundation of the proof.

\begin{prop}\label{v}
	Under the assumptions of Theorem \ref{main.Th}, the vertical components $(u^3,b^3)$ satisfy
	\begin{align}\label{u3b3H2}
		\frac12 \frac{\mathrm{d}}{\dt}\|(u^3,b^3)\|^2_{H^2}+\|(\nabla_{\mathrm{h}} u^3,\partial_1b^3)\|^2_{H^2}
		\leqq &\frac{9}{100} \|(\nabla_{\mathrm{h}} u^3,\partial_1b^3)\|_{H^2}^2+\frac{1}{25}\|(\partial_2b^\mathrm{h},\partial_2b^3)\|_{H^1}^2 \nonumber\\
		&+ \|(\nabla_{\mathrm{h}} u^\mathrm{h},\partial_1b^\mathrm{h})\|_{H^2}^2\|(u^\mathrm{h},b^\mathrm{h})\|_{H^2}^2\nonumber\\
		&+C\Bigl(\|\nabla_{\mathrm{h}} u^3\|_{H^2}^2+\|(\partial_2b^\mathrm{h},\partial_2b^3)\|_{H^1}^2\Bigr)\|(u^\mathrm{h},b^\mathrm{h})\|_{H^2}^2\nonumber\\
		&+C\|(\nabla_{\mathrm{h}} u^\mathrm{h},\partial_1b^\mathrm{h})\|_{H^{2}}^2\|(u^3,b^3)\|_{H^2}^2,
	\end{align}
	and the auxiliary estimate
	\begin{align}\label{p2bH1}
		\|(\partial_2b^\mathrm{h},\partial_2b^3)\|_{H^1}^2\leqq &\frac{\mathrm{d}}{\mathrm{d}t}\int_{\R^3}\bigl(u^\mathrm{h}\partial_2 b^\mathrm{h}+u^3\partial_2 b^3+\partial_{i}u^\mathrm{h}\partial_{i}\partial_{2}b^\mathrm{h}+\partial_{i}u^3\partial_{i}\partial_{2}b^3\bigr)\dx\nonumber\\
		&+\frac{11}{20}\|(\partial_{2}b^\mathrm{h},\partial_{2}b^3)\|_{H^1}^2+\frac{41}{10}\|\nabla_{\mathrm{h}} u^3\|_{H^2}^2+\|\partial_1b^3\|_{H^2}\nonumber\\
		&+C\|(\nabla_{\mathrm{h}} u^\mathrm{h},\partial_{1}b^\mathrm{h})\|_{H^{2}}^2\bigl(1+\|( u^\mathrm{h}, b^\mathrm{h})\|_{H^2}^2\bigr)\nonumber\\
		&+C\Bigl(\|(\nabla_{\mathrm{h}} u^3,\partial_1b^3)\|_{H^2}^2+\|(\partial_2b^\mathrm{h},\partial_2b^3)\|_{H^{1}}^2\Bigr)\|(u^\mathrm{h},b^\mathrm{h})\|_{H^{2}}^2\nonumber\\
		&+C\|(\nabla_{\mathrm{h}} u^\mathrm{h},\partial_1b^\mathrm{h})\|_{H^{2}}^2\|(u^3,b^3)\|_{H^2}^2,
	\end{align}
	where the summation over $i=1,2,3$ is understood.
\end{prop}

\begin{prop}\label{h}
	Under the assumptions of Theorem \ref{main.Th}, the horizontal components $(u^\mathrm{h},b^\mathrm{h})$ satisfy
	\begin{align}\label{uhbhH2}
		\frac12 \frac{\mathrm{d}}{\dt}&\|(u^\mathrm{h},b^\mathrm{h})\|^2_{H^2}+\|(\nabla_{\mathrm{h}} u^\mathrm{h},\partial_1b^\mathrm{h})\|^2_{H^2}\nonumber\\
		\leqq &\frac{\mathrm{d}}{\dt}\int_{\R^3} \bigl(b^i\partial_3^2b^j\partial_3^2b^k-b^i\partial_3^2u^j\partial_3^2u^k\bigr)\dx+\frac{1}{50}\|\partial_2b^\mathrm{h}\|_{H^1}^2\nonumber\\
		&+\frac{11}{100}\|(\nabla_{\mathrm{h}} u^\mathrm{h},\partial_1b^\mathrm{h})\|_{H^2}^2+C\|(\nabla_{\mathrm{h}} u^\mathrm{h},\partial_1b^\mathrm{h})\|_{H^2}^2\|(u^\mathrm{h},b^\mathrm{h})\|_{H^{2}}^2\nonumber\\
		&+C\Bigl(\|(\nabla_{\mathrm{h}} u^3,\partial_1b^3)\|_{H^2}^2+\|(\partial_2b^\mathrm{h},\partial_2b^3)\|_{H^1}^2\Bigr)\|(u^\mathrm{h},b^\mathrm{h})\|_{H^2}^2\nonumber\\
		&+C\Bigl(\|(\nabla_{\mathrm{h}} u^\mathrm{h},\nabla_{\mathrm{h}} u^3)\|_{H^2}^2\|u^3\|_{H^2}^2+ \|(\partial_2b^\mathrm{h},\partial_2b^3)\|_{H^1}^2\|(b^\mathrm{h},b^3)\|_{H^2}^2\Bigr)\|u^\mathrm{h}\|_{H^2}^2\nonumber\\
		&+C\Bigl(\|\partial_2b^3\|_{H^1}^2\|b^3\|_{H^2}^2+\|(\nabla_{\mathrm{h}} u^\mathrm{h},\nabla_{\mathrm{h}} u^3)\|_{H^2}^2\|(u^\mathrm{h},u^3)\|_{H^2}^2\Bigr)\|b^\mathrm{h}\|_{H^2}^2,
	\end{align}
	where in \eqref{uhbhH2} the indices $i,j,k$ belong to $\{1,2\}$. And the auxiliary estimate
	\begin{align}\label{p2bhH1}
		\|\partial_2 b^\mathrm{h}\|_{H^1}^2\leqq&\frac{\mathrm{d}}{\mathrm{d}t}\int_{\R^3}\bigl(\partial_{i}u^\mathrm{h}\partial_{i}\partial_{2}b^\mathrm{h}+ u^\mathrm{h}\partial_2 b^\mathrm{h}\bigr)\dx+\frac{203}{50}\|\nabla_{\mathrm{h}} u^\mathrm{h}\|_{H^2}^2+\frac{53}{50}\|\partial_1 b^\mathrm{h}\|_{H^2}^2+\frac{7}{25}\|\partial_2 b^\mathrm{h}\|_{H^1}^2\nonumber\\
		&+C\Bigl(\|(\nabla_{\mathrm{h}} u^\mathrm{h},\partial_1b^\mathrm{h})\|_{H^2}^2+\|(\partial_2b^\mathrm{h},\partial_2b^3)\|_{H^1}^2+\|\nabla_{\mathrm{h}}u^3\|_{H^{2}}^2\Bigr)\|(u^\mathrm{h},b^\mathrm{h})\|_{H^{2}}^2\nonumber\\
		&+C\Big(\|(\nabla_{\mathrm{h}} u^\mathrm{h},\nabla_{\mathrm{h}} u^3)\|_{H^2}^2\|( u^\mathrm{h}, u^3)\|_{H^2}^2+ \|(\partial_2b^\mathrm{h},\partial_2b^3)\|_{H^1}^2\|(b^\mathrm{h},b^3)\|_{H^2}^2\nonumber\\
		&+\|(\partial_1b^\mathrm{h},\partial_1b^3)\|_{H^2}^2\|(b^\mathrm{h},b^3)\|_{H^2}^2\Big)\|u^\mathrm{h}\|_{H^2}^2\nonumber\\
		&+C\Bigl(\|(\nabla_{\mathrm{h}} u^\mathrm{h},\nabla_{\mathrm{h}} u^3)\|_{H^2}^2\|( u^\mathrm{h}, u^3)\|_{H^2}^2+ \|(\partial_2b^\mathrm{h},\partial_2b^3)\|_{H^1}^2\|(b^\mathrm{h},b^3)\|_{H^2}^2\Bigr)\|b^\mathrm{h}\|_{H^2}^2,
	\end{align}
	where the summation over $i=1,2,3$ is understood.
\end{prop}

The detailed proofs of Propositions \ref{v} and \ref{h} are presented in Sections \ref{3} and \ref{4}, respectively. We now show how these estimates imply Theorem \ref{main.Th}.

\begin{proof}[Proof of Theorem \ref{main.Th}]
	Since the local well-posedness of system \eqref{eq.MHD} in the space $H^2(\R^3)$ follows from standard arguments  (see, e.g., \cite{02majda}), we focus primarily on establishing global bounds for $(u,b)$. Our strategy  is to establish the global estimates of horizontal component $(u^\mathrm{h},b^\mathrm{h})$ and vertical component $(u^3,b^3)$ in $H^2(\R^3)$, respectively.
	
	Let $(u,b)$ be a sufficiently smooth solution of \eqref{eq.MHD} on the maximal time interval $[0,T^*)$. We define the exit time
	\begin{align}\label{eq:exit-time}
		T^{**}\eqdefa \sup\Bigl\{t<T^*:\|(u^\mathrm{h},b^\mathrm{h})\|_{L_{t}^{\infty}(H^2)}^2+\|(\nabla_\mathrm{h}u^\mathrm{h},\partial_1 b^\mathrm{h})\|_{L_{t}^{2}(H^2)}^2\leqq \eta \Bigr\},
	\end{align}
	for a constant $\eta>0$ to be chosen later. Our goal is to prove that $T^{**}=T^*$ for an appropriate $\eta$. Otherwise, if $T^{**}<T^*$, we combine the estimates \eqref{u3b3H2} and \eqref{p2bH1} to obtain a closed inequality for the vertical part. Multiplying \eqref{p2bH1} by $1/5$ and adding it to \eqref{u3b3H2} yields, for any $t\in[0,T^{**}]$,
	\begin{align*}
		\frac12 \frac{\mathrm{d}}{\dt}&\|(u^3,b^3)\|^2_{H^2}+\|(\nabla_{\mathrm{h}} u^3,\partial_1b^3)\|^2_{H^2}+\frac{1}{5}\|(\partial_2b^\mathrm{h},\partial_2b^3)\|_{H^1}^2\\
		\leqq& \|(\nabla_{\mathrm{h}} u^3,\partial_1b^3)\|_{H^2}^2\Bigl(\frac{41}{50}+C_1\eta\Bigr)+\|(\partial_2b^\mathrm{h},\partial_2b^3)\|_{H^1}^2\Bigl(\frac{3}{20}+C_1\eta\Bigr)\\
		&+\frac{1}{5}\frac{\mathrm{d}}{\mathrm{d}t}\int_{\R^3}\bigl(u^\mathrm{h}\partial_2 b^\mathrm{h}+u^3\partial_2 b^3+\partial_{i}u^\mathrm{h}\partial_{i}\partial_{2}b^\mathrm{h}+\partial_{i}u^3\partial_{i}\partial_{2}b^3\bigr)\dx\\
		&+C_1\|(\nabla_{\mathrm{h}} u^\mathrm{h},\partial_1b^\mathrm{h})\|_{H^{2}}^2\|(u^3,b^3)\|_{H^2}^2+C_1\|(\nabla_{\mathrm{h}} u^\mathrm{h},\partial_{1}b^\mathrm{h})\|_{H^{2}}^2\bigl(1+\|( u^\mathrm{h}, b^\mathrm{h})\|_{H^2}^2\bigr).
	\end{align*}
	Now choose $\eta$ sufficiently small, such that $\eta\leqq 1/(50C_1)$. Then the coefficients of the dissipative terms become strictly smaller than one, and we obtain
	\begin{align}\label{eq:vert-ineq}
		\frac12 \frac{\mathrm{d}}{\dt}&\|(u^3,b^3)\|^2_{H^2}+\frac{4}{25}\|(\nabla_{\mathrm{h}} u^3,\partial_1b^3)\|_{H^2}^2+\frac{3}{100}\|(\partial_2b^\mathrm{h},\partial_2b^3)\|_{H^1}^2\nonumber\\
		\leqq&C_1\|(\nabla_{\mathrm{h}} u^\mathrm{h},\partial_1b^\mathrm{h})\|_{H^{2}}^2\|(u^3,b^3)\|_{H^2}^2+C_1\|(\nabla_{\mathrm{h}} u^\mathrm{h},\partial_{1}b^\mathrm{h})\|_{H^{2}}^2\bigl(1+\|( u^\mathrm{h}, b^\mathrm{h})\|_{H^2}^2\bigr) \nonumber\\
		&+\frac{1}{5}\frac{\mathrm{d}}{\mathrm{d}t}\int_{\R^3}u^\mathrm{h}\partial_2 b^\mathrm{h}+u^3\partial_2 b^3+|\nabla u^\mathrm{h}\nabla\partial_{2}b^\mathrm{h}|+|\nabla u^3\nabla\partial_{2}b^3|\dx.
	\end{align}
	Integrating \eqref{eq:vert-ineq} over $[0,t]$ gives
	\begin{align}
		\frac12 \|(u^3&(t),b^3(t))\|^2_{H^2}+\int_{0}^{t}\frac{4}{25}\|(\nabla_{\mathrm{h}} u^3,\partial_1b^3)\|_{H^2}^2\dt'+\int_{0}^{t}\frac{3}{100}\|(\partial_2b^\mathrm{h},\partial_2b^3)\|_{H^1}^2\dt'\nonumber\\
		\leqq&C_1\int_{0}^{t}\|(\nabla_{\mathrm{h}} u^\mathrm{h},\partial_1b^\mathrm{h})\|_{H^{2}}^2\|(u^3,b^3)\|_{H^2}^2+\|(\nabla_{\mathrm{h}} u^\mathrm{h},\partial_{1}b^\mathrm{h})\|_{H^{2}}^2\bigl(1+\|( u^\mathrm{h}, b^\mathrm{h})\|_{H^2}^2\bigr)\dt'\nonumber\\
		&+\frac12\|(u^3_0,b^3_0)\|^2_{H^2}+\frac{1}{5}(Y(t)-Y(0)),
	\end{align}
	where
	\begin{align*}
		Y(\cdot)\eqdefa& \int_{\R^3}u^\mathrm{h}(\cdot)\partial_2 b^\mathrm{h}(\cdot)+u^3(\cdot)\partial_2 b^3(\cdot)+|\nabla u^\mathrm{h}(\cdot)\nabla\partial_{2}b^\mathrm{h}(\cdot)|+|\nabla u^3(\cdot)\nabla\partial_{2}b^3(\cdot)|\dx.
	\end{align*}
	Using the bound $Y(\tau) \leqq \|(u(\tau),b(\tau))\|_{H^2}^2$ for any $\tau \in [0,t]$ (which follows from Hölder's inequality), we have in particular
	\begin{align*}
		Y(t) - Y(0) \leqq\|(u(t),b(t))\|_{H^2}^2 + \|(u_0,b_0)\|_{H^2}^2,
	\end{align*}
	which implies
	\begin{align*}
		\frac{3}{10}\|(u^3(t)&,b^3(t))\|^2_{H^2}+\frac{4}{25}\int_{0}^{t}\|(\nabla_{\mathrm{h}} u^3,\partial_1b^3)\|_{H^2}^2\dt'+\frac{3}{100}\int_{0}^{t}\|(\partial_2b^\mathrm{h},\partial_2b^3)\|_{H^1}^2\dt'\\
		\leqq&C_1\int_{0}^{t}\|(\nabla_{\mathrm{h}} u^\mathrm{h},\partial_1b^\mathrm{h})\|_{H^2}^2\bigl(1+\|(u^\mathrm{h},b^\mathrm{h})\|_{H^2}^2\bigr)+\|(\nabla_{\mathrm{h}} u^\mathrm{h},\partial_1b^\mathrm{h})\|_{H^2}^2\|(u^3,b^3)\|_{H^2}^2\dt'\\
		&+\frac{7}{10}\|(u^3_0,b^3_0)\|^2_{H^2}+\frac15\|(u^\mathrm{h}_0,b^\mathrm{h}_0)\|_{H^{2}}^2+\frac15\|(u^\mathrm{h}(t),b^\mathrm{h}(t))\|_{H^{2}}^2.
	\end{align*}
	Applying Gronwall's inequality and using the bound \eqref{eq:exit-time} for the horizontal quantities on $[0,T^{**}]$, we obtain
	\begin{align}\label{eq:vert-final-bound}
		\|(u^3&,b^3)\|^2_{L_{t}^{\infty}(H^2)}+\|(\nabla_{\mathrm{h}} u^3,\partial_1b^3)\|^2_{L_{t}^{2}(H^2)}+\|(\partial_2b^\mathrm{h},\partial_2b^3)\|_{L_{t}^{2}(H^1)}^2\nonumber\\
		\leqq& C_1'\Bigl(\|(u^3_0,b^3_0)\|^2_{H^2}+\|(u^\mathrm{h}_0,b^\mathrm{h}_0)\|_{H^{2}}^2+\eta\Bigr)e^{C_1'\eta}\nonumber\\
		\leqq& C_1''\|(u^3_0,b^3_0)\|^2_{H^2},
	\end{align}
	provided $\eta$ is chosen small enough (independent of the initial data) and the smallness condition \eqref{small} is satisfied.
	
	Now we turn to controlling the horizontal components. Multiplying \eqref{p2bhH1} by $1/5$ and adding it to \eqref{uhbhH2} gives
	\begin{align*}
		\frac12 \frac{\mathrm{d}}{\dt}&\|(u^\mathrm{h},b^\mathrm{h})\|^2_{H^2}+\|(\nabla_{\mathrm{h}} u^\mathrm{h},\partial_1b^\mathrm{h})\|^2_{H^2}+\frac{1}{5}\|\partial_2 b^\mathrm{h}\|_{H^1}^2\\
		\leqq & \|(\nabla_{\mathrm{h}} u^\mathrm{h},\partial_1b^\mathrm{h})\|_{H^2}^2\Bigl(\frac{461}{500}+C_2\eta\Bigr)+\|\partial_2b^\mathrm{h}\|_{H^1}^2\Bigl(\frac{19}{250}+C_2\eta\Bigr)+E^{(u^\mathrm{h},b^\mathrm{h})}+E^{u^\mathrm{h}}+E^{b^\mathrm{h}}\\
		&+\frac15\frac{\mathrm{d}}{\dt}\int_{\R^3} b^\mathrm{h} |\partial_3^2b^\mathrm{h} |^2-b^\mathrm{h} |\partial_3^2u^\mathrm{h} |^2+|\nabla u^\mathrm{h} \nabla\partial_{2}b^\mathrm{h} |+ u^\mathrm{h} \partial_2 b^\mathrm{h} \dx,
	\end{align*}
	where
	\begin{align*}
		E^{(u^\mathrm{h},b^\mathrm{h})}&\eqdefa C_2\Bigl(\|(\nabla_{\mathrm{h}} u^3,\partial_1b^3)\|_{H^2}^2+\|(\partial_2b^\mathrm{h},\partial_2b^3)\|_{H^1}^2\Bigr)\|(u^\mathrm{h},b^\mathrm{h})\|_{H^2}^2,\\
		E^{u^\mathrm{h}}&\eqdefa C_2\Bigl(\|(\nabla_{\mathrm{h}} u^\mathrm{h},\nabla_{\mathrm{h}} u^3)\|_{H^2}^2\|( u^\mathrm{h}, u^3)\|_{H^2}^2+ \|(\partial_2b^\mathrm{h},\partial_2b^3)\|_{H^1}^2\|(b^\mathrm{h},b^3)\|_{H^2}^2\\
		&\qquad\qquad+\|(\partial_1b^\mathrm{h},\partial_1b^3)\|_{H^2}^2\|(b^\mathrm{h},b^3)\|_{H^2}^2\Bigr)\|u^\mathrm{h}\|_{H^2}^2,\\
		E^{b^\mathrm{h}}&\eqdefa C_2\Bigl(\|(\nabla_{\mathrm{h}} u^\mathrm{h},\nabla_{\mathrm{h}} u^3)\|_{H^2}^2\|( u^\mathrm{h}, u^3)\|_{H^2}^2+ \|(\partial_2b^\mathrm{h},\partial_2b^3)\|_{H^1}^2\|(b^\mathrm{h},b^3)\|_{H^2}^2\Bigr)\|b^\mathrm{h}\|_{H^2}^2 .
	\end{align*}
	For $\eta > 0$ in \eqref{eq:exit-time} being small enough again that $\eta\leqq1/(50C_2)$, then absorbing  the leading dissipative terms and integrating over $[0,t]$ gives
	\begin{align}\label{eq:horiz-ineq}
		\frac12 \|(u^\mathrm{h}&(t),b^\mathrm{h}(t))\|^2_{H^2}+\frac{2}{5}\int_{0}^{t}\|(\nabla_{\mathrm{h}} u^\mathrm{h},\partial_1b^\mathrm{h})\|_{H^2}^2\dt'+\frac{1}{10}\int_{0}^{t}\|\partial_2 b^\mathrm{h}\|_{H^1}^2\dt'\nonumber\\
		\leqq & \frac12 \|(u^\mathrm{h}_0,b^\mathrm{h}_0)\|^2_{H^2}+\frac15(Z(t)-Z(0))+\int_{0}^{t}E^{(u^\mathrm{h},b^\mathrm{h})}+E^{u^\mathrm{h}}+E^{b^\mathrm{h}}\dt',
	\end{align}
	where
	\begin{align*}
		Z(\cdot)\eqdefa& \int_{\R^3} b^\mathrm{h} (\cdot)|\partial_3^2b^\mathrm{h} (\cdot)|^2-b^\mathrm{h} (\cdot)|\partial_3^2u^\mathrm{h} (\cdot)|^2+|\nabla u^\mathrm{h} (\cdot)\nabla\partial_{2}b^\mathrm{h} (\cdot)|+ u^\mathrm{h} (\cdot)\partial_2 b^\mathrm{h} (\cdot)\dx.
	\end{align*}
	By a similar derivation of $Y(t)$, we get
	\begin{align*}
		Z(t)\leqq \|(u^\mathrm{h}(t),b^\mathrm{h}(t))\|_{H^2}^2(1+C_3\|(u^\mathrm{h}(t),b^\mathrm{h}(t))\|_{H^2})\leqq \|(u^\mathrm{h}(t),b^\mathrm{h}(t))\|_{H^2}^2(1+C_3\eta^\frac12).
	\end{align*}
	Further reducing $\eta$ if necessary, we have $\eta\leqq (1/C_3)^2$, which gives
	\begin{align*}
		Z(t)\leqq2 \|(u^\mathrm{h}(t),b^\mathrm{h}(t))\|_{H^2}^2.
	\end{align*}
	Inserting the bound of $Z(t)$(and $Z(0)$ estimated by the same method) into \eqref{eq:horiz-ineq}, it holds
	\begin{align*}
		\frac{1}{10}\|(u^\mathrm{h}&(t),b^\mathrm{h}(t))\|^2_{H^2}+\frac{2}{5}\int_{0}^{t}\|(\nabla_{\mathrm{h}} u^\mathrm{h},\partial_1b^\mathrm{h})\|_{H^2}^2\dt'+\frac{1}{10}\int_{0}^{t}\|\partial_2 b^\mathrm{h}\|_{H^1}^2\dt'\\
		\leqq & \frac{9}{10}\|(u^\mathrm{h}_0,b^\mathrm{h}_0)\|^2_{H^2}
		+\int_{0}^{t}\bigl(E^{(u^\mathrm{h},b^\mathrm{h})}+E^{u^\mathrm{h}}+E^{b^\mathrm{h}}\bigr)\dt'.
	\end{align*}
	Then we insert the bound \eqref{eq:vert-final-bound} for the vertical components and use the definition \eqref{eq:exit-time} of $\eta$ for the horizontal ones. After applying Gronwall's inequality we arrive at
	\begin{align}\label{eq:horiz-final-bound}
		\|(u^\mathrm{h}&,b^\mathrm{h})\|^2_{L_{t}^{\infty}(H^2)}+\|(\nabla_{\mathrm{h}} u^\mathrm{h},\partial_1b^\mathrm{h})\|_{L_{t}^{2}(H^2)}^2+\|\partial_2b^\mathrm{h}\|_{L_{t}^{2}(H^1)}^2\nonumber\\
		\leqq& C_2'\|(u^\mathrm{h}_0,b^\mathrm{h}_0)\|^2_{H^2}
		\exp\Bigl(C\|(u^3_0,b^3_0)\|^2_{H^2}\bigl(1+\|(u^3_0,b^3_0)\|^2_{H^2}\bigr)\Bigr).
	\end{align}
	Recalling the smallness condition \eqref{small}, there exists an absolute $\varepsilon>0$ such that
	\begin{align*}
		\|(u^\mathrm{h}_0,b^\mathrm{h}_0)\|_{H^2}^2
		\exp\Bigl(C\|(u^3_0,b^3_0)\|_{H^2}^2\bigl(1+\|(u^3_0,b^3_0)\|_{H^2}^2\bigr)\Bigr)\leqq\varepsilon .
	\end{align*}
	Choose $\eta=2C_2'\varepsilon$ in the definition \eqref{eq:exit-time}, where $C_2'$ is the constant from \eqref{eq:horiz-final-bound}. Then \eqref{eq:horiz-final-bound} implies that for every $t\in[0,T^{**}]$,
	\begin{align*}
		\|(u^\mathrm{h},b^\mathrm{h})\|_{L_{t}^{\infty}(H^2)}^2+\|(\nabla_\mathrm{h}u^\mathrm{h},\partial_1 b^\mathrm{h})\|_{L_{t}^{2}(H^2)}^2 \leqq C_2'\varepsilon = \frac{\eta}{2} < \eta .
	\end{align*}
	This strict inequality contradicts the definition \eqref{eq:exit-time} of $T^{**}$ as the supremum of times satisfying the bound $\leqq\eta$. Hence $T^{**}=T^*$, and the solution satisfies the uniform bounds \eqref{eq:vert-final-bound} and \eqref{eq:horiz-final-bound} on the whole maximal interval $[0,T^*)$. Standard continuation arguments (see, e.g., \cite{06tao}) then show that $T^*=\infty$, i.e., the solution is global.
	
	Finally, the claimed regularity
	\begin{align*}
		(u,b)\in L^\infty(\mathbb{R}^+;H^2(\mathbb{R}^3)),\quad
		(\nabla_\mathrm{h}u,\partial_1b)\in L^2(\mathbb{R}^+;H^2(\mathbb{R}^3)),\quad
		\partial_2b\in L^2(\mathbb{R}^+;H^1(\mathbb{R}^3))
	\end{align*}
	follows directly from \eqref{eq:vert-final-bound} and \eqref{eq:horiz-final-bound}. This completes the proof of Theorem \ref{main.Th}.
\end{proof}

\section{The proof of Proposition \ref{v}}\label{3}
In this section, we shall prove the Proposition \ref{v}. It is divided into two subsections and the first one is the estimate of $\|(u^3,b^3)\|_{H^{2}}$. The subsequent subsection is the estimate of $\|\partial_2 (b^\h,b^3)\|_{H^{1}}$.
\subsection{The proof of \eqref{u3b3H2}} Due to
\begin{align}\label{Hsdef}
	\|f\|_{H^{s}}^{2}\sim\|f\|_{L^2}^2+\sum_{i=1}^{3}\|\partial_{i}^sf\|_{L^2}^{2},
\end{align}
we only intend to  estimate  $\|(u^3,b^3)\|_{L^2}^2$ and $\sum_{i=1}^{3}$ $\|\partial_{i}^2(u^3,b^3)\|_{L^2}^{2}$. We split it as two parts. Initially, by the standard $L^2$ energy method, \eqref{eq.MHD}  holds
\begin{align}\label{u3b3l2}
	\frac12\frac{\mathrm{d}}{\dt}  \|(u^3,b^3)\|_{L^2}^2 +\|(\nabla_\h u^3,\partial_1 b^3)\|_{ L^2}^2=-(\partial_3 p|u^3)_{L^2}\eqdefa A.
\end{align}
Since \eqref{p_same} and \eqref{ppp}, we have
\begin{align*}
	A=-\sum_{i,j=1}^3 ((-\Delta)^{-1} \partial_i\partial_j(u^i u^j-b^i b^j) |\mathrm{div_\h}u^\h)_{L^2}=A_{1}+\cdots+A_{6},
\end{align*}
where
\begin{align*}
	A_{1}=&  -\sum_{i,j=1}^2 ((-\Delta)^{-1} \partial_i\partial_j(u^i u^j) |\mathrm{div_\h}u^\h)_{L^2}, &&
	A_{2}=  -2\sum_{j=1}^2 ((-\Delta)^{-1} \partial_3\partial_j(u^3 u^j) |\mathrm{div_\h}u^\h)_{L^2},\\
	A_{3}= & 2((-\Delta)^{-1} \partial_3(u^3 \mathrm{div_h}u^\h) |\mathrm{div_\h}u^\h)_{L^2},&&A_{4}= \sum_{i,j=1}^2 ((-\Delta)^{-1} \partial_i\partial_j(b^i b^j) |\mathrm{div_\h}u^\h)_{L^2}
	,\\
	A_{5}=&  2\sum_{j=1}^2 ((-\Delta)^{-1} \partial_3\partial_j(b^3 b^j) |\mathrm{div_\h}u^\h)_{L^2} , &&A_{6}=  -2((-\Delta)^{-1} \partial_3(b^3 \mathrm{div_h}b^\h) |\mathrm{div_\h}u^\h)_{L^2}.
\end{align*}
To $A_{1}$ and $A_{2}$, using Hölder's inequality, Young's inequality and \eqref{yi} of Lemma \ref{tool} with $f=u^\h$(resp. $u^3$), $g=u^\h$(resp. $u^\h$),  we get
\begin{align}\label{A1}
	A_{1}\lesssim&\||u^\h|^2\|_{L^2} \|\nablah u^\h\|_{L^2} \lesssim\|\nablah u^\h\|_{H^{1}}\|u^\h\|_{H^{1}}\|\nablah u^\h\|_{L^2}
	\nonumber \\
	\leqq&\frac{1}{100}\|\nablah u^\h\|_{L^2}^2+C\|\nablah u^\h\|_{H^{1}}^2\|u^\h\|_{H^1}^2,\\
	\label{A2(v)}
	A_{2}\lesssim& \|u^\h u^3\|_{L^2}\|\nablah u^\h\|_{L^2}\lesssim\|\nablah u^3\|_{H^{1}}^\frac12\|u^3\|_{H^{1}}^\frac12\|\nablah u^\h\|_{L^{2}}^\frac12\|u^\h\|_{L^{2}}^\frac12\|\nablah u^\h\|_{L^2}\nonumber\\
	\leqq&\frac{1}{100}\|\nablah u^\h\|_{L^2}^2+C\|\nablah u^3\|_{H^{1}}^2\|u^\h\|_{L^2}^2+C\|\nablah u^\h\|_{L^2}^2\|u^3\|_{H^{1}}^2.
\end{align}
Using  Young's inequality again, we also obtain the another estimate of $A_{2}$:
\begin{align}\label{A2(h)}
	A_{2}\leqq&\frac{1}{100}\|\nablah u^\h\|_{L^2}^2+\|\nablah u^3\|_{H^{1}}^2\|u^3\|_{H^{1}}^2\|u^\h\|_{L^2}^2.
\end{align}
And for $A_{3}$, by \eqref{L^infty_1D} and \eqref{GNS2d}, we obtain
\begin{align}\label{A3}
	A_{3}\lesssim&\|u^3 \mathrm{div_h}u^\h\| _{L^{\frac43}_\h(L^2_\mathrm{v})}\|(-\Delta)^{-1} \partial_3\mathrm{div_h}u^\h\| _{L^{4}_\h(L^2_\mathrm{v})} \nonumber\\
	\lesssim&\|u^3 \|_{L^{2}_\h(L^\infty_\mathrm{v})}\|\mathrm{div_h}u^\h\| _{L^{4}_\h(L^2_\mathrm{v})}\|u^\h\| _{L^{4}_\h(L^2_\mathrm{v})}\nonumber\\
	\lesssim&\|u^3\|_{L^2}^\frac12\|\partial_3u^3\|_{L^2}^\frac12\|u^\h\|_{L^2}^\frac12\|\nablah u^\h\|_{L^2}\|\nablah^2u^\h\|_{L^2}^\frac12\nonumber\\
	\lesssim&\|u^3\|_{L^2}^\frac12\|u^\h\|_{L^2}^\frac12\|\nablah^2u^\h\|_{L^2}^\frac12\|\nablah u^\h\|_{L^2}^\frac32\nonumber\\
	\leqq&\frac{1}{100}\|\nablah u^\h\|_{L^2}^2+C\|\nablah u^\h\|_{H^{1}}^2\|u^3\|_{L^2}^2\|u^\h\|_{L^2}^2.
\end{align}
In view of $A_{1}-A_{3}$, we bound
\begin{align}\label{A4}
	A_{4}\lesssim &\||b^\h|^2\|_{L^2} \|\nablah u^\h\|_{L^2}\lesssim\|\partial_1b^\h\|_{H^1}^\frac12\|\partial_2b^\h\|_{L^{2}}^\frac12\|b^\h\|_{H^{1}}\|\nablah u^\h\|_{L^2} \nonumber\\
	\leqq&\frac{1}{100}\|\nablah u^\h\|_{L^2}^2+C\left(\|\partial_1b^\h\|_{H^2}^2+\|\partial_2b^\h\|_{H^{1}}^2\right)\|b^\h\|_{H^{2}}^2,\\ \label{A5(v)}
	A_{5}\lesssim&\|b^\h b^3\|_{L^2}\|\nablah u^\h\|_{L^2} \lesssim\|\partial_1b^\h\|_{H^1}^\frac12\|\partial_2b^3\|_{L^{2}}^\frac12\|b^3\|_{L^{2}}^\frac12\|b^\h\|_{H^{1}}^\frac12\|\nablah u^\h\|_{L^2}\nonumber\\
	\leqq&\frac{1}{100}\|\nablah u^\h\|_{L^2}^2+C\|\partial_1b^\h\|_{H^1}^2\|b^3\|_{L^{2}}^2+C\|\partial_2b^3\|_{L^{2}}^2\|b^\h\|_{H^{1}}^2,\\
	\label{A6(v)}
	A_{6}\lesssim&\|b^3\|_{L^2}^\frac12\|\nablah^2b^\h\|_{L^2}^\frac12\|\nablah b^\h\|_{L^2}\|u^\h\|_{L^2}^\frac12\|\nablah u^\h\|_{L^2}^\frac12\nonumber\\
	\leqq&\frac{1}{100}\left(\|\partial_1b^\h\|_{H^2}^2 +\|\partial_2b^\h\|_{H^1}^2\right)+C\|(\partial_1b^\h,\partial_2b^\h)\|_{H^1}^2\|u^\h\|_{L^2}^2+\|\nablah u^\h\|_{L^2}^2\|b^3\|_{L^2}^2.
\end{align}
Using  Young's inequality again, we also obtain the another estimate of $A_{5}$ and $A_6$:
\begin{align}\label{A5(h)}
	A_{5}\leqq&\frac{1}{100}\|(\nablah u^\h,\partial_1b^\h)\|_{L^2}^2+C\|\partial_2b^3\|_{H^1}^2\|b^3\|_{L^{2}}^2\|b^\h\|_{H^{1}}^2,\\
	\label{A6(h)}
	A_{6}\leqq&\frac{1}{100}\left(\|\partial_1b^\h\|_{H^2}^2 +\|\partial_2b^\h\|_{H^1}^2\right)+C\|\nablah u^\h\|_{L^2}^2\|u^\h\|_{L^2}^2\|b^3\|_{L^2}^2.
\end{align}
Combining \eqref{A1}, \eqref{A2(v)}, \eqref{A3}, \eqref{A4}, \eqref{A5(v)} and \eqref{A6(v)}, we get
\begin{align}\label{A(v)}
	A\leqq&\frac{1}{20}\|(\nablah u^\h,\partial_1b^\h)\|_{H^2}^2 +\frac{1}{100}\|\partial_2b^\h\|_{H^1}^2+C\|\nablah (u^\h,u^3)\|_{H^{1}}^2\|u^\h\|_{L^2}^2+C\|\partial_1b^\h\|_{H^1}^2\|b^3\|_{L^{2}}^2\nonumber\\&+C\left(\|\partial_1b^\h\|_{H^2}^2+\|(\partial_2b^\h,\partial_2b^3)\|_{H^{1}}^2\right)\|b^\h\|_{H^{2}}^2+C\|\nablah u^\h\|_{H^1}^2\left(1+\| u^\h\|_{H^1}^2\right)\|u^3\|_{H^{1}}^2.
\end{align}
Inserting \eqref{A(v)} into \eqref{u3b3l2}, this yields
\begin{align}\label{Esti.v.L2}
	\frac12&\frac{\mathrm{d}}{\dt}  \|(u^3,b^3)\|_{L^2}^2 +\|(\nabla_\h u^3,\partial_1 b^3)\|_{ L^2}^2\nonumber\\
	&\leqq \frac{1}{20}\|(\nablah u^\h,\partial_1b^\h)\|_{H^2}^2 +\frac{1}{100}\|\partial_2b^\h\|_{H^1}^2+C\|\nablah (u^\h,u^3)\|_{H^{1}}^2\|u^\h\|_{L^2}^2\nonumber\\&+C\|\partial_1b^\h\|_{H^1}^2\|b^3\|_{L^{2}}^2+C\left(\|\partial_1b^\h\|_{H^2}^2+\|(\partial_2b^\h,\partial_2b^3)\|_{H^{1}}^2\right)\|b^\h\|_{H^{2}}^2\nonumber\\
	&+C\|\nablah u^\h\|_{H^1}^2\left(1+\| u^\h\|_{H^1}^2\right)\|u^3\|_{H^{1}}^2.
\end{align}

Nest, let us treat the estimate of $\sum_{i=1}^3\|\partial_{i}^2(u^3,b^3)\|_{L^2}^{2}$.  We first get, by applying $\partial_{i}^2(i=1,2,3)$ to the equations of  $u^3$ and $b^3$ of \eqref{eq.MHD}  and then taking the $L^2$ inner product to it with $\partial_{i}^2u^3$ and $\partial_{i}^2b^3$, respectively, that
\begin{align}\label{en_H2_v}
	\frac12\frac{\mathrm{d}}{\mathrm{d}t}\sum_{i=1}^{3}\|\partial_{i}^2(u^3,b^3)\|_{L^2}^2+\sum_{i=1}^{3}\|\partial_{i}^2(\nabla_{\mathrm{h}}u^3,\partial_{1}b^3)\|_{L^2}^2&=I_{1}+\cdots+I_{8},
\end{align}
where
\begin{align*}
	I_{1}
	=&-\left(\partial_{i}^2u^j\partial_ju^3+2\partial_{i}u^j\partial_j \partial_{i}u^3|\partial_{i}^2u^3\right)_{L^2},&&I_{2}=\left(\partial_{i}^2b^j\partial_jb^3+2\partial_{i}b^j\partial_j \partial_{i}b^3|\partial_{i}^2u^3\right)_{L^2},\\
	I_{3}
	=&-\left(\partial_{i}^2u^j\partial_jb^3+2\partial_{i}u^j\partial_j \partial_{i}b^3|\partial_{i}^2b^3\right)_{L^2},&&I_{4}=\left(\partial_{i}^2b^j\partial_ju^3+2\partial_{i}b^j\partial_j \partial_{i}u^3|\partial_{i}^2b^3\right)_{L^2},\\
	I_{5}
	=&\left(-u\cdot\nabla \partial_{i}^2u^3+b\cdot\nabla \partial_{i}^2b^3|\partial_{i}^2u^3\right)_{L^2},&&I_{6}=\left(-u\cdot\nabla \partial_{i}^2b^3+b\cdot\nabla \partial_{i}^2u^3|\partial_{i}^2b^3\right)_{L^2},\\
	I_{7}=&\left( \partial_{i}^2\partial_{2}b^3|\partial_{i}^2u^3 \right)_{L^2}+\left( \partial_{i}^2\partial_{2}u^3|\partial_{i}^2b^3 \right)_{L^2} ,&&I_{8}=-\left(\partial_{i}^2\partial_{3}p|\partial_{i}^2u^3\right)_{L^2} \quad \mathrm{with}\quad i,j=1,2,3.
\end{align*}
By virtue of $\partial_3u^3=-\mathrm{div_h}u^\h$, the expression for $I_1$, which exhibits numerous repeated terms, is simplified as follows
\begin{align}\label{I1rewrite}
	I_{1}\leqq& C\int_{\R^3}|\nablah^2u^3\nablah u^\h \nablah^2 u^3|+|\nablah^2u^\h\nablah u^3 \nablah^2 u^3|+|\partial_3^2u^\h\nablah u^3 \nablah\partial_3 u^\h|\dx\nonumber\\
	&+C\int_{\R^3}|\nablah\nabla u^\h\nabla u^\h \nablah \partial_3u^\h|\dx\eqdefa I_{11}+\cdots+I_{14}.
\end{align}
Applying \eqref{er} with $f=u^3$ (resp. $u^3,u^\h,u^\h$), $g=u^{\mathrm{h}}$ (resp. $\nablah u^\h,u^3,u^\h$)  and $h=\nabla_{\mathrm{h}}u^3$ (resp. $u^{3},\partial_3u^\h,\nablah u^\h$), it holds that
\begin{align*}
	I_{11}\lesssim&\|\partial_1u^3\|_{H^2}^\frac12\|u^3\|_{H^2}^\frac12\|\nablah u^\h\|_{H^1}\|\nablah u^3\|_{H^2}^\frac12\|u^3\|_{H^2}^\frac12\lesssim \|\nablah u^3\|_{H^2}\|\nablah u^\h\|_{H^2}\| u^3\|_{H^2},\\
	I_{12}\lesssim&\|\partial_1u^3\|_{H^2}^\frac12\|u^3\|_{H^2}^\frac12\|\nablah u^\h\|_{H^2}\|\nablah u^3\|_{H^1}^\frac12\|u^3\|_{H^1}^\frac12\lesssim \|\nablah u^3\|_{H^2}\|\nablah u^\h\|_{H^2}\| u^3\|_{H^2},\\
	I_{13}\lesssim&\|\partial_1u^\h\|_{H^2}^\frac12\|u^\h\|_{H^2}^\frac12\|\nablah u^3\|_{H^1}\|\nablah u^\h\|_{H^2}^\frac12\|u^\h\|_{H^2}^\frac12\lesssim \|\nablah u^3\|_{H^2}\|\nablah u^\h\|_{H^2}\| u^\h\|_{H^2},\\
	I_{14}\lesssim&\|\partial_1u^\h\|_{H^2}^\frac12\|u^\h\|_{H^2}^\frac12\|\nablah u^\h\|_{H^1}\|\nablah u^\h\|_{H^2}^\frac12\|u^\h\|_{H^2}^\frac12\lesssim \|\nablah u^\h\|_{H^2}\|\nablah u^\h\|_{H^2}\| u^\h\|_{H^2}.\\
\end{align*}
In a conclusion, $I_{1}$ can be handled as
\begin{align}\label{I_1}
	I_{1}&\lesssim \|\nablah u^3\|_{H^2}\|\nablah u^\h\|_{H^2}\| (u^\h,u^3)\|_{H^2}+\|\nablah u^\h\|_{H^2}^2\| u^\h\|_{H^2}\nonumber\\
	&\leqq\frac{1}{100}\|(\nablah u^\h,\nablah u^3)\|_{H^{2}}^{2}+C\|\nablah u^\h\|_{H^2}^2\| (u^\h,u^3)\|_{H^2}^2.
\end{align}
In the same argument, as to $I_2$, we have
\begin{align*}
	I_2\leqq &C\int_{\R^3}|\nablah^2b^3\nablah b^\h \nablah^2 u^3|+|\nablah^2b^\h\nablah b^3 \nablah^2 u^3|+|\partial_3^2b^\h\nablah b^3 \nablah\partial_3 u^\h|\dx\\
	&+C\int_{\R^3}|\nablah\nabla b^\h\nabla b^\h \nablah \nabla u^\h|\dx\eqdefa I_{21}+\cdots+I_{24}.
\end{align*}
Making the best of \eqref{er} with $f=b^3$ (resp. $b^\h,b^\h$), $g=b^\h$ (resp. $ b^3,b^3$)  and $h=\nabla_{\mathrm{h}}u^3$ (resp. $\nablah u^{3},\partial_3u^\h$), we obtain
\begin{align*}
	I_{21}\lesssim&\|\partial_1b^3\|_{H^2}^\frac12\|(\partial_1b^3,\partial_2b^3)\|_{H^1}^\frac12\|b^\h\|_{H^2}\|\nablah u^3\|_{H^2},\\
	I_{22}\lesssim&\|\partial_1b^\h\|_{H^2}^\frac12\|b^\h\|_{H^2}^\frac12\|(\partial_1b^3,\partial_2b^3)\|_{H^1}\|\nablah u^3\|_{H^2}^\frac12\|u^3\|_{H^2}^\frac12,\\
	I_{23}\lesssim&\|\partial_1b^\h\|_{H^2}^\frac12\|b^\h\|_{H^2}^\frac12\|(\partial_1b^3,\partial_2b^3)\|_{H^1}\|\nablah u^\h\|_{H^2}^\frac12\|u^\h\|_{H^2}^\frac12.
\end{align*}
For $I_{24}$, utilizing \eqref{L^infty_1D}, we have
\begin{align*}
	I_{24}\lesssim&\|\nablah\nabla b^\h\|_{L^\infty_{x_1}L^2_{x_2}L^2_{x_3}}\|\nabla b^\h \|_{L^2_{x_1}L^\infty_{x_2}L^2_{x_3}}\|\nablah \nabla u^\h\|_{L^2_{x_1}L^2_{x_2}L^\infty_{x_3}}\\
	\lesssim&\|\partial_1b^\h\|_{H^2}^\frac12\|b^\h\|_{H^2}^\frac12\|\partial_2b^\h\|_{H^1}^\frac12\|b^\h\|_{H^2}^\frac12\|\nablah u^\h\|_{H^2}.
\end{align*}
To a conclusion, $I_{2}$ can be estimated as
\begin{align}\label{I_2}
	I_2\lesssim& \left(\|\partial_1b^3\|_{H^2}^\frac12\|(\partial_1b^3,\partial_2b^3)\|_{H^1}^\frac12\|\nablah u^3\|_{H^2}+\|\partial_1b^\h\|_{H^2}^\frac12\|\partial_2b^\h\|_{H^1}^\frac12\|\nablah u^\h\|_{H^2}\right)\|b^\h\|_{H^2}\nonumber\\
	&+\|\partial_1b^\h\|_{H^2}^\frac12\|b^\h\|_{H^2}^\frac12\|(\partial_1b^3,\partial_2b^3)\|_{H^1}\left(\|\nablah u^\h\|_{H^2}^\frac12\|u^\h\|_{H^2}^\frac12+\|\nablah u^3\|_{H^2}^\frac12\|u^3\|_{H^2}^\frac12\right)\nonumber\\
	\leqq&\frac{1}{100}\left(\|(\partial_1b^\h,\partial_1b^3)\|_{H^2}^2+\|(\partial_2b^\h,\partial_2b^3)\|_{H^1}^2\right)+C\|(\nablah u^\h,\nablah u^3)\|_{H^2}^2\|b^\h\|_{H^2}^2\nonumber\\
	&+C\|\partial_1b^\h\|_{H^2}^2\|(u^\h,u^3)\|_{H^2}^2.
\end{align}
Indeed, for the specified terms of $I_{3},I_{4}$, a subset of them admit estimates identical to those for term $I_{2}$. Furthermore, there are repeated terms among them as well. Thus, we consolidate them into the following expressions
\begin{align*}
	I_{3}+I_{4}\leqq &C\int_{\R^3}|\nablah^2b^3\nablah b^\h \nablah^2 u^3|+|\nablah^2b^3\nablah b^3 \nablah^2 u^\h|+|\nablah\partial_3b^\h\nablah b^3 \partial_3^2 u^\h|\dx\\
	&+C\int_{\R^3}|\nablah^2b^\h \nablah u^3\nablah^2 b^3|+|\nablah^2b^3\nablah u^\h\nablah^2 b^3 |+|\partial_3^2b^\h\nablah u^3\nablah\partial_3b^\h|\dx
	\\
	&+C\int_{\R^3}|\nablah\nabla b^\h\nabla b^\h \nablah \nabla u^\h|+|\nablah\nabla b^\h\nabla u^\h\nablah\partial_3 b^\h  |\dx\eqdefa I_{31}+\cdots +I_{38},
\end{align*}
into the following terms, where the term $I_{31},I_{37}$ share same estimate with $I_{21},I_{24}$, respectively. And then we treat the other terms. For $I_{32}$, by getting the utmost of \eqref{er} with $f=b^3, g=b^3, h=\nablah u^\h$, we deduce that
\begin{align*}
	I_{32}\lesssim\|\partial_1b^3\|_{H^2}^\frac12\|(\partial_1b^3,\partial_2b^3)\|_{H^1}^\frac12\|b^3\|_{H^2}\|\nablah u^\h\|_{H^2}
	\lesssim\left(\|\partial_1b^3\|_{H^2}+\|\partial_2b^3\|_{H^1}\right)\|b^3\|_{H^2}\|\nablah u^\h\|_{H^2}.
\end{align*}
In view of $I_{24}$, taking full advantage of  \eqref{L^infty_1D}, we achieve
\begin{align*}
	I_{33}\lesssim &\|\nablah\partial_3b^\h\|_{L^\infty_{x_1}L^2_{x_2}L^2_{x_3}}\|\nablah b^3 \|_{L^2_{x_1}L^2_{x_2}L^\infty_{x_3}}\|\partial_3^2 u^\h \|_{L^2_{x_1}L^\infty_{x_2}L^2_{x_3}}\\
	\lesssim&\|\partial_1b^\h\|_{H^2}^\frac12\|b^\h\|_{H^2}^\frac12\|(\partial_1b^3,\partial_2b^3)\|_{H^1}\|\nablah u^\h\|_{H^2}^\frac12\| u^\h\|_{H^2}^\frac12,\\
	I_{38}\lesssim &\|\nablah\nabla b^\h\|_{L^\infty_{x_1}L^2_{x_2}L^2_{x_3}}\|\nablah\partial_3 b^\h \|_{L^2}\|\nabla u^\h \|_{L^2_{x_1}L^\infty_{x_2}L^\infty_{x_3}}\\
	\lesssim&\|\partial_1b^\h\|_{H^2}^\frac12\|b^\h\|_{H^2}^\frac12\|(\partial_1b^\h,\partial_2b^\h)\|_{H^1}\|\nablah u^\h\|_{H^2}^\frac12\| u^\h\|_{H^2}^\frac12.
\end{align*}
With respect to $I_{34}-I_{36}$, by H\"older's inequality and Sobolev embedding theorem, it implies
\begin{align*}
	I_{34}\lesssim&\|\nablah^2b^\h \|_{L^2}\|\nablah u^3\|_{L^\infty}\|\nablah^2 b^3\|_{L^2}\lesssim\|b^\h \|_{H^2}\|\nablah u^3\|_{H^2}\|(\partial_1b^3,\partial_2b^3)\|_{H^1},\\
	I_{35}\lesssim&\|\nablah^2b^3 \|_{L^2}\|\nablah u^\h\|_{L^\infty}\|\nablah^2 b^3\|_{L^2}\lesssim\|(\partial_1b^3,\partial_2b^3)\|_{H^1}\|\nablah u^\h\|_{H^2}\|b^3 \|_{H^2},\\
	I_{36}\lesssim&\|\partial_3^2b^\h \|_{L^2}\|\nablah u^3\|_{L^\infty}\|\nablah\partial_3 b^3\|_{L^2}\lesssim\|b^\h \|_{H^2}\|\nablah u^3\|_{H^2}\|(\partial_1b^3,\partial_2b^3)\|_{H^1}.
\end{align*}
As a result, we have
\begin{align}\label{I_3,I_4}
	I_{3}+I_{4}\lesssim&\left(\|\partial_1b^\h\|_{H^2}^\frac12\|\partial_2b^\h\|_{H^1}^\frac12\|\nablah u^\h\|_{H^2}+\left(\|\partial_1b^3\|_{H^2}+\|\partial_2b^3\|_{H^1}\right)\|\nablah u^3\|_{H^2}\right)\|b^\h\|_{H^2}\nonumber\\
	&+\left(\|(\partial_1b^\h,\partial_1b^3)\|_{H^2}+\|(\partial_2b^\h,\partial_2b^3)\|_{H^1}\right)\|\nablah u^\h\|_{H^2}^\frac12\|\partial_1b^\h\|_{H^2}^\frac12\| u^\h\|_{H^2}^\frac12\|b^\h\|_{H^2}^\frac12\nonumber\\
	&+\left(\|\partial_1b^3\|_{H^2}+\|\partial_2b^3\|_{H^1}\right)\|\nablah u^\h\|_{H^2}\|b^3 \|_{H^2}\nonumber\\
	\leqq& \frac{1}{100}\left(\|(\partial_1b^\h,\partial_1b^3)\|_{H^2}^2+\|(\partial_2b^\h,\partial_2b^3)\|_{H^1}^2\right)+C\|(\nablah u^\h,\nablah u^3)\|_{H^2}^2\|b^\h\|_{H^2}^2\nonumber\\
	&+C\|\partial_1b^\h\|_{H^2}^2\| u^\h\|_{H^2}^2+C\|\nablah u^\h\|_{H^2}^2\|b^3 \|_{H^2}^2.
\end{align}
By integration by parts and divergence free condition of $u$ and $b$, it implies $I_{5}+I_{6}=I_{7}=0$. At last, for $I_{8}$, by H\"older's inequality, we find
\begin{align*}
	I_{8}&=-
	\left(\partial_{i}^2p|\partial_{i}^2\mathrm{div}_\h u^\h\right)_{L^2}\lesssim\|\partial_{i}^2p\|_{L^2}\|\nabla_{\mathrm{h}}u^\mathrm{h}\|_{H^{2}}.
\end{align*}
And then, we note that \eqref{ppp}
and  $\mathrm{div}u=0$, it is easy to find
\begin{align*}
	\|\partial_{i}^2p\|_{L^2}=\|\partial_{i}^2(-\Delta)^{-1}\mathrm{divdiv}(u\otimes u-b\otimes b)\|\lesssim\|\partial_iu^j\partial_ju^i-\partial_ib^j\partial_jb^i\|_{L^2},
\end{align*}
where $i,j=1,2,3$. Applying \eqref{yi} with $f=\nablah u^\h$ (resp. $\partial_3 u^\h$) and $g=\nablah u^\h$ (resp. $\nablah u^3$), it holds
\begin{align*}
	\|\partial_iu^j\partial_ju^i\|_{L^2}\lesssim&\||\nablah u^\h|^2\|_{L^2}+\|\partial_3u^\h\cdot\nablah u^3\|_{L^2}\\
	\lesssim&\|\partial_1u^\h\|_{H^2}^\frac12\|\partial_2u^\h\|_{H^1}^\frac12\|u^\h\|_{H^2}+\|\partial_1u^\h\|_{H^2}^\frac12\|\partial_2u^3\|_{H^1}^\frac12\|u^\h\|_{H^2}^\frac12\|u^3\|_{H^2}^\frac12.
\end{align*}
By directly substituting $u$ with $b$ in the above estimates, we immediately arrive at
\begin{align*}
	\|\partial_ib^j\partial_jb^i\|_{L^2}\lesssim\|\partial_1b^\h\|_{H^2}^\frac12\|\partial_2b^\h\|_{H^1}^\frac12\|b^\h\|_{H^2}+\|\partial_1b^\h\|_{H^2}^\frac12\|\partial_2b^3\|_{H^1}^\frac12\|b^\h\|_{H^2}^\frac12\|b^3\|_{H^2}^\frac12.
\end{align*}
Combining the above estimates, we obtain
\begin{align}\label{P_H2}
	\|\partial_{i}^2p\|_{L^2}\lesssim&\|\nablah u^\h\|_{H^2}\|u^\h\|_{H^2}+ \|\nablah u^\h\|_{H^2}^\frac12\|\nablah u^3\|_{H^2}^\frac12\|u^\h\|_{H^2}^\frac12\|u^3\|_{H^2}^\frac12 \nonumber\\
	&+\|\partial_1b^\h\|_{H^2}^\frac12\|\partial_2b^\h\|_{H^1}^\frac12\|b^\h\|_{H^2}+\|\partial_1b^\h\|_{H^2}^\frac12\|\partial_2b^3\|_{H^1}^\frac12\|b^\h\|_{H^2}^\frac12\|b^3\|_{H^2}^\frac12.
\end{align}
Hence, we conclude that
\begin{align}\label{ii}
	I_8\lesssim&\left(\|\nablah u^\h\|_{H^2}\|u^\h\|_{H^2}+ \|\nablah u^\h\|_{H^2}^\frac12\|\nablah u^3\|_{H^2}^\frac12\|u^\h\|_{H^2}^\frac12\|u^3\|_{H^2}^\frac12 \right)\|\nablah u^\h\|_{H^{2}}\nonumber\\
	&+\left(\|\partial_1b^\h\|_{H^2}^\frac12\|\partial_2b^\h\|_{H^1}^\frac12\|b^\h\|_{H^2}+\|\partial_1b^\h\|_{H^2}^\frac12\|\partial_2b^3\|_{H^1}^\frac12\|b^\h\|_{H^2}^\frac12\|b^3\|_{H^2}^\frac12\right)\|\nablah u^\h\|_{H^{2}}.
\end{align}
Further, by Young's inequality, we find
\begin{align}\label{I_8v}
	I_8	\leqq& \frac{1}{100}\left( \|(\nablah u^\h,\nablah u^3,\partial_1b^\h)\|_{H^2}^2+\|(\partial_2b^\h,\partial_2b^3)\|_{H^1}^2\right)\nonumber\\
	&+C\|\nablah u^\h\|_{H^{2}}^2\left(\|(u^\h,b^\h)\|_{H^2}^2+\|(u^3,b^3)\|_{H^2}^2\right).
\end{align}
Combining \eqref{u3b3l2}, \eqref{I_1}, \eqref{I_2}, \eqref{I_3,I_4} and \eqref{I_8v}, we get \eqref{u3b3H2}.

\subsection{The proof of \eqref{p2bH1}}
We only need to obtain the estimate of $\|(\partial_2b^\h,\partial_2 b^3)\|_{L^2}$ and $\sum_{i=1}^3\|(\partial_i\partial_2b^\h,\partial_i\partial_2 b^3)\|_{L^2}$ owing to \eqref{Hsdef}. By taking $L^2$ inner product to $u$ equation of \eqref{eq.MHD} with $\partial_2b$, it yields
\begin{align}\label{p2b_enl2}
	\|(\partial_2b^\h,\partial_2 b^3)\|_{L^2}^2	=&\left(\partial_t u^\h+ u\cdot\nabla u^\h -\Delta_{\h}u^\h- b\cdot\nabla b^\h+\nablah p|\partial_2 b^\h\right)_{L^2}\nonumber\\
	&+\left(\partial_t u^3+ u\cdot\nabla u^3 -\Delta_{\h}u^3- b\cdot\nabla b^3+\partial_3 p|\partial_2 b^3\right)_{L^2} \eqdefa F_1+\cdots+F_{10}.
\end{align}
For $F_1$ and $F_6$, by integration by parts and \eqref{ha}, we observe
\begin{align}\label{f1deco}
	F_1 &=\int \frac{\mathrm{d}}{\mathrm{d} t}\left(u^\h\partial_2 b^\h\right)+\partial_2u^\h\partial_t b^\h \mathrm{d} x\nonumber\\&
	=\int \frac{\mathrm{d}}{\mathrm{d} t}\left(u^\h\partial_2 b^\h\right)\mathrm{d} x+\int\partial_2 u^\h\left(-u\cdot\nabla b^\h +\partial_1^2 b^\h+b\cdot\nabla u^\h +\partial_2 u^\h\right)\mathrm{d} x,
\end{align}
and
\begin{align*}
	F_6 &=\int \frac{\mathrm{d}}{\mathrm{d} t}\left(u^3\partial_2 b^3\right)+\partial_2u^3\partial_t b^3 \mathrm{d} x\\&
	=\int \frac{\mathrm{d}}{\mathrm{d} t}\left(u^3\partial_2 b^3\right)\mathrm{d} x+\int\partial_2 u^3\left(-u\cdot\nabla b^3 +\partial_1^2 b^3+b\cdot\nabla u^3 +\partial_2 u^3\right)\mathrm{d} x.
\end{align*}
Applying H\"older's inequality and \eqref{yi} with $f=\nabla b^\h$ (resp. $b$), $g=u$ (resp. $\nabla u^\h$), we have
\begin{align}\label{f_1}
	\int_{\R^3}\partial_2 &u^\h\left(-u\cdot\nabla b^\h +b\cdot\nabla u^\h\right)\mathrm{d} x+\int_{\R^3}\partial_2 u^\h\left(\partial_1^2 b^\h +\partial_2 u^\h\right)\mathrm{d}x\nonumber\\
	&\leqq \|\partial_2u^\h\|_{L^2}\left(\|u\cdot\nabla b^\h\|_{L^2}+\|b\cdot\nabla u^\h\|_{L^2}\right)+\|\partial_2u^\h\|_{L^2}\|\partial_1^2 b^\h\|_{L^2}+\|\partial_2u^\h\|_{L^2}^2\nonumber\\
	&\leqq C\|\nablah u^\h\|_{H^2}\|(\nablah u^\h,\nablah u^3)\|_{H^2}^\frac12\|\partial_1 b^\h\|_{H^2}^\frac12\|( u^\h, u^3)\|_{H^2}^\frac12\|b^\h\|_{L^2}^\frac12+\|\nablah u^\h\|_{H^2}\|\partial_1 b^\h\|_{H^2}\nonumber\\
	&\quad+C\|\nablah u^\h\|_{H^2}\|(\partial_1b^\h,\partial_1b^3)\|_{H^2}^\frac12\|\nablah u^\h\|_{H^2}^\frac12\|(b^\h,b^3)\|_{H^2}^\frac12\| u^\h\|_{H^2}^\frac12+\|\nablah u^\h\|_{H^2}^2.
\end{align}
Then, applying $\mathrm{div} u=\mathrm{div} b=0$, H\"older's inequality and Sobolev embedding theorem, we find that
\begin{align*}
	\int_{\R^3}\partial_2 &u^3\left(b^\h \cdot\nabla_\h u^3- b^3\mathrm{div_h}u^\h -u^\h\cdot\nabla_\h b^3+u^3\mathrm{div_h} b^\h \right)\mathrm{d} x+\int_{\R^3}\partial_2 u^3\left(\partial_1^2 b^3 +\partial_2 u^3\right)\mathrm{d}x
	\\&\leqq\|\partial_2 u^3\|_{L^\infty}\left( \|b^\h\|_{L^2}\|\nabla_\h u^3\|_{L^2}+\|b^3\|_{L^2}\|\mathrm{div_h}u^\h\|_{L^2}+\|u^\h\|_{L^2}\|\nabla_\h b^3\|_{L^2}\right)\\
	&\quad+\|\partial_2 u^3\|_{L^\infty}\|u^3\|_{L^2}\|\mathrm{div_h} b^\h\|_{L^2}+\|\partial_2u^3\|_{L^2}\|\partial_1^2 b^3\|_{L^2}+\|\partial_2u^3\|_{L^2}^2\\
	&\leqq C\|\nablah u^\h\|_{H^2}\left( \|\nabla_\h u^3\|_{L^2}\|b^\h\|_{L^2}+\|\nabla_\h u^\h\|_{L^2}\|b^3\|_{L^2}+\|(\partial_1 b^3,\partial_2 b^3)\|_{L^2}\|u^\h\|_{L^2}\right)\\
	&\quad+C\|\nablah u^\h\|_{H^2}\left(\|\partial_1 b^\h\|_{L^2}+\|\partial_2 b^\h\|_{L^2}\right)\|u^3\|_{L^2}+\|\nablah u^3\|_{L^2}\|\partial_1 b^3\|_{H^1}+\|\nablah u^3\|_{L^2}^2.
\end{align*}
Hence, by Young inequality, we have
\begin{align}\label{F1}
	F_1&\leqq\frac{\mathrm{d}}{\mathrm{d} t}\int_{\R^3} u^\h\partial_2 b^\h\mathrm{d} x+(\frac32+\frac{1}{100})\|\nablah u^\h\|_{H^2}^2+\frac12\|\partial_1 b^\h\|_{H^2}^2+C\|\nablah u^3\|_{H^2}^2\|b^\h\|_{H^2}^2\nonumber\\
	&\quad+C\|\partial_1 b^\h\|_{H^2}^2\| u^3\|_{H^2}^2+C\|(\partial_1b^\h,\partial_1b^3)\|_{H^2}^2\| u^\h\|_{H^2}^2+C\|\nablah u^\h\|_{H^2}^2\|(b^\h,b^3)\|_{H^2}^2,\\\label{F6}
	F_6&\leqq\frac{\mathrm{d}}{\mathrm{d} t}\int_{\R^3}u^3\partial_2 b^3\mathrm{d} x+\frac32\|\nablah u^3\|_{H^2}^2+\frac12\|\partial_1 b^3\|_{H^2}^2+\frac{1}{100}\left( \|\nablah u^\h\|_{H^2}^2+\|\partial_2 b^\h\|_{H^1}^2\right)\nonumber\\
	&\quad+C\|\nabla_\h u^3\|_{H^2}^2\|b^\h\|_{H^2}^2+C\|\nabla_\h u^\h\|_{H^2}^2\|b^3\|_{H^2}^2+C\|(\nablah u^\h,\partial_1 b^\h)\|_{H^2}^2\|u^3\|_{H^2}^2\nonumber\\
	&\quad+C\left(\|\partial_1 b^3\|_{H^2}^2+\|\partial_2 b^3\|_{H^1}^2\right)\|u^\h\|_{H^2}^2.
\end{align}
To $F_2$ and $F_4$, by applying \eqref{yi} with $f=\nabla u^\h$ (resp. $\nabla b^\h$) and $g=u$ (resp. $b$), we obtain
\begin{align}\label{f_2}
	F_2\lesssim &\|\partial_2 b^\h\|_{L^2}\|u\cdot \nabla u^\h\|_{L^2}\nonumber\\
	\lesssim &\|\partial_2 b^\h\|_{H^1}\|\nablah u^\h\|_{H^2}^\frac12\|(\nablah u^\h,\nablah u^3)\|_{H^2}^\frac12\| u^\h\|_{H^2}^\frac12\|( u^\h, u^3)\|_{H^2}^\frac12,\\
	\label{f_4}
	F_4\lesssim &\|\partial_2 b^\h\|_{L^2}\|b\cdot \nabla b^\h\|_{L^2}\nonumber\\
	\lesssim &\|\partial_2 b^\h\|_{H^1}\|\partial_1 b^\h\|_{H^2}^\frac12\|(\partial_2 b^\h,\partial_2 b^3)\|_{H^1}^\frac12\| b^\h\|_{H^2}^\frac12\|( b^\h, b^3)\|_{H^2}^\frac12.
\end{align}
By Young's inequality, we bound $F_2$ and $F_4$ as
\begin{align}\label{F2}
	F_2\leqq &\frac{1}{100} \|\partial_2 b^\h\|_{H^1}^2+C\|(\nablah u^\h,\nablah u^3)\|_{H^2}^2\| u^\h\|_{H^2}^2+\|\nablah u^\h\|_{H^2}^2\| u^3\|_{H^2}^2,\\\label{F4}
	F_4\leqq &\frac{1}{100} \|\partial_2 b^\h\|_{H^1}^2+C\|(\partial_2 b^\h,\partial_2 b^3)\|_{H^1}^2\| b^\h\|_{H^2}^2+\|\partial_1 b^\h\|_{H^2}^2\|( b^\h, b^3)\|_{H^2}^2.
\end{align}
With respect to $F_7,F_9$, due to $\mathrm{div} u=\mathrm{div} b=0$, we rewrite them as
\begin{align*}
	F_7=\left(u^\h\cdot \nablah u^3-u^3\mathrm{div_h} u^\h|\partial_2 b^3\right)_{L^2}, \quad\mathrm{and}\quad F_9=\left(b^\h\cdot \nablah b^3-b^3\mathrm{div_h} b^\h|\partial_2 b^3\right)_{L^2}.
\end{align*}
In the same as $F_6$, we have
\begin{align}\label{F7}
	F_7\lesssim&\|\partial_2 b^3\|_{L^2}\left( \|u^\h\|_{L^\infty}\| \nablah u^3\|_{L^2}+\|u^3\|_{L^\infty}\|\mathrm{div_h} u^\h\|_{L^2}\right)\nonumber\\
	\lesssim&\|\partial_2 b^3\|_{L^2}\left( \|u^\h\|_{H^2}\| \nablah u^3\|_{L^2}+\|u^3\|_{H^2}\|\nablah u^\h\|_{L^2}\right)\nonumber\\
	\leqq &\frac{1}{100} \|\partial_2 b^3\|_{H^1}^2+C\| \nablah u^3\|_{H^2}^2\|u^\h\|_{H^2}^2+C\|\nablah u^\h\|_{H^2}^2\|u^3\|_{H^2}^2.
\end{align}
For $F_9$, making the best of \eqref{yi} with $f=\mathrm{div_h} b^\h$ and $g=b^3$ to the second term, it yields
\begin{align}\label{F9}
	F_9\lesssim&\|\partial_2 b^3\|_{L^2}\left( \|b^\h\|_{L^\infty}\| \nablah b^3\|_{L^2}+\|b^3\mathrm{div_h} b^\h\|_{L^2}\right)\nonumber\\
	\lesssim&\|\partial_2 b^3\|_{L^2}\left( \|b^\h\|_{H^2}\|( \partial_1b^3,\partial_2b^3)\|_{L^2}+\|\partial_2b^3\|_{H^1}^\frac12\| \partial_1b^\h\|_{H^2}^\frac12\|b^3\|_{H^1}^\frac12\| b^\h\|_{H^2}^\frac12\right)\nonumber\\
	\leqq &\frac{1}{100} \|\partial_2 b^3\|_{H^1}^2+C\left(\|\partial_1b^3\|_{H^2}^2+\|\partial_2b^3\|_{H^1}^2\right)\|b^\h\|_{H^2}^2+\| \partial_1b^\h\|_{H^2}^2\|b^3\|_{H^1}^2.
\end{align}
By  H\"older's inequality, we see that
\begin{align}\label{F3}
	F_3\leqq& \|\Delta_\h u^\h\|_{L^2}\|\partial_2 b^\h\|_{L^2}\leqq \frac12\|\nabla_\h u^\h\|_{H^1}^2+\frac12\|\partial_2 b^\h\|_{L^2},\\\label{F8}
	F_8\leqq &\|\Delta_\h u^3\|_{L^2}\|\partial_2 b^3\|_{L^2}\leqq \frac12\|\nabla_\h u^3\|_{H^1}^2+\frac12\|\partial_2 b^3\|_{L^2}.
\end{align}
Due to $\mathrm{div} b=0$, we have $F_5+F_{10}=0$. Inserting \eqref{F1}, \eqref{F6} and \eqref{F2}-\eqref{F8} into \eqref{p2b_enl2}, we get
\begin{align}\label{Esti.p2bL2}
	\|(\partial_2b^\h,\partial_2 b^3)\|_{L^2}^2\leqq& \frac{\mathrm{d}}{\mathrm{d} t}\int_{\R^3} u^\h\partial_2 b^\h+u^3\partial_2 b^3\mathrm{d} x+2\|\nablah u^3\|_{H^2}^2+\frac12\|\partial_1 b^3\|_{H^2}^2+\frac12\|(\partial_2 b^\h,\partial_2b^3)\|_{L^2}^2\nonumber\\
	&+\frac{3}{100}\|(\partial_2 b^\h,\partial_2b^3)\|_{H^1}^2+C\|(\nablah u^\h,\partial_1 b^\h)\|_{H^2}^2\left(1+\|(u^\h,b^\h)\|_{H^2}^2\right)\nonumber\\
	&+C\left(\|(\nablah u^3,\partial_1b^3)\|_{H^2}^2+\|(\partial_2 b^\h,\partial_2 b^3)\|_{H^1}^2\right)\|(u^\h,b^\h)\|_{H^2}^2\nonumber\\
	&+C\|(\nablah u^\h,\partial_1 b^\h)\|_{H^2}^2\|(u^3,b^3)\|_{H^2}^2.
\end{align}

Then let us treat $\sum_{i=1}^3\|(\partial_i\partial_2b^\h,\partial_i\partial_2b^3)\|_{L^2}$. We first get, by applying $\partial_{i}$ to the equations of  $u$ and then taking the $L^2$ inner product to it with $\partial_{i}\partial_{2}b$, that
\begin{align}\label{p2b_enH1}
	\|(\partial_i\partial_2b^\h,\partial_i\partial_2b^3)\|_{L^2}^2=&\left(\partial_{t}\partial_{i}u^\h+\partial_{i}(u\cdot \nabla u^\h)-\Delta_{\mathrm{h}}\partial_{i}u^\h-\partial_{i}(b\cdot \nabla b^\h)+\nablah\partial_{i}p|\partial_{i}\partial_{2}b^\h\right)_{L^2}\nonumber\\
	&+\left(\partial_{t}\partial_{i}u^3+\partial_{i}(u\cdot \nabla u^3)-\Delta_{\mathrm{h}}\partial_{i}u^3-\partial_{i}(b\cdot \nabla b^3)+\partial_{3}\partial_{i}p|\partial_{i}\partial_{2}b^3\right)_{L^2}\nonumber\\\eqdefa &K_{1}+\cdots+K_{10}.
\end{align}
In order to estimate $K_{1}$ and $K_{6}$, we get, by integration by parts, that
\begin{align*}
	K_{1}=&\frac{\mathrm{d}}{\mathrm{d}t}\int_{\R^3}\partial_{i}u^\h\partial_{i}\partial_{2}b^\h\mathrm{d}x+\int_{\R^3}\partial_{i}^2u^\h\partial_{t}\partial_{2}b^\h\mathrm{d}x,\\
	K_{6}=&\frac{\mathrm{d}}{\mathrm{d}t}\int_{\R^3}\partial_{i}u^3\partial_{i}\partial_{2}b^3\mathrm{d}x+\int_{\R^3}\partial_{i}^2u^3\partial_{t}\partial_{2}b^3\mathrm{d}x.
\end{align*}
and then by using the $b$ equation, the second term of $K_1$ can be treated as
\begin{align*}
	\int_{\R^3}\partial_{i}^2&u^\h\partial_{t}\partial_{2}b^\h\mathrm{d}x=\int_{\R^3} \partial_{i}^2u^\h\partial_{2}\left(b\cdot\nabla u^\h+\partial^{2}_{1}b^\h-u\cdot\nabla b^\h+\partial_{2}u^\h\right)\mathrm{d}x\\&=
	\int_{\R^3} \partial_{i}^2u^\h\left(\partial_{2}b\cdot\nabla u^\h+b\cdot\nabla \partial_{2}u^\h+\partial^{2}_{1}\partial_{2}b^\h-\partial_{2}u\cdot\nabla  b^\h-u \cdot\nabla  \partial_{2}b^\h+\partial_{2}^2u^\h\right)\mathrm{d}x
	\\&\eqdefa K^{1}_{1}+\cdots+ K^{6}_{1}.
\end{align*}
Getting the utmost of \eqref{er} with $f=\partial_{i}^2 u^\h$, $g=b$ (resp. $u$), $h=u^\h$(resp. $b^\h$), it yields
\begin{align*}
	K^{1}_{1}\lesssim\|\nabla_{\mathrm{h}}u^\h\|_{H^{2}}\|(\partial_2b^\mathrm{h},\partial_2b^3)\|_{H^{1}}\|u^\h\|_{H^{2}},\quad  K^{4}_{1}\lesssim\|\nabla_{\mathrm{h}}u^\h\|_{H^{2}}\|(\nabla_{\mathrm{h}}u^\h,\nabla_{\mathrm{h}}u^3)\|_{H^{2}}\|b^\h\|_{H^{2}}.
\end{align*}
And by using H\"{o}lder inequality, \eqref{yi} with $f=b,g=\nabla \partial_{2}u^\h$ and \eqref{L^infty_1D}, we observe
\begin{align*}
	K^{2}_{1}\lesssim&\|\partial_{i}^2u^\h\|_{L^{2}}\|b\cdot \nabla \partial_{2}u^\h\|_{L^{2}}
	\lesssim \|\nablah u^\h\|_{H^{2}}^\frac32\|(\partial_1b^\h,\partial_1b^3)\|_{H^2}^\frac12\|u^\h\|_{H^2}^\frac12\|(b^\h,b^3)\|_{H^2}^\frac12,\\
	K^{5}_{1}\lesssim&\|\partial_{i}^2u^\h\|_{L^{2}}\|u\|_{L^2_{x_1}L^\infty_{x_2}L^\infty_{x_2}}\| \nabla \partial_{2}b^\h\|_{L^\infty_{x_1}L^{2}_{x_2}L^{2}_{x_3}}\\
	\lesssim& \|\nablah u^\h\|_{H^{2}}\|(\nablah u^\h,\nablah u^3)\|_{H^2}^\frac12\|\partial_1b^\h\|_{H^2}^\frac12\|(u^\h,u^3)\|_{H^2}^\frac12\|b^\h\|_{H^2}^\frac12.
\end{align*}
For the remaining terms, using H\"{o}lder inequality again, we get
\begin{align*}
	K^{3}_{1}\leqq \|\partial_{i}^2u^\h\|_{L^{2}}\|\partial_{1}^2\partial_2b^\h\|_{L^{2}}\leqq \|\nablah u^\h\|_{H^{2}}\|\partial_{1}b^\h\|_{H^{2}},\quad 	K^{6}_{1}\leqq\|\nablah u^\h\|_{H^{2}}^2.
\end{align*}
Combining the above estimates, we have
\begin{align}\label{k_1}
	K_{1}\leqq&\frac{\mathrm{d}}{\mathrm{d}t}\int_{\R^3}\partial_{i}u^\h\partial_{i}\partial_{2}b^\h\mathrm{d}x+\|\nablah u^\h\|_{H^{2}}\|\partial_{1}b^\h\|_{H^{2}}+\|\nablah u^\h\|_{H^{2}}^2\nonumber\\
	&+C\|\nabla_{\mathrm{h}}u^\h\|_{H^{2}}\left(\|(\partial_2b^\mathrm{h},\partial_2b^3)\|_{H^{1}}\|u^\h\|_{H^{2}}+\|(\nabla_{\mathrm{h}}u^\h,\nabla_{\mathrm{h}}u^3)\|_{H^{2}}\|b^\h\|_{H^{2}}\right)\nonumber\\
	&+C\|\nablah u^\h\|_{H^{2}}^\frac32\|(\partial_1b^\h,\partial_1b^3)\|_{H^2}^\frac12\|u^\h\|_{H^2}^\frac12\|(b^\h,b^3)\|_{H^2}^\frac12\nonumber\\
	&+C\|\nablah u^\h\|_{H^{2}}\|(\nablah u^\h,\nablah u^3)\|_{H^2}^\frac12\|\partial_1b^\h\|_{H^2}^\frac12\|(u^\h,u^3)\|_{H^2}^\frac12\|b^\h\|_{H^2}^\frac12.
\end{align}
As a consequence, applying Young's inequality, $K_{1}$ can be dominated as
\begin{align}\label{K1}
	K_{1}\leqq&\frac{\mathrm{d}}{\mathrm{d}t}\int_{\R^3}\partial_{i}u^\h\partial_{i}\partial_{2}b^\h\mathrm{d}x+(\frac32+\frac{1}{100})\|\nablah u^\h\|_{H^{2}}^2+\frac12\|\partial_{1}b^\h\|_{H^{2}}^2\nonumber\\
	&+C\left(\|(\partial_1b^\h,\partial_1b^3)\|_{H^2}^2+\|(\partial_2b^\mathrm{h},\partial_2b^3)\|_{H^{1}}^2\right)\|u^\h\|_{H^{2}}^2+C\|\partial_1b^\h\|_{H^2}^2\|u^3\|_{H^2}^2\nonumber\\
	&+C\|(\nabla_{\mathrm{h}}u^\h,\nabla_{\mathrm{h}}u^3)\|_{H^{2}}^2\|b^\h\|_{H^{2}}^2+C\|\nablah u^\h\|_{H^{2}}^2\|b^3\|_{H^2}^2.
\end{align}
In view of $K_1$, for the second term of $K_6$, we find
\begin{align*}
	\int_{\R^3}\partial_{i}^2&u^3\partial_{t}\partial_{2}b^3\mathrm{d}x=\int_{\R^3} \partial_{i}^2u^3\partial_{2}\left(b\cdot\nabla u^3+\partial^{2}_{1}b^3-u\cdot\nabla b^3+\partial_{2}u^3\right)\mathrm{d}x\\&=
	\int_{\R^3} \partial_{i}^2u^3\Big(\partial_{2}b^{\mathrm{h}}\cdot\nabla_{\mathrm{h}} u^3+\partial_{2}b^3\partial_{3}u^3+b^{\mathrm{h}}\cdot\nabla_{\mathrm{h}} \partial_{2}u^3+b^3\partial_{3}\partial_{2}u^3+\partial^{2}_{1}\partial_{2}b^3\\&\quad-\partial_{2}u^\mathrm{h}\cdot\nabla_{\mathrm{h}} b^3-\partial_{2}u^3\partial_{3}b^3-u^\mathrm{h}\cdot\nabla_{\mathrm{h}} \partial_{2}b^3-u^3\partial_{3}\partial_{2}b^3+\partial_{2}^2u^3\Big)\mathrm{d}x\\&=
	\int_{\R^3} \partial_{i}^2u^3\Big(\partial_{2}b^{\mathrm{h}}\cdot\nabla_{\mathrm{h}} u^3-\partial_{2}b^3\mathrm{div_h}u^\h+b^{\mathrm{h}}\cdot\nabla_{\mathrm{h}} \partial_{2}u^3-b^3\partial_{2}\mathrm{div_h}u^\h+\partial^{2}_{1}\partial_{2}b^3\\&\quad-\partial_{2}u^\mathrm{h}\cdot\nabla_{\mathrm{h}} b^3+\partial_{2}u^3\mathrm{div_h}b^\h-u^\mathrm{h}\cdot\nabla_{\mathrm{h}} \partial_{2}b^3+u^3\partial_{2}\mathrm{div_h}b^\h+\partial_{2}^2u^3\Big)\mathrm{d}x
	\\& \eqdefa
	K^{1}_{6}+\cdots+ K^{10}_{6}.
\end{align*}
Since all the terms above possess sufficient horizontal derivatives, they can be simply estimated as
\begin{align*}
	K^{1}_{6}\lesssim& \|\partial_{i}^2u^3\|_{L^2}\|\partial_{2}b^{\mathrm{h}}\|_{L^2}\|\nabla_{\mathrm{h}} u^3\|_{L^\infty},
	&&K^{2}_{6}\lesssim \|\partial_{i}^2u^3\|_{L^2}\|\partial_{2}b^3\|_{L^2}\|\mathrm{div_h}u^\h\|_{L^\infty},\\
	K^{3}_{6}\lesssim& \|\partial_{i}^2u^3\|_{L^2}\|b^{\mathrm{h}}\|_{L^\infty}\|\nabla_{\mathrm{h}}\partial_2 u^3\|_{L^2},
	&&K^{4}_{6}\lesssim \|\partial_{i}^2u^3\|_{L^2}\|b^3\|_{L^\infty}\|\partial_{2}\mathrm{div_h}u^\h \|_{L^2},\\
	K^{6}_{6}\lesssim& \|\partial_{i}^2u^3\|_{L^2}\|\partial_{2}u^{\mathrm{h}}\|_{L^\infty}\|\nabla_{\mathrm{h}} b^3\|_{L^2},
	&&K^{7}_{6}\lesssim \|\partial_{i}^2u^3\|_{L^2}\|\partial_{2}u^3\|_{L^\infty}\|\mathrm{div_h}b^\h\|_{L^2},\\
	K^{8}_{6}\lesssim& \|\partial_{i}^2u^3\|_{L^2}\|u^{\mathrm{h}}\|_{L^\infty}\|\nabla_{\mathrm{h}} \partial_{2}b^3\|_{L^2},
\end{align*}
and
\begin{align*}
	K^{5}_{6}\leqq\|\partial_{i}^2u^3\|_{L^2}\|\partial_1^2\partial_{2}b^3\|_{L^2}\leqq\|\nablah u^3\|_{H^2}\|\partial_1b^3\|_{H^2} ,\quad
	K^{10}_{6}\leqq \|\partial_{i}^2u^3\|_{L^2}\|\partial_2^2u^3\|_{L^2}\leqq \|\nablah u^3\|_{H^2}^2.
\end{align*}
Using Sobolev embedding theorem or \eqref{L^infty_1D}, we immediately obtain
\begin{align*}
	K^{1}_{6}+K^{3}_{6}+K^{7}_{6}\lesssim \|\nablah u^3\|_{H^2}^2\|b^\h\|_{H^2},\quad 	K^{2}_{6}+K^{4}_{6}+K^{6}_{6}\lesssim \|\nablah u^3\|_{H^2}\|\nablah u^\h\|_{H^2}\|b^3\|_{H^2},
\end{align*}
and
\begin{align*}
	K^{8}_{6}\lesssim \|\nablah u^3\|_{H^2}\| \partial_{2}b^3\|_{H^1}\|u^\h\|_{H^2}.
\end{align*}
For $K^{9}_{6}$, by \eqref{L^infty_1D}, we achieve
\begin{align*}
	K^{9}_{6}\lesssim\|\partial_{i}^2u^3\|_{L^2}\|u^3\|_{L^2_{x_1}L^\infty_{x_2}L^\infty_{x_3}}\| \partial_{2}\mathrm{div_h}b^\h\|_{L^\infty_{x_1}L^{2}_{x_2}L^{2}_{x_3}}
	\lesssim\|\nablah u^3\|_{H^{2}}^\frac32\|\partial_1b^\h\|_{H^2}^\frac12\|u^3\|_{H^2}^\frac12\|b^\h\|_{H^2}^\frac12.
\end{align*}
To a conclusion, applying Young inequality, $K_6$ can be dominated as
\begin{align}\label{K6}
	K_{6}\leqq&\frac{\mathrm{d}}{\mathrm{d}t}\int_{\R^3}\partial_{i}u^3\partial_{i}\partial_{2}b^3\mathrm{d}x+\|\nablah u^3\|_{H^2}\|\partial_1b^3\|_{H^2} + \|\nablah u^3\|_{H^2}^2\nonumber\\
	&+ C\|\nablah u^3\|_{H^2}\left(\|\nablah u^3\|_{H^2}\|b^\h\|_{H^2}+\|\nablah u^\h\|_{H^2}\|b^3\|_{H^2}+\| \partial_{2}b^3\|_{H^1}\|u^\h\|_{H^2}\right)\nonumber\\
	&+C\|\nablah u^3\|_{H^{2}}^\frac32\|\partial_1b^\h\|_{H^2}^\frac12\|u^3\|_{H^2}^\frac12\|b^\h\|_{H^2}^\frac12\nonumber\\
	\leqq&\frac{\mathrm{d}}{\mathrm{d}t}\int_{\R^3}\partial_{i}u^3\partial_{i}\partial_{2}b^3\mathrm{d}x+(\frac32+\frac{1}{10})\|\nablah u^3\|_{H^2}^2+\frac12\|\partial_1b^3\|_{H^2}+C\| \partial_{2}b^3\|_{H^1}^2\|u^\h\|_{H^2}^2\nonumber\\
	&+ C\|\nablah u^3\|_{H^2}^2\|b^\h\|_{H^2}^2+C\|\nablah u^\h\|_{H^2}^2\|b^3\|_{H^2}^2+C\|\partial_1b^\h\|_{H^2}^2\|u^3\|_{H^2}^2.
\end{align}
For the term $K_{2}$ and $K_4$, thanks to $\mathrm{div}b=0$ and \eqref{L^infty_1D}, it leads
\begin{align}\label{k_2}
	K_2=&\left(\partial_{i}u\cdot \nabla u^\h+u\cdot \nabla\partial_{i} u^\h|\partial_{i}\partial_{2}b^\h\right)_{L^2}\nonumber\\
	\lesssim&\|\partial_{i}\partial_{2}b^\h\|_{L^2}\left(\|\partial_{i}u\|_{L^2_{x_1}L^\infty_{x_2}L^2_{x_3}}\| \nabla u^\h\|_{L^\infty_{x_1}L^{2}_{x_2}L^{\infty}_{x_3}}+\|u\|_{L^2_{x_1}L^\infty_{x_2}L^\infty_{x_3}}\| \nabla\partial_{i} u^\h\|_{L^\infty_{x_1}L^{2}_{x_2}L^{2}_{x_3}}\right)\nonumber\\
	\lesssim&\|\partial_{2}b^\h\|_{H^1}\|(\nablah u^\h,\nablah u^3)\|_{H^2}^\frac12\|\nablah u^\h\|_{H^2}^\frac12\|( u^\h, u^3)\|_{H^2}^\frac12\| u^\h \|_{H^2}^\frac12,\\\label{k_4}
	K_4=&\left(\partial_{i}b\cdot \nabla u^\h+b\cdot \nabla\partial_{i} u^\h|\partial_{i}\partial_{2}b^\h\right)_{L^2}\nonumber\\
	\lesssim&\|\partial_{i}\partial_{2}b^\h\|_{L^2}\left(\|\partial_{i}b\|_{L^2_{x_1}L^\infty_{x_2}L^2_{x_3}}\| \nabla b^\h\|_{L^\infty_{x_1}L^{2}_{x_2}L^{\infty}_{x_3}}+\|b\|_{L^2_{x_1}L^\infty_{x_2}L^\infty_{x_3}}\| \nabla\partial_{i} b^\h\|_{L^\infty_{x_1}L^{2}_{x_2}L^{2}_{x_3}}\right)\nonumber\\
	\lesssim&\|\partial_{2}b^\h\|_{H^1}\|(\partial_2 b^\h,\partial_2 b^3)\|_{H^1}^\frac12\|\partial_1 b^\h\|_{H^2}^\frac12\|( b^\h, b^3)\|_{H^2}^\frac12\| b^\h \|_{H^2}^\frac12.
\end{align}
Therefore, we get
\begin{align}\label{K2}
	K_2	\leqq&\frac{1}{100}\|\partial_{2}b^\h\|_{H^1}^2+\|\nablah u^3\|_{H^2}^2\| u^\h \|_{H^2}^2+\|\nablah u^\h\|_{H^2}^2\|( u^\h, u^3)\|_{H^2}^2,\\\label{K4}
	K_4\leqq&\frac{1}{100}\|\partial_{2}b^\h\|_{H^1}^2+\|(\partial_2 b^\h,\partial_2 b^3)\|_{H^1}^2\| b^\h \|_{H^2}^2+\|\partial_1 b^\h\|_{H^2}^2\|( b^\h, b^3)\|_{H^2}^2.
\end{align}
By utilizing $\mathrm{div}u=0$, H\"older's inequality, the term $K_{7}$ could be bounded
\begin{align}\label{K7}
	K_{7}=&\sum_{i=1}^{2}
	\left(\partial_{i}u^\mathrm{h}\cdot\nabla_{\mathrm{h}}u^3-\partial_{i}u^3\mathrm{div_h}u^\mathrm{h}+u^\mathrm{h}\cdot\nabla_{\mathrm{h}}\partial_{i}u^3-u^3\partial_{i}\mathrm{div_h}u^\mathrm{h}|\partial_{i}\partial_{2}b^3\right)_{L^2}\nonumber\\
	&+\left(\partial_{3}u^\mathrm{h}\cdot\nabla_{\mathrm{h}}u^3+(\mathrm{div_h}u^\h)^2-u^\mathrm{h}\cdot\nabla_{\mathrm{h}}\mathrm{div_h}u^\h+u^3\partial_{3}\mathrm{div_h}u^\h|\partial_{3}\partial_{2}b^3\right)_{L^2}\nonumber\\
	\lesssim &\|\nabla\partial_{2}b^3\|_{L^2}\left(\|\nabla_\h u^\mathrm{h}\|_{L^\infty}\|\nabla_{\mathrm{h}}u^3\|_{L^2}+\|u^\mathrm{h}\|_{L^\infty}\|\nabla_{\mathrm{h}}^2u^3\|_{L^2}+\|u^3\|_{L^\infty}\|\nabla_\h\nabla u^\h\|_{L^2}\right)\nonumber\\
	&+\|\partial_3\partial_{2}b^3\|_{L^2}\left(\|\nabla_\h u^\mathrm{h}\|_{L^\infty}\|\nabla_{\mathrm{h}}u^\h\|_{L^2}+\|u^\mathrm{h}\|_{L^\infty}\|\nabla_{\mathrm{h}}^2u^\h\|_{L^2}\right)\nonumber\\
	\lesssim&\|\partial_{2}b^3\|_{H^1}\left(\|\nabla_\h u^\mathrm{h}\|_{H^2}\|(u^\h,u^3)\|_{H^2}+\|\nabla_{\mathrm{h}}u^3\|_{H^2}\|u^\mathrm{h}\|_{H^2}\right)\nonumber\\
	\leqq& \frac{1}{100}\|\partial_{2}b^3\|_{H^1}^2+C\|\nabla_\h u^\mathrm{h}\|_{H^2}^2\|(u^\h,u^3)\|_{H^2}^2+C\|\nabla_{\mathrm{h}}u^3\|_{H^2}^2\|u^\mathrm{h}\|_{H^2}^2.
\end{align}
Along the same lines, $K_9$ can be handled as
\begin{align}\label{K9}
	K_{9}=&\sum_{i=1}^{2}
	\left(\partial_{i}b^\mathrm{h}\cdot\nabla_{\mathrm{h}}b^3-\partial_{i}b^3\mathrm{div_h}b^\mathrm{h}+b^\mathrm{h}\cdot\nabla_{\mathrm{h}}\partial_{i}b^3-b^3\partial_{i}\mathrm{div_h}b^\mathrm{h}|\partial_{i}\partial_{2}b^3\right)_{L^2}\nonumber\\
	&-\left(\partial_{3}b^\mathrm{h}\cdot\nabla_{\mathrm{h}}b^3+(\mathrm{div_h}b^\h)^2-b^\mathrm{h}\cdot\nabla_{\mathrm{h}}\mathrm{div_h}b^\h+b^3\partial_{3}\mathrm{div_h}b^\h|\partial_3\partial_{2}b^3\right)_{L^2}\nonumber\\
	\lesssim &\|\nabla\partial_{2}b^3\|_{L^2}\left(\|\nabla_\h b^\mathrm{h}\|_{L^\infty_{x_1}L^2_{x_2}L^\infty_{x_3}}\|\nabla_{\mathrm{h}}b^3\|_{L^2_{x_1}L^\infty_{x_2}L^2_{x_3}}+\|b^\mathrm{h}\|_{L^\infty}\|\nabla_{\mathrm{h}}^2b^3\|_{L^2}\right)\nonumber\\
	&+\|\nabla\partial_{2}b^3\|_{L^2}\|b^3\|_{L^2_{x_1}L^\infty_{x_2}L^\infty_{x_3}}\|\nabla_\h\nabla b^\h\|_{L^\infty_{x_1}L^2_{x_2}L^2_{x_3}}\nonumber\\
	&+\|\partial_3\partial_{2}b^3\|_{L^2}\left(\|\nabla_\h b^\mathrm{h}\|_{L^\infty_{x_1}L^2_{x_2}L^\infty_{x_3}}\|\nabla_{\mathrm{h}}b^\h\|_{L^2_{x_1}L^\infty_{x_2}L^2_{x_3}}+\|b^\mathrm{h}\|_{L^\infty}\|\nabla_{\mathrm{h}}^2b^\h\|_{L^2}\right)\nonumber\\
	\lesssim&\|\partial_{2}b^3\|_{H^1}\left(\|\partial_1 b^\mathrm{h}\|_{H^2}^\frac12\|\partial_2 b^3\|_{H^1}^\frac12\|b^\h\|_{H^2}^\frac12\|b^3\|_{H^2}^\frac12+\|(\partial_1b^3,\partial_2b^3)\|_{H^1}\|b^\mathrm{h}\|_{H^2}\right)\nonumber\\
	&+\|\partial_{2}b^3\|_{H^1}\left(\|\partial_1 b^\mathrm{h}\|_{H^2}^\frac12\|\partial_2 b^\h\|_{H^1}^\frac12\|b^\h\|_{H^2}+\|(\partial_1b^\h,\partial_2b^\h)\|_{H^1}\|b^\mathrm{h}\|_{H^2}\right)\nonumber\\
	\leqq& \frac{1}{100}\|\partial_{2}b^3\|_{H^1}^2+C\left(\|(\partial_1 b^\mathrm{h},\partial_1 b^3)\|_{H^2}^2+\|(\partial_2 b^\h,\partial_2 b^3)\|_{H^1}^2\right)\|b^\h\|_{H^2}^2\nonumber\\
	&+C\|\partial_1 b^\mathrm{h}\|_{H^2}^2\|b^3\|_{H^2}^2.
\end{align}
Ultimately, by H\"older's inequality, it yields
\begin{align}\label{K3}
	K_3\leqq &\|\Delta_{\mathrm{h}}\partial_{i}u^\h\|_{L^2}\|\partial_{i}\partial_{2}b^\h\|_{L^2}\leqq \frac12\|\nablah u^\h\|_{H^2}^2+\frac12\|\partial_{2}\partial_ib^\h\|_{L^2}^2,\\\label{K8}
	K_8\leqq &\|\Delta_{\mathrm{h}}\partial_{i}u^3\|_{L^2}\|\partial_{i}\partial_{2}b^3\|_{L^2}\leqq \frac12\|\nablah u^3\|_{H^2}^2+\frac12\|\partial_{2}\partial_ib^3\|_{L^2}^2,
\end{align}
and by integration by parts and divergence free condition of $b$, it implies $K_5+K_{10} = 0$. As a consequence, inserting \eqref{K1}, \eqref{K6} and \eqref{K2}-\eqref{K8} into \eqref{p2b_enH1}, we deduce that
\begin{align}\label{Esti.p2bH1}
	\|(\partial_i\partial_2b^\h,\partial_i\partial_2b^3)\|_{L^2}^2\leqq &\frac{\mathrm{d}}{\mathrm{d}t}\int_{\R^3}\partial_{i}u^\h\partial_{i}\partial_{2}b^\h+\partial_{i}u^3\partial_{i}\partial_{2}b^3\mathrm{d}x+\frac{201}{100}\|\nablah u^3\|_{H^2}^2+\frac12\|\partial_1b^3\|_{H^2}^2\nonumber\\
	&+\frac{1}{50}\|(\partial_{2}b^\h,\partial_{2}b^3)\|_{H^1}^2+\frac12\|(\partial_2\partial_{i}b^\h,\partial_2\partial_{i}b^3)\|_{L^2}^2\nonumber\\
	&+C\left(\|(\nablah u^3,\partial_1b^3)\|_{H^2}^2+\|(\partial_2b^\mathrm{h},\partial_2b^3)\|_{H^{1}}^2\right)\|(u^\h,b^\h)\|_{H^{2}}^2\nonumber\\
	&+C\|(\nablah u^\h,\partial_1b^\h)\|_{H^{2}}^2\|(u^3,b^3)\|_{H^2}^2\nonumber\\
	&+C\|(\nablah u^\h,\partial_{1}b^\h)\|_{H^{2}}^2\left(1+\|( u^\h, b^\h)\|_{H^2}^2\right).
\end{align}
Combining \eqref{Esti.p2bL2} and \eqref{Esti.p2bH1}, we have \eqref{p2bH1}. This completes the proof of Proposition \ref{v}.

\section{The proof of Proposition \ref{h}}\label{4}
In this section, we shall present the proof of Proposition \ref{h}. We divide it into two subsections. The first
one is $H^2$ energy estimates of $(u^\h, b^\h)$ and the second one is $H^1$ energy estimates of $\partial_2b^\h$.

\subsection{The proof of \eqref{uhbhH2}}
In the same manner as before, we only need to  estimate  $\|(u^\mathrm{h},b^\mathrm{h})\|_{L^2}^2$ and $\sum_{i=1}^{3}\|\partial_{i}^2(u^\mathrm{h},b^\mathrm{h})\|_{L^2}^{2}$, and we split it into  two parts. Initially, by the standard $L^2$ energy method, the system \eqref{eq.MHD} implies
\begin{align}\label{uhbhl2}
	\frac12\frac{\mathrm{d}}{\dt}  \|(u^\mathrm{h},b^\mathrm{h})\|_{L^2}^2 +\|(\nabla_\mathrm{h}u^\mathrm{h},\partial_3 b^\mathrm{h})\|_{ L^2}^2=-(\nablah p|u^\h)_{L^2}.
\end{align}
Due to \eqref{p_same}, combining \eqref{A1}, \eqref{A2(h)}, \eqref{A3}, \eqref{A4}, \eqref{A5(h)}, \eqref{A6(h)} and \eqref{uhbhl2}, it holds that
\begin{align}\label{Esti.h.L2}
	\frac12&\frac{\mathrm{d}}{\dt}  \|(u^\h,b^\h)\|_{L^2}^2 +\|(\nabla_\h u^\h,\partial_1b^\h)\|_{ L^2}^2\nonumber\\
	&\leqq \frac{1}{20}\|(\nablah u^\h,\partial_1b^\h)\|_{H^2}^2 +\frac{1}{100}\|\partial_2b^\h\|_{H^1}^2+C\|\nablah u^\h\|_{H^{1}}^2\|u^\h\|_{L^2}^2\nonumber\\&+C\left(\|\partial_1b^\h\|_{H^2}^2+\|\partial_2b^\h\|_{H^{1}}^2+\|\partial_2b^3\|_{H^1}^2\|b^3\|_{L^{2}}^2\right)\|b^\h\|_{H^{2}}^2\nonumber\\
	&+C\|(\nablah u^\h,\nablah u^3)\|_{H^1}^2\|u^3\|_{H^{1}}^2\| u^\h\|_{H^1}^2.
\end{align}

Then let us handle  $\sum_{i=1}^{3}\|\partial_{i}^2(u^\h,b^\h)\|_{L^2}^{2}$. We get, by applying $\partial_{i}^2$ to the equations of  $u^\mathrm{h}$ and $b^\mathrm{h}$ in systems \eqref{eq.MHD}, and then taking the $L^2$ inner product to the new equations with $\partial_{i}^2u^\mathrm{h}$ and $\partial_{i}^2b^\mathrm{h}$, respectively, that
\begin{align}\label{en_H2_h}
	\frac12\frac{\mathrm{d}}{\mathrm{d}t}\sum_{i=1}^{3}\|\partial_{i}^2(u^\h,b^\h)\|_{L^2}^2+\sum_{i=1}^{3}\|\partial_{i}^2(\nabla_{\mathrm{h}}u^\h,\partial_{1}b^\h)\|_{L^2}^2&=J_{1}+\cdots+J_{8},
\end{align}
where
\begin{align*}
	J_{1}
	=&-\left(\partial_{i}^2u^j\partial_ju^\h+2\partial_{i}u^j\partial_j \partial_{i}u^\h|\partial_{i}^2u^\h\right)_{L^2},&&J_{2}=\left(\partial_{i}^2b^j\partial_jb^\h+2\partial_{i}b^j\partial_j \partial_{i}b^\h|\partial_{i}^2u^\h\right)_{L^2},\\
	J_{3}
	=&-\left(\partial_{i}^2u^j\partial_jb^\h+2\partial_{i}u^j\partial_j \partial_{i}b^\h|\partial_{i}^2b^\h\right)_{L^2},&&J_{4}=\left(\partial_{i}^2b^j\partial_ju^\h+2\partial_{i}b^j\partial_j \partial_{i}u^\h|\partial_{i}^2b^\h\right)_{L^2},\\
	J_{5}
	=&\left(-u\cdot\nabla \partial_{i}^2u^\h+b\cdot\nabla \partial_{i}^2b^\h|\partial_{i}^2u^\h\right)_{L^2},&&J_{6}=\left(-u\cdot\nabla \partial_{i}^2b^\h+b\cdot\nabla \partial_{i}^2u^\h|\partial_{i}^2b^\h\right)_{L^2},\\
	J_{7}=&\left( \partial_{i}^2\partial_{2}b^\h|\partial_{i}^2u^\h \right)_{L^2}+\left( \partial_{i}^2\partial_{2}u^\h|\partial_{i}^2b^\h \right)_{L^2} ,&&J_{8}=-\left(\partial_{i}^2\partial_{3}p|\partial_{i}^2u^\h\right)_{L^2}\quad \mathrm{with}\quad i,j=1,2,3.
\end{align*}
In view of \eqref{I1rewrite}, for $J_1$ and $J_2$, due to $\mathrm{div}u=\mathrm{div}b=0$, we reform them as follow
\begin{align*}
	J_{1}\leqq& C\int_{\R^3}|\nablah\nabla u\nabla u^\h \nablah^2 u^\h|+|\nablah\partial_3u^\h\partial_3 u^\h \partial_3^2 u^\h|+|\partial_3^2u^\h\nablah u^\h \partial_3^2 u^\h|\dx\nonumber\\
	&+C\int_{\R^3}|\nablah\partial_3 u^\h\nablah u^3 \nablah^2 u ^\h|\dx\eqdefa J_{11}+\cdots+J_{14},\\
	J_{2}\leqq& C\int_{\R^3}|\nablah\nabla b^\h\nablah b\nablah^2 u^\h|+|\nablah\partial_3b^\h\partial_3 b^\h \partial_3^2 u^\h|+|\partial_3^2b^\h\nablah b^\h \partial_3^2 u^\h|\dx\nonumber\\
	&+C\int_{\R^3}|\nablah^2 b^3\partial_3 b^\h \nablah^2 u ^\h|\dx\eqdefa J_{21}+\cdots+J_{24}.
\end{align*}
For $J_{11}$ and $J_{21}$,
getting the utmost of \eqref{er} with $f=u$ (resp. $b^\h$), $g= u^\mathrm{h}$ (resp. $b$), $h=\nabla_\mathrm{h} u^\mathrm{h}$ (resp. $\nablah u^\h$), it holds
\begin{align*}
	J_{11}\lesssim& \|\partial_1u\|_{H^2}^\frac12\|\nabla_\h u\|_{H^1}^\frac12\|\nabla u^\h\|_{H^1}\|\nabla^2_\mathrm{h} u^\mathrm{h}\|_{H^1}^\frac12\|\nabla_\mathrm{h} u^\mathrm{h}\|_{H^1}^\frac12\\
	\lesssim&\|\nabla_\mathrm{h} u^\mathrm{h}\|_{H^2}\|\nabla_\mathrm{h} u^3\|_{H^2}\| u^\mathrm{h}\|_{H^2}+\|\nabla_\mathrm{h} u^\mathrm{h}\|_{H^2}^2\| u^\mathrm{h}\|_{H^2},\\
	J_{21}\lesssim &\|\partial_1b^\h\|_{H^2}^\frac12\|\nabla_\h b^\h\|_{H^1}^\frac12\|\nablah b\|_{H^1}\|\nabla^2_\mathrm{h} u^\mathrm{h}\|_{H^1}^\frac12\|\nabla_\mathrm{h} u^\mathrm{h}\|_{H^1}^\frac12\\
	\lesssim&\|\nabla_\mathrm{h} u^\mathrm{h}\|_{H^2}^\frac12\|\partial_1b^\h\|_{H^2}^\frac12\left(\|(\partial_1 b^\h,\partial_1b^3)\|_{H^2}+\|(\partial_2 b^\h,\partial_2b^3)\|_{H^1}\right)\|u^\mathrm{h}\|_{H^2}^\frac12\| b^\h\|_{H^2}^\frac12.
\end{align*}
For the remaining terms, using H\"older's inequality and \eqref{L^infty_1D}, we get
\begin{align*}
	J_{12}\lesssim &\|\nablah\partial_3u^\h\|_{L^2_{x_1}L^2_{x_2}L^\infty_{x_3}}\|\partial_3 u^\h  \|_{L^\infty_{x_1}L^2_{x_2}L^2_{x_3}}\|\partial_3^2 u^\h \|_{L^2_{x_1}L^\infty_{x_2}L^2_{x_3}}
	\lesssim\|\nabla_\mathrm{h} u^\mathrm{h}\|_{H^2}^2\| u^\mathrm{h}\|_{H^2} ,\\
	J_{13}\lesssim &\|\partial_3^2u^\h\|_{L^\infty_{x_1}L^2_{x_2}L^2_{x_3}}\|\nablah u^\h\|_{L^2_{x_1}L^2_{x_2}L^\infty_{x_3}}\|\partial_3^2 u^\h \|_{L^2_{x_1}L^\infty_{x_2}L^2_{x_3}}
	\lesssim\|\nabla_\mathrm{h} u^\mathrm{h}\|_{H^2}^2\| u^\mathrm{h}\|_{H^2},\\
	J_{14}\lesssim&\|\nablah\partial_3 u^\h \|_{L^2}\|\nablah u^3\|_{L^\infty}\|\nablah^2 u ^\h\|_{L^2}\lesssim\|\nabla_\mathrm{h} u^\mathrm{h}\|_{H^2}\|\nabla_\mathrm{h} u^3\|_{H^2}\| u^\mathrm{h}\|_{H^2},
\end{align*}
and
\begin{align*}
	J_{22}\lesssim &\|\nablah\partial_3b^\h\|_{L^2}\|\partial_3 b^\h  \|_{L^\infty_{x_1}L^2_{x_2}L^\infty_{x_3}}\|\partial_3^2 u^\h \|_{L^2_{x_1}L^\infty_{x_2}L^2_{x_3}}\\
	\lesssim&\|\nabla_\mathrm{h} u^\mathrm{h}\|_{H^2}^\frac12\|\partial_1b^\h\|_{H^2}^\frac12\left(\|\partial_1 b^\h\|_{H^2}+\|\partial_2 b^\h\|_{H^1}\right)\|u^\mathrm{h}\|_{H^2}^\frac12\| b^\h\|_{H^2}^\frac12 ,\\
	J_{23}\lesssim &\|\partial_3^2b^\h\|_{L^\infty_{x_1}L^2_{x_2}L^2_{x_3}}\|\nablah b^\h\|_{L^2_{x_1}L^2_{x_2}L^\infty_{x_3}}\|\partial_3^2 u^\h \|_{L^2_{x_1}L^\infty_{x_2}L^2_{x_3}}	\\
	\lesssim&\|\nabla_\mathrm{h} u^\mathrm{h}\|_{H^2}^\frac12\|\partial_1b^\h\|_{H^2}^\frac12\left(\|\partial_1 b^\h\|_{H^2}+\|\partial_2 b^\h\|_{H^1}\right)\|u^\mathrm{h}\|_{H^2}^\frac12\| b^\h\|_{H^2}^\frac12 ,\\
	J_{24}\lesssim&\|\nablah^2 b^3 \|_{L^2}\|\partial_3 b^\h\|_{L^\infty_{x_1}L^2_{x_2}L^\infty_{x_3}}\|\nablah^2 u ^\h\|_{L^2_{x_1}L^\infty_{x_2}L^2_{x_3}}\\
	\lesssim&\|\nabla_\mathrm{h} u^\mathrm{h}\|_{H^2}^\frac12\|\partial_1b^\h\|_{H^2}^\frac12\left(\|\partial_1 b^3\|_{H^2}+\|\partial_2 b^3\|_{H^1}\right)\|u^\mathrm{h}\|_{H^2}^\frac12\| b^\h\|_{H^2}^\frac12 .
\end{align*}
Then we obtain
\begin{align}\label{J_1}
	J_1\lesssim& \|\nabla_\mathrm{h} u^\mathrm{h}\|_{H^2}\|\nabla_\mathrm{h} u^3\|_{H^2}\| u^\mathrm{h}\|_{H^2}+\|\nabla_\mathrm{h} u^\mathrm{h}\|_{H^2}^2\| u^\mathrm{h}\|_{H^2}\nonumber\\
	\leqq &\frac{1}{100}\|\nabla_\mathrm{h} u^\mathrm{h}\|_{H^2}^2+\|(\nabla_\mathrm{h} u^\mathrm{h},\nabla_\mathrm{h} u^{3})\|_{H^2}^2\| u^\mathrm{h}\|_{H^2}^2,
\end{align}
and
\begin{align}\label{J_2}
	J_2\lesssim&\|\nabla_\mathrm{h} u^\mathrm{h}\|_{H^2}^\frac12\|\partial_1b^\h\|_{H^2}^\frac12\left(\|(\partial_1 b^\h,\partial_1b^3)\|_{H^2}+\|(\partial_2 b^\h,\partial_2b^3)\|_{H^1}\right)\|u^\mathrm{h}\|_{H^2}^\frac12\| b^\h\|_{H^2}^\frac12\nonumber\\
	\leqq &\frac{1}{100}\|(\nabla_\mathrm{h} u^\mathrm{h},\partial_1b^\h)\|_{H^2}^2+\left(\|(\partial_1 b^\h,\partial_1b^3)\|_{H^2}^2+\|(\partial_2 b^\h,\partial_2b^3)\|_{H^1}^2\right)\| (u^\mathrm{h},b^\h)\|_{H^2}^2.
\end{align}
For $J_{3}$ and $J_{4}$, to facilitate the analysis, we divide the discussion for indices $i$ into two cases: $i=1,2$ and $i=3$. Namely, we denote
\begin{align}\label{J3,4deco}
	J_{3}+J_{4} \eqdefa J_{31}+J_{32} ,
\end{align}
where
\begin{align*}
	J_{31}=&\sum_{i=1}^2\left(-\partial_{i}^2u\cdot \nabla b^\h-2\partial_{i}u\cdot \nabla  \partial_{i}b^\h+\partial_{i}^2b\cdot \nabla u^\h+2\partial_{i}b\cdot \nabla  \partial_{i}u^\h|\partial_{i}^2b^\h\right)_{L^2},\\
	J_{32}=&\sum_{j=1}^2\left(-\partial_{3}^2u^j\partial_jb^\h-2\partial_{3}u^j\partial_j \partial_{3}b^\h+\partial_{3}^2b^j\partial_ju^\h+2\partial_{3}b^j\partial_j \partial_{3}u^\h|\partial_{3}^2b^\h\right)_{L^2}\\
	&+\left(-\partial_{3}^2u^3\partial_3b^\h-2\partial_{3}u^3\partial_3^2b^\h+\partial_{3}^2b^3\partial_3u^\h+2\partial_{3}b^3\partial_3^2u^\h|\partial_{3}^2b^\h\right)_{L^2}.
\end{align*}
For $J_{31}$, it can be bounded easily
\begin{align*}
	J_{31}\lesssim& \left(\|\nabla_\h^2 u\|_{L^2_{x_1}L^2_{x_2}L^\infty_{x_3}}\| \nabla b^\h \|_{L^\infty_{x_1}L^\infty_{x_2}L^2_{x_3}}+\|\nabla_\h u\|_{L^\infty}\|\nabla_\h\nabla b^\h\|_{L^2} \right)\|\nabla_\h^2 b^\h \| _{L^2}  \nonumber\\
	&+\left(\|\nabla_\h^2 b\|_{L^\infty_{x_1}L^2_{x_2}L^2_{x_3}}\| \nabla u^\h \|_{L^2_{x_1}L^\infty_{x_2}L^\infty_{x_3}}+\|\nabla_\h b\|_{L^\infty_{x_1}L^\infty_{x_2}L^2_{x_3}}\|\nabla_\h\nabla u^\h\|_{L^2_{x_1}L^2_{x_2}L^\infty_{x_3}} \right)\|\nabla_\h^2 b^\h \| _{L^2}\nonumber\\
	\lesssim&\|(\nabla_\h u^\h,\nabla_\h u^3)\|_{H^2}\left(\|\partial_1 b^\h\|_{H^2}+\|\partial_2 b^\h\|_{H^1}\right)\|b^\h\|_{H^2}\nonumber\\
	&+\|\nabla_\h u^\h\|_{H^2}\left(\|(\partial_1 b^\h,\partial_1b^3)\|_{H^2}+\|(\partial_2 b^\h,\partial_2b^3)\|_{H^1}\right)\|b^\h\|_{H^2}.
\end{align*}
Using Young's inequality, we see that
\begin{align}\label{J31}
	J_{31}\leqq& \frac{1}{100}\left(\|(\nabla_\h u^\h,\partial_1 b^\h)\|_{H^2}^2+\|\partial_2 b^\h\|_{H^1}^2\right)+C\|(\nabla_\h u^\h,\nabla_\h u^3)\|_{H^2}^2\|b^\h\|_{H^2}^2\nonumber\\
	&+C\left(\|(\partial_1 b^\h,\partial_1b^3)\|_{H^2}^2+\|(\partial_2 b^\h,\partial_2b^3)\|_{H^1}^2\right)\|b^\h\|_{H^2}^2.
\end{align}
$J_{32}$ is difficult term owing to the loss of vertical dissipation in velocity and magnetic equation. To skip this obstacle, firstly, on account of  $\mathrm{div}u=\mathrm{div}b=0$, we subtly split $J_{32}$ as
\begin{align*}
	J_{32}=&\sum_{j=1}^2\int_{\R^3}\left(-\partial_{3}^2u^j\partial_jb^\h-2\partial_{3}u^j\partial_j \partial_{3}b^\h+\partial_{3}^2b^1\partial_1u^\h+\partial_{3}^2b^2\partial_2u^\h+2\partial_{3}b^j\partial_j \partial_{3}u^\h\right)\partial_{3}^2b^\h\dx\\
	&+\int_{\R^3}\left(\partial_{3}\mathrm{div_h}u^\h\partial_3b^\h+2(\partial_1u^1+\partial_2u^2)\partial_3^2b^\h-\partial_{3}\mathrm{div_h}b^\h\partial_3u^\h-2\mathrm{div_h}b^\h\partial_3^2u^\h\right)\partial_{3}^2b^\h\dx\\
	\leqq&C\int_{\R^3}|\partial_{3}^2u^\h\nablah b^\h\partial_{3}^2b^\h|\dx+C\int_{\R^3}|\partial_{3}u^\h\nablah \partial_{3}b^\h\partial_{3}^2b^\h|\dx+C\int_{\R^3}|\partial_{3}b^\h\nablah \partial_{3}u^\h\partial_{3}^2b^\h|\dx\\
	&+\int_{\R^3}\left(2\partial_1u^1\partial_3^2b^\h+\partial_{3}^2b^1\partial_1u^\h\right)\partial_{3}^2b^\h\dx+\int_{\R^3}\left(2\partial_2u^2\partial_3^2b^\h+\partial_{3}^2b^2\partial_2u^\h\right)\partial_{3}^2b^\h\dx\\
	\eqdefa& J_{321}+\cdots+J_{325}.
\end{align*}
Applying H\"older's inequality and \eqref{L^infty_1D} again, we deduce that
\begin{align*}
	J_{321}\lesssim&\| \partial_3^2u^\h \|_{L^2_{x_1}L^\infty_{x_2}L^2_{x_3}} \|\nablah b^\h\|_{L^2_{x_1}L^2_{x_2}L^\infty_{x_3}}\|\partial_{3}^2b^\h \|_{L^\infty_{x_1}L^2_{x_2}L^2_{x_3}} \\
	\lesssim& \|\nabla_\mathrm{h} u^\mathrm{h}\|_{H^2}^\frac12\|\partial_1b^\h\|_{H^2}^\frac12\left(\|\partial_1 b^\h\|_{H^2}+\|\partial_2 b^\h\|_{H^1}\right)\|u^\mathrm{h}\|_{H^2}^\frac12\| b^\h\|_{H^2}^\frac12,\\
	J_{322}\lesssim&  \| \partial_3u^\h \|_{L^2_{x_1}L^\infty_{x_2}L^\infty_{x_3}}\|\nablah\partial_{3}b^\h\|_{L^2}\|\partial_{3}^2b^\h \|_{L^\infty_{x_1}L^2_{x_2}L^2_{x_3}} \\
	\lesssim& \|\nabla_\mathrm{h} u^\mathrm{h}\|_{H^2}^\frac12\|\partial_1b^\h\|_{H^2}^\frac12\left(\|\partial_1 b^\h\|_{H^2}+\|\partial_2 b^\h\|_{H^1}\right)\|u^\mathrm{h}\|_{H^2}^\frac12\| b^\h\|_{H^2}^\frac12,\\
	J_{323}\lesssim& \| \partial_3b^\h \|_{L^2_{x_1}L^\infty_{x_2}L^2_{x_3}}\|\partial_{3}\nablah u^\h\|_{L^2_{x_1}L^2_{x_2}L^\infty_{x_3}}\|\partial_{3}^2b^\h \|_{L^\infty_{x_1}L^2_{x_2}L^2_{x_3}} \\
	\lesssim& \|\nablah u^\h\|_{H^2}\|\partial_1 b^\h\|_{H^2}^\frac12\|\partial_2 b^\h\|_{H^1}^\frac12\| b^\h\|_{H^2}.
\end{align*}
For $J_{324}$, by integration by parts, H\"older's inequality and \eqref{L^infty_1D}, it yields
\begin{align*}
	J_{324}=&-2\int_{\R^3}2u^1\partial_1\partial_3^2b^\h\partial_{3}^2b^\h+u^\h\partial_1\partial_3^2b^1\partial_{3}^2b^\h+u^1\partial_1\partial_3^2b^\h\partial_{3}^2b^1\dx\\
	\lesssim&\| u^\h \|_{L^2_{x_1}L^\infty_{x_2}L^\infty_{x_3}}\|\partial_1\partial_3^2b^\h\|_{L^2}\|\partial_{3}^2b^\h \|_{L^\infty_{x_1}L^2_{x_2}L^2_{x_3}}\\
	\lesssim &\|\nabla_\mathrm{h} u^\mathrm{h}\|_{H^2}^\frac12\|\partial_1b^\h\|_{H^2}^\frac32\|u^\mathrm{h}\|_{H^2}^\frac12\| b^\h\|_{H^2}^\frac12,
\end{align*}
As a result, applying Young's inequality, we obtain
\begin{align}\label{J321-324}
	\sum_{i=1}^{4}J_{32i}\lesssim &\|\nabla_\mathrm{h} u^\mathrm{h}\|_{H^2}^\frac12\|\partial_1b^\h\|_{H^2}^\frac12\left(\|\partial_1 b^\h\|_{H^2}+\|\partial_2 b^\h\|_{H^1}\right)\|u^\mathrm{h}\|_{H^2}^\frac12\| b^\h\|_{H^2}^\frac12\nonumber\\
	&+\|\nablah u^\h\|_{H^2}\|\partial_1 b^\h\|_{H^2}^\frac12\|\partial_2 b^\h\|_{H^1}^\frac12\| b^\h\|_{H^2}\nonumber\\
	\leqq& \frac{1}{100}\|(\nabla_\h u^\h,\partial_1 b^\h)\|_{H^2}^2+C\left(\|\partial_1 b^\h\|_{H^2}^2+\|\partial_2 b^\h\|_{H^1}^2\right)\|(u^\h ,b^\h)\|_{H^2}^2,
\end{align}
and for $J_{325}$, we present the estimate directly, while the detailed proof is provided in the Lemma \ref{wild},
\begin{align}\label{J325}
	J_{325}\leqq&\frac{\mathrm{d}}{\dt}\int_{\R^3} b^i\partial_3^2b^j\partial_3^2b^k-b^i\partial_3^2u^j\partial_3^2u^k\dx+\frac{3}{100}\|(\nablah u^\h,\partial_1b^\h)\|_{H^2}^2\nonumber\\
	&+C\left(\|(\nablah u^\h,\partial_1b^\h)\|_{H^2}^2+\|(\partial_2b^\h,\partial_2b^3)\|_{H^1}^2\right)\|(u^\h,b^\h)\|_{H^2}^2\nonumber\\
	&+C\left(\|\nablah u^3\|_{H^2}^2\|u^3\|_{H^2}^2+ \|(\partial_2b^\h,\partial_2b^3)\|_{H^1}^2\|(b^\h,b^3)\|_{H^2}^2\right)\|u^\h\|_{H^2}^2\nonumber\\	&+C\left(\|\partial_2b^3\|_{H^1}^2\|b^3\|_{H^2}^2+\|(\nablah u^\h,\nablah u^3)\|_{H^2}^2\|(u^\h,u^3)\|_{H^2}^2\right)\|b^\h\|_{H^2}^2.
\end{align}
Combining \eqref{J321-324} and \eqref{J325} gives the estimate of $J_{32}$, together with \eqref{J3,4deco} and \eqref{J31}, we then obtain
\begin{align}\label{J_3+J_4}
	J_3+J_4\leqq&\frac{\mathrm{d}}{\dt}\int_{\R^3} b^i\partial_3^2b^j\partial_3^2b^k-b^i\partial_3^2u^j\partial_3^2u^k\dx+\frac{1}{20}\|(\nablah u^\h,\partial_1b^\h)\|_{H^2}^2+\frac{1}{100}\|\partial_2b^\h\|_{H^1}^2\nonumber\\
	&+C\left(\|(\nablah u^\h,\partial_1b^\h)\|_{H^2}^2+\|(\nablah u^3,\partial_1b^3)\|_{H^2}^2+\|(\partial_2b^\h,\partial_2b^3)\|_{H^1}^2\right)\|(u^\h,b^\h)\|_{H^2}^2\nonumber\\
	&+C\left(\|\nablah u^3\|_{H^2}^2\|u^3\|_{H^2}^2+ \|(\partial_2b^\h,\partial_2b^3)\|_{H^1}^2\|(b^\h,b^3)\|_{H^2}^2\right)\|u^\h\|_{H^2}^2\nonumber\\	&+C\left(\|\partial_2b^3\|_{H^1}^2\|b^3\|_{H^2}^2+\|(\nablah u^\h,\nablah u^3)\|_{H^2}^2\|(u^\h,u^3)\|_{H^2}^2\right)\|b^\h\|_{H^2}^2.
\end{align}
By integration by parts and $\mathrm{div}u=\mathrm{div}b=0$, it is obvious that $J_{5}+J_{6}=J_{7}=0$, and
\begin{align*}
	J_{8}&=-\left(\partial_i^2\nabla_\mathrm{h}p| \partial_i^2u^\mathrm{h}\right) =\left( \partial_i^2p| \partial_i^2\mathrm{div_h}u^\mathrm{h}\right)=-I_8.
\end{align*}
Hence, from \eqref{ii}, by Young inequality, the term $J_8$ could be dominated as
\begin{align}\label{J_8}
	J_8\lesssim&\left(\|\nablah u^\h\|_{H^2}\|u^\h\|_{H^2}+ \|\nablah u^\h\|_{H^2}^\frac12\|\nablah u^3\|_{H^2}^\frac12\|u^\h\|_{H^2}^\frac12\|u^3\|_{H^2}^\frac12 \right)\|\nablah u^\h\|_{H^{2}}\nonumber\\
	&+\left(\|\partial_1b^\h\|_{H^2}^\frac12\|\partial_2b^\h\|_{H^1}^\frac12\|b^\h\|_{H^2}+\|\partial_1b^\h\|_{H^2}^\frac12\|\partial_2b^3\|_{H^1}^\frac12\|b^\h\|_{H^2}^\frac12\|b^3\|_{H^2}^\frac12\right)\|\nablah u^\h\|_{H^{2}}\nonumber\\
	\leqq&\frac{1}{100}\|(\nablah u^\h,\partial_1b^\h)\|_{H^2}^2+C\left(\|(\nablah u^\h,\partial_1b^\h)\|_{H^2}^2+\|\partial_2b^\h\|_{H^1}^2\right)\|(u^\h,b^\h)\|_{H^2}^2\nonumber\\
	&+C\|\nablah u^3\|_{H^2}^2\|u^3\|_{H^2}^2\|u^\h\|_{H^2}^2+C\|\partial_2b^3\|_{H^1}^2\|b^3\|_{H^2}^2\|b^\h\|_{H^2}^2.
\end{align}
Inserting \eqref{J_1}, \eqref{J_2}, \eqref{J_3+J_4}, \eqref{J_8} into \eqref{en_H2_h} and combining \eqref{Esti.h.L2}, we get \eqref{uhbhH2}.

\subsection{The proof of \eqref{p2bhH1}}
Finally, to control $H^{2}$-energy estimates of  $(u^\h,b^\h)$, we require a separate estimate for $\|\partial_2b^\h\|_{H^1}$.
By taking $L^2$ inner product to $u^\h$ equation of \eqref{eq.MHD} with $\partial_2b^\h$, it yields
\begin{align}\label{p2bh_enl2}
	\|\partial_2b^\h\|_{L^2}^2= F_1+\cdots+F_{5},
\end{align}
where $F_1-F_{5}$ are determined by \eqref{p2b_enl2}. Combining \eqref{f1deco}, \eqref{f_1} and by Young's inequality, we find
\begin{align}\label{F1h}
	F_1\leqq& \frac{\mathrm{d}}{\mathrm{d} t}\int_{\R^3} u^\h\partial_2 b^\h\mathrm{d} x+\|\nablah u^\h\|_{H^2}\|\partial_1 b^\h\|_{H^2}+\|\nablah u^\h\|_{H^2}^2\nonumber\\
	&+C\|\nablah u^\h\|_{H^2}\|(\nablah u^\h,\nablah u^3)\|_{H^2}^\frac12\|\partial_1 b^\h\|_{H^2}^\frac12\|( u^\h, u^3)\|_{H^2}^\frac12\|b^\h\|_{L^2}^\frac12\nonumber\\
	&+C\|\nablah u^\h\|_{H^2}\|(\partial_1b^\h,\partial_1b^3)\|_{H^2}^\frac12\|\nablah u^\h\|_{H^2}^\frac12\|(b^\h,b^3)\|_{H^2}^\frac12\| u^\h\|_{H^2}^\frac12\nonumber\\
	\leqq&\frac{\mathrm{d}}{\mathrm{d} t}\int_{\R^3} u^\h\partial_2 b^\h\mathrm{d} x+(\frac32+\frac{1}{100})\|\nablah u^\h\|_{H^2}^2+(\frac12+\frac{1}{100})\|\partial_1 b^\h\|_{H^2}^2\nonumber\\
	&+C\|(\nablah u^\h,\nablah u^3)\|_{H^2}^2\|( u^\h, u^3)\|_{H^2}^2\| b^\h\|_{H^2}^2+C\|(\partial_2 b^\h,\partial_2 b^3)\|_{H^1}^2\|( b^\h, b^3)\|_{H^2}^2\| u^\h\|_{H^2}^2.
\end{align}
Similarly, from \eqref{f_2} and \eqref{f_4}, we implies that
\begin{align}\label{F2h}
	F_2	\lesssim &\|\partial_2 b^\h\|_{H^1}\|\nablah u^\h\|_{H^2}^\frac12\|(\nablah u^\h,\nablah u^3)\|_{H^2}^\frac12\| u^\h\|_{H^2}^\frac12\|( u^\h, u^3)\|_{H^2}^\frac12,\nonumber\\
	\leqq &\frac{1}{100}\left(\|\nablah u^\h\|_{H^2}^2+\|\partial_2 b^\h\|_{H^1}^2\right)+C\|(\nablah u^\h,\nablah u^3)\|_{H^2}^2\|( u^\h, u^3)\|_{H^2}^2\| u^\h\|_{H^2}^2,\\
	\label{F4h}
	F_4\lesssim& \|\partial_2 b^\h\|_{H^1}\|\partial_1 b^\h\|_{H^2}^\frac12\|(\partial_2 b^\h,\partial_2 b^3)\|_{H^1}^\frac12\| b^\h\|_{H^2}^\frac12\|( b^\h, b^3)\|_{H^2}^\frac12\nonumber\\
	\leqq &\frac{1}{100}\left(\|\partial_1 b^\h\|_{H^2}^2+\|\partial_2 b^\h\|_{H^1}^2\right)+C\|(\partial_2 b^\h,\partial_2 b^3)\|_{H^1}^2\|( b^\h, b^3)\|_{H^2}^2\| b^\h\|_{H^2}^2.
\end{align}
As to $F_5$, due to \eqref{ppp}, the pressure term $F_5$ could be transformed as
\begin{align*}
	F_5=\int_{\R^3} \nabla_\h(\Delta)^{-1}\mathrm{div}\left(u\cdot\nabla u-b\cdot\nabla b\right) \partial_2 b^\h \mathrm{d}x\lesssim
	\left(\|u\cdot\nabla u\|_{L^2}+\|b\cdot\nabla b\|_{L^2}\right)\|\partial_2 b^\h\|_{L^2}.
\end{align*}
Due to $\mathrm{div}u=0$, it is easy to observe
\begin{align*}
	\|u \cdot \nabla u\|_{L^2} \lesssim \|u^\h \cdot \nablah u\|_{L^2} + \|u^3 \partial_3 u^\h\|_{L^2} + \|u^3 \mathrm{div_h} u^\h\|_{L^2}\lesssim \|u^\h \cdot \nablah u\|_{L^2} + \|u^3 \nabla u^\h\|_{L^2}.
\end{align*}
Then taking the full advantage of \eqref{yi} with $f=u^\h$ (resp. $u^3$), $g=\nabla_\h u$ (resp. $\nabla u^\h$), we obtain
\begin{align*}
	\|u\cdot\nabla u\|_{L^2}\lesssim&\|u^\h\cdot\nabla_\h u\|_{L^2}+\|u^3\nabla u^\h\|_{L^2}\\
	\lesssim& \|\partial_1u^\h\|_{H^1} ^\frac12\|(\partial_2u^\h,\partial_2u^3)\|_{H^1} ^\frac12\| u^\h\|_{H^1} ^\frac12\| (u^\h,u^3)\|_{H^1} ^\frac12\\
	&+\|\partial_1u^3\|_{H^1} ^\frac12\|\partial_2u^\h\|_{H^1} ^\frac12\| u^3\|_{H^1} ^\frac12\| u^\h\|_{H^1} ^\frac12\\
	\lesssim& \|\nablah u^\h\|_{H^1} ^\frac12\|(\nablah u^\h,\nablah u^3)\|_{H^1} ^\frac12\| u^\h\|_{H^1} ^\frac12\| (u^\h,u^3)\|_{H^1} ^\frac12\\
	&+\|\nablah u^3\|_{H^1} ^\frac12\|\nablah u^\h\|_{H^1} ^\frac12\| u^3\|_{H^1} ^\frac12\| u^\h\|_{H^1} ^\frac12.
\end{align*}
Similarly, applying \eqref{yi} with $f=b^\h$ (resp. $\nabla b^\h$), $g=\nabla_\h b$ (resp. $b^3$), we obtain the result for the other term
\begin{align*}
	\|b\cdot\nabla b\|_{L^2}\lesssim&\|b^\h\cdot\nabla_\h b\|_{L^2}+\|b^3\nabla b^\h\|_{L^2}\\
	\lesssim& \|\partial_1b^\h\|_{H^1} ^\frac12\|(\partial_2b^\h,\partial_2b^3)\|_{H^1} ^\frac12\| b^\h\|_{H^1} ^\frac12\| (b^\h,b^3)\|_{H^1} ^\frac12+\|\partial_1b^\h\|_{H^2} ^\frac12\|\partial_2b^3\|_{L^2} ^\frac12\| b^\h\|_{H^2} ^\frac12\| b^3\|_{L^2} ^\frac12.
\end{align*}
Thus, $F_5$ can be estimated as follows
\begin{align}\label{F5}
	F_5\lesssim& \|\partial_2 b^\h\|_{L^2}\Big(\|\nablah u^3\|_{H^1} ^\frac12\|\nablah u^\h\|_{H^1} ^\frac12\| u^3\|_{H^1} ^\frac12\| u^\h\|_{H^1} ^\frac12\nonumber\\
	&+\|\nablah u^\h\|_{H^1} ^\frac12\|(\nablah u^\h,\nablah u^3)\|_{H^1} ^\frac12\| u^\h\|_{H^1} ^\frac12\| (u^\h,u^3)\|_{H^1} ^\frac12\Big)\nonumber\\
	&+\|\partial_2 b^\h\|_{L^2}\Big(\|\partial_1b^\h\|_{H^2} ^\frac12\|\partial_2b^3\|_{L^2} ^\frac12\| b^\h\|_{H^2} ^\frac12\| b^3\|_{L^2} ^\frac12\nonumber\\
	&+\|\partial_1b^\h\|_{H^1} ^\frac12\|(\partial_2b^\h,\partial_2b^3)\|_{H^1} ^\frac12\| b^\h\|_{H^1} ^\frac12\| (b^\h,b^3)\|_{H^1} ^\frac12\Big)\nonumber\\
	\leqq &\frac{1}{100}\left(\|(\partial_1b^\h,\nablah u^\h)\|_{H^2}^2+\|\partial_2 b^\h\|_{L^2}^2\right)+C\|(\partial_2b^\h,\partial_2b^3)\|_{H^1} ^2\| (b^\h,b^3)\|_{H^1} ^2\| b^\h\|_{H^1} ^2\nonumber\\
	&+C\|(\nablah u^\h,\nablah u^3)\|_{H^1} ^2\| (u^\h,u^3)\|_{H^1} ^2\| u^\h\|_{H^1} ^2.
\end{align}
Inserting \eqref{F3}, \eqref{F1h}-\eqref{F5} into \eqref{p2bh_enl2}, we get
\begin{align}\label{Esti.p2bhL2}
	\|\partial_2b^\h\|_{L^2}^2&\leqq\frac{\mathrm{d}}{\mathrm{d} t}\int_{\R^3} u^\h\partial_2 b^\h\mathrm{d} x+\frac{203}{100}\|\nablah u^\h\|_{H^2}^2+\frac{53}{100}\|\partial_1 b^\h\|_{H^2}^2+\frac{3}{100}\|\partial_2 b^\h\|_{H^1}^2+\frac12\|\partial_2 b^\h\|_{L^2}^2\nonumber\\
	&+C\left(\|(\nablah u^\h,\nablah u^3)\|_{H^2}^2\|( u^\h, u^3)\|_{H^2}^2+\|(\partial_2 b^\h,\partial_2 b^3)\|_{H^1}^2\|( b^\h, b^3)\|_{H^2}^2\right)\| (u^\h,b^\h)\|_{H^2}^2.
\end{align}

Then let us handle $\sum_{i=1}^3\|\partial_i\partial_2b^\h\|_{L^2}$.
We first get, by applying $\partial_{i}$ to the equations of  $u^\mathrm{h}$ of \eqref{eq.MHD} and then taking the $L^2$ inner product to it with $\partial_{i}\partial_{2}b^\mathrm{h}$, that
\begin{align}\label{p2bh_enH1}
	\|\partial_{i}\partial_{2}b^\mathrm{h}\|_{L^2}^2=K_1+\cdots+K_5,
\end{align}
where $K_1-K_{5}$ are determined by \eqref{p2b_enH1}. In the same method, from \eqref{k_1}, \eqref{k_2}, \eqref{k_4} and by Young's inequality, we obtain
\begin{align}\label{K1h}
	K_{1}
	\leqq&\frac{\mathrm{d}}{\mathrm{d}t}\int_{\R^3}\partial_{i}u^\h\partial_{i}\partial_{2}b^\h\mathrm{d}x+(\frac32+\frac{1}{100})\|\nablah u^\h\|_{H^2}^2+(\frac12+\frac{1}{100})\|\partial_1 b^\h\|_{H^2}^2\nonumber\\
	&+C\|(\partial_2b^\mathrm{h},\partial_2b^3)\|_{H^{1}}^2\|u^\h\|_{H^{2}}^2+C\|(\nabla_{\mathrm{h}}u^\h,\nabla_{\mathrm{h}}u^3)\|_{H^{2}}^2\|b^\h\|_{H^{2}}^2\nonumber\\
	&+C\|(\partial_1b^\h,\partial_1b^3)\|_{H^2}^2\|(b^\h,b^3)\|_{H^2}^2\|u^\h\|_{H^2}^2+C\|(\nablah u^\h,\nablah u^3)\|_{H^2}^2\|(u^\h,u^3)\|_{H^2}^2\|b^\h\|_{H^2}^2,
\end{align}
and
\begin{align}\label{K2h}
	K_2
	\leqq &\frac{1}{100}\left(\|\nablah u^\h\|_{H^2}^2+\|\partial_{2}b^\h\|_{H^1}^2\right)+C\|(\nablah u^\h,\nablah u^3)\|_{H^2}^2\|( u^\h, u^3)\|_{H^2}^2\| u^\h \|_{H^2}^2,\\\label{K4h}
	K_4
	\leqq &\frac{1}{100}\left(\|\partial_1 b^\h\|_{H^2}^2+\|\partial_{2}b^\h\|_{H^1}^2\right)+C\|(\partial_2 b^\h,\partial_2 b^3)\|_{H^1}^2\|( b^\h, b^3)\|_{H^2}^2\| b^\h \|_{H^2}^2.
\end{align}
Ultimately, with the same derivation of $I_8$, we find
\begin{align*}
	K_{5}=\left(\partial_{i}\nablah p|\partial_{i}\partial_2b^\h\right)_{L^2}\lesssim\|\partial_{i}\nablah p\|_{L^2}\|\partial_2\partial_ib^\h\|_{L^2},
\end{align*}
Then, combining \eqref{P_H2}, $K_5$ can be estimated as
\begin{align}\label{K5}
	K_5\lesssim&\left(\|\nablah u^\h\|_{H^2}\|u^\h\|_{H^2}+ \|\nablah u^\h\|_{H^2}^\frac12\|\nablah u^3\|_{H^2}^\frac12\|u^\h\|_{H^2}^\frac12\|u^3\|_{H^2}^\frac12 \right)\|\partial_2b^\h\|_{H^{1}}\nonumber\\
	&+\left(\|\partial_1b^\h\|_{H^2}^\frac12\|\partial_2b^\h\|_{H^1}^\frac12\|b^\h\|_{H^2}+\|\partial_1b^\h\|_{H^2}^\frac12\|\partial_2b^3\|_{H^1}^\frac12\|b^\h\|_{H^2}^\frac12\|b^3\|_{H^2}^\frac12\right)\|\partial_2b^\h\|_{H^{1}}\nonumber\\
	\leqq&\frac{1}{100}\left(\|(\nablah u^\h,\partial_1b^\h)\|_{H^2}^2+\|\partial_2b^\h\|_{H^{1}}^2\right)+C\left(\| \partial_1b^\h\|_{H^2}^2+\|\partial_2b^\h\|_{H^1}^2\right)\|b^\h\|_{H^2}^2\nonumber\\
	&+C\|\nablah u^\h\|_{H^2}^2\|u^\h\|_{H^2}^2+C\|\nablah u^3\|_{H^2}^2\|u^3\|_{H^2}^2\|u^\h\|_{H^2}^2+C\|\partial_2b^3\|_{H^1}^2\|b^3\|_{H^2}^2\|b^\h\|_{H^2}^2.
\end{align}
Inserting \eqref{K3}, \eqref{K1h}-\eqref{K5} into \eqref{p2bh_enH1}, we get
\begin{align}\label{Esti.p2bhH1}
	\|\partial_{i}\partial_{2}b^\mathrm{h}\|_{L^2}^2\leqq&\frac{\mathrm{d}}{\mathrm{d}t}\int_{\R^3}\partial_{i}u^\h\partial_{i}\partial_{2}b^\h\mathrm{d}x+\frac{203}{100}\|\nablah u^\h\|_{H^2}^2+\frac{53}{100}\|\partial_1 b^\h\|_{H^2}^2+\frac{3}{100}\|\partial_2 b^\h\|_{H^1}^2\nonumber\\
	&+\frac12\|\partial_2 \partial_ib^\h\|_{L^2}^2+C\|(\nabla_{\mathrm{h}}u^\h,\partial_1b^\h)\|_{H^{2}}^2\|(u^\h,b^\h)\|_{H^{2}}^2\nonumber\\
	&+C\|(\partial_2b^\mathrm{h},\partial_2b^3)\|_{H^{1}}^2\|(u^\h,b^\h)\|_{H^{2}}^2+C\|\nabla_{\mathrm{h}}u^3\|_{H^{2}}^2\|b^\h\|_{H^{2}}^2\nonumber\\
	&+C\left(\|(\partial_1b^\h,\partial_1b^3)\|_{H^2}^2\|(b^\h,b^3)\|_{H^2}^2+\|(\nablah u^\h,\nablah u^3)\|_{H^2}^2\|(u^\h,u^3)\|_{H^2}^2\right)\|u^\h\|_{H^2}^2\nonumber\\
	&+C\left(\|(\partial_2b^\h,\partial_2b^3)\|_{H^1}^2\|(b^\h,b^3)\|_{H^2}^2+\|(\nablah u^\h,\nablah u^3)\|_{H^2}^2\|(u^\h,u^3)\|_{H^2}^2\right)\|b^\h\|_{H^2}^2.
\end{align}
Combining \eqref{Esti.p2bhL2} and \eqref{Esti.p2bhH1}, we have \eqref{p2bhH1}.

\appendix
\section{The estimate of wild terms}
In this part,  we shall give a key tool to treat the  problematic term $J_{325}$. The term derives from the $H^{2}$-energy estimates of  $(u^\h,b^\h)$.
Due to their high-order derivative structures, these terms cannot be directly bounded via standard anisotropic Sobolev inequalities. To address this challenge, we develop a more refined analytical approach to derive the required estimates.
\begin{lem}\label{wild}
	The following estimate hold when the right-hand terms are all bounded.
	\begin{align}\label{Esti.wild}
		\int_{\R^3}\partial_2u^i\partial_3^2b^j\partial_3^2b^k\dx\leqq&\frac{\mathrm{d}}{\dt}\int_{\R^3} b^i\partial_3^2b^j\partial_3^2b^k-b^i\partial_3^2u^j\partial_3^2u^k\dx+\frac{3}{100}\|(\nablah u^\h,\partial_1b^\h)\|_{H^2}^2\nonumber\\
		&+C\left(\|(\nablah u^\h,\partial_1b^\h)\|_{H^2}^2+\|(\partial_2b^\h,\partial_2b^3)\|_{H^1}^2\right)\|(u^\h,b^\h)\|_{H^2}^2\nonumber\\
		&+C\left(\|\nablah u^3\|_{H^2}^2\|u^3\|_{H^2}^2+ \|(\partial_2b^\h,\partial_2b^3)\|_{H^1}^2\|(b^\h,b^3)\|_{H^2}^2\right)\|u^\h\|_{H^2}^2\nonumber\\	&+C\left(\|\partial_2b^3\|_{H^1}^2\|b^3\|_{H^2}^2+\|(\nablah u^\h,\nablah u^3)\|_{H^2}^2\|(u^\h,u^3)\|_{H^2}^2\right)\|b^\h\|_{H^2}^2,
	\end{align}
	where $i,j,k=1,2$.
\end{lem}
\begin{proof} We first note that \eqref{ha} gives
	\begin{align*}
		\int_{\R^3}&\partial_2u^i\partial_3^2b^j\partial_3^2b^k\dx=\int_{\R^3}({\partial_t}b^i+u\cdot\nabla{b}^i-\partial_1^2b^i -b\cdot\nabla{u}^i)\partial_3^2b^j\partial_3^2b^k\dx \\
		&=\int_{\R^3}\frac{\mathrm{d}}{\dt}(b^i\partial_3^2b^j\partial_3^2b^k)-b^i\partial_t\partial_3^2b^j\partial_3^2b^k-b^i\partial_t\partial_3^2b^k\partial_3^2b^j+u\cdot\nabla b^i\partial_3^2b^j\partial_3^2b^k\dx+S^{(i,j,k)}_1\\
		&=\frac{\mathrm{d}}{\dt}\int_{\R^3} b^i\partial_3^2b^j\partial_3^2b^k\dx+\int_{\R^3} b^i\partial_3^2b^k\partial_3^2(u\cdot\nabla{b}^j-\partial_1^2b^j -b\cdot\nabla{u}^j-\partial_2u^j)\dx\\
		&\quad\int_{\R^3} b^i\partial_3^2b^j\partial_3^2(u\cdot\nabla{b}^k-\partial_1^2b^k -b\cdot\nabla{u}^k-\partial_2u^k)\dx+\int_{\R^3} u\cdot\nabla b^i\partial_3^2b^j\partial_3^2b^k\dx+S^{(i,j,k)}_1\\
		&=\int_{\R^3} b^i\partial_3^2b^k\partial_3^2(u\cdot\nabla{b}^j)+b^i\partial_3^2b^j\partial_3^2(u\cdot\nabla{b}^k)+u\cdot\nabla b^i\partial_3^2b^j\partial_3^2b^k\dx+S^{(i,j,k)}_1+S^{(i,j,k)}_2\\
		&=\int_{\R^3} b^i\partial_3^2b^k(\partial_3^2u\cdot\nabla{b}^j+2\partial_3u\cdot\nabla\partial_3{b}^j)+b^i\partial_3^2b^j(\partial_3^2u\cdot\nabla{b}^k+2\partial_3u\cdot\nabla\partial_3{b}^k)\dx\\
		&\quad+\int_{\R^3} \big(b^i\partial_3^2b^ku\cdot\nabla\partial_3^2{b}^j+
		b^i\partial_3^2b^ju\cdot\nabla\partial_3^2{b}^k+u\cdot\nabla b^i\partial_3^2b^j\partial_3^2b^k\big)\dx+S^{(i,j,k)}_1+S^{(i,j,k)}_2\\
		&=\int_{\R^3} u\cdot \nabla(b^i\partial_3^2b^k\partial_3^2{b}^j)\dx+S^{(i,j,k)}_1+S^{(i,j,k)}_2+S^{(i,j,k)}_3=S^{(i,j,k)}_1+S^{(i,j,k)}_2+S^{(i,j,k)}_3,
	\end{align*}
	where
	\begin{align*}
		S^{(i,j,k)}_1&=-\int_{\R^3} (\partial_1^2b^i +b\cdot\nabla{u}^i)\partial_3^2b^j\partial_3^2b^k\dx,\\
		S^{(i,j,k)}_2&=\frac{\mathrm{d}}{\dt}\int_{\R^3} b^i\partial_3^2b^j\partial_3^2b^k\dx-\int_{\R^3} b^i\partial_3^2b^k\partial_3^2(\partial_1^2b^j +b\cdot\nabla{u}^j+\partial_2u^j)\dx\\
		&\quad-\int_{\R^3} b^i\partial_3^2b^j\partial_3^2(\partial_1^2b^k +b\cdot\nabla{u}^k+\partial_2u^k)\dx,\\
		S^{(i,j,k)}_3&=\int_{\R^3} b^i\partial_3^2b^k(\partial_3^2u\cdot\nabla{b}^j+2\partial_3u\cdot\nabla\partial_3{b}^j)+b^i\partial_3^2b^j(\partial_3^2u\cdot\nabla{b}^k+2\partial_3u\cdot\nabla\partial_3{b}^k)\dx,
	\end{align*}
	and we bound these decomposed terms in order. As to $S^{(i,j,k)}_1$, by integration by parts, H\"older's inequality and \eqref{san} with $f=b^\h$ (resp. $b^3$), $g=\nabla{u}^i,h=\partial_3^2b^j,v=\partial_3^2b^k$, $(m,n)=(2,1)$, we deduce
	\begin{align}\label{S1}
		S^{(i,j,k)}_1=&\int_{\R^3}-b\cdot\nabla{u}^i\partial_3^2b^j\partial_3^2b^k+ \partial_1b^i(\partial_1\partial_3^2b^j\partial_3^2b^k+\partial_3^2b^j\partial_1\partial_3^2b^k) \dx\nonumber\\
		\lesssim&\|(\partial_2b^\h,\partial_2b^3)\|_{H^1}^\frac12\|(b^\h,b^3)\|_{H^1}^\frac12\|\partial_2u^\h\|_{H^2}^\frac12\|u^\h\|_{H^2}^\frac12\|\partial_1b^\h\|_{H^2}\|b^\h\|_{H^2}\nonumber\\
		&+\|\partial_1b^\h\|_{L^\infty}\|\partial_1\partial_3^2b^\h\|_{L^2}\|\partial_3^2b^\h\|_{L^2}\nonumber\\
		\lesssim&\|(\partial_2b^\h,\partial_2b^3)\|_{H^1}^\frac12\|\nablah u^\h\|_{H^2}^\frac12\|\partial_1b^\h\|_{H^2}\|(b^\h,b^3)\|_{H^2}^\frac12\|u^\h\|_{H^2}^\frac12\|b^\h\|_{H^2}+\|\partial_1b^\h\|_{H^2}^2\|b^\h\|_{H^2}\nonumber\\
		\leqq &\frac{1}{100}\|(\nablah u^\h,\partial_1b^\h)\|_{H^2}^2+C\|\partial_1b^\h\|_{H^2}^2\|b^\h\|_{H^2}^2+C\|(\partial_2b^\h,\partial_2b^3)\|_{H^1}^2\|(b^\h,b^3)\|_{H^2}^2\|u^\h\|_{H^2}^2.
	\end{align}
	The term $S^{(i,j,k)}_2$ is a difficult term. It could be written as
	\begin{align*}
		S^{(i,j,k)}_2=&\frac{\mathrm{d}}{\dt}\int_{\R^3} b^i\partial_3^2b^j\partial_3^2b^k\dx-\int_{\R^3} b^i\partial_3^2b^k b\cdot\nabla{\partial_3^2u}^j+b^i\partial_3^2b^jb\cdot\nabla{\partial_3^2u}^k\dx\\
		&-\int_{\R^3} b^i\partial_3^2b^k(\partial_3^2 b\cdot\nabla{u}^j+2\partial_3 b\cdot\nabla{\partial_3u}^j)+b^i\partial_3^2b^j(\partial_3^2b\cdot\nabla{u}^k+2\partial_3b\cdot\nabla{\partial_3u}^k)\dx\\
		&-\int_{\R^3} b^i\partial_3^2b^k\partial_3^2(\partial_1^2b^j +\partial_2u^j)+ b^i\partial_3^2b^j\partial_3^2(\partial_1^2b^k +\partial_2u^k)\dx\\
		\eqdefa&\frac{\mathrm{d}}{\dt}\int_{\R^3} b^i\partial_3^2b^j\partial_3^2b^k\dx+S^{(i,j,k)}_{21}+S^{(i,j,k)}_{22}+S^{(i,j,k)}_{23},
	\end{align*}
	where
	\begin{align*}
		S^{(i,j,k)}_{21}&=-\int_{\R^3} b^i\partial_3^2b^k b\cdot\nabla{\partial_3^2u}^j+b^i\partial_3^2b^jb\cdot\nabla{\partial_3^2u}^k\dx,\\
		S^{(i,j,k)}_{22}&=-\int_{\R^3} b^i\partial_3^2b^k(\partial_3^2 b\cdot\nabla{u}^j+2\partial_3 b\cdot\nabla{\partial_3u}^j)+b^i\partial_3^2b^j(\partial_3^2b\cdot\nabla{u}^k+2\partial_3b\cdot\nabla{\partial_3u}^k)\dx,\\
		S^{(i,j,k)}_{23}&=-\int_{\R^3} b^i\partial_3^2b^k\partial_3^2(\partial_1^2b^j +\partial_2u^j)+ b^i\partial_3^2b^j\partial_3^2(\partial_1^2b^k +\partial_2u^k)\dx.
	\end{align*}
	For $S^{(i,j,k)}_{22}$, it is evident that we have
	\begin{align*}
		S^{(i,j,k)}_{22}\lesssim\int_{\R^3} |b^\h\partial_3^2b^\h\partial_3^2 b\cdot\nabla{u}^\h|+|b^\h\partial_3^2b^\h\partial_3 b\cdot\nabla{\partial_3u}^\h|\dx.
	\end{align*}
	And then, for the two terms above, assign different indices of \eqref{san}: take $f=b^\h,g=\nabla{u}^\h,h=\partial_3^2b^\h,v=\partial_3^2 b,(m,n)=(2,1)$ for the first term and $f=b^\h,g=\partial_3 b,h=\nabla{\partial_3u}^\h,v=\partial_3^2 b,(m,n)=(1,2)$ for the second, we get
	\begin{align}\label{S22}
		S^{(i,j,k)}_{22}
		\lesssim&\|\partial_1b^\h\|_{H^2}^\frac12\|(\partial_1b^\h,\partial_1b^3)\|_{H^2}^\frac12\|\partial_2b^\h\|_{H^1}^\frac12\|\nablah u^\h\|_{H^2}^\frac12\|b^\h\|_{H^2}\|(b^\h,b^3)\|_{H^2}^\frac12\|u^\h\|_{H^2}^\frac12.
	\end{align}
	$S^{(i,j,k)}_{23}$ also can be bounded directly. By  integration by parts, H\"older's inequality, Sobolev embedding theorem and \eqref{L^infty_1D}, it leads
	\begin{align}\label{S23}
		S^{(i,j,k)}_{23}=&\int_{\R^3} \partial_1b^i(\partial_3^2b^k\partial_3^2\partial_1b^j + \partial_3^2b^j\partial_3^2\partial_1^2b^k) +2b^i\partial_3^2\partial_1b^k\partial_3^2\partial_1b^j\dx\nonumber\\
		&-\int_{\R^3} b^i\partial_3^2b^k\partial_3^2\partial_2u^j+ b^i\partial_3^2b^j\partial_3^2\partial_2u^k\dx\nonumber\\
		\lesssim&\|\partial_1b^\h\|_{L^\infty}\|\partial_3^2b^\h\|_{L^2}\|\partial_1\partial_3^2b^\h\|_{L^2}+\|b^\h\|_{L^\infty}\|\partial_3^2\partial_1b^\h\|_{L^2}^2\nonumber\\
		&+\|b^\h\|_{L^2_{x_1}L^\infty_{x_2}L^\infty_{x_3}}\|\partial_3^2b^\h\|_{L^\infty_{x_1}L^2_{x_2}L^2_{x_3}}\|\partial_2\partial_3^2u^\h\|_{L^2}\nonumber\\
		\lesssim&\|\partial_1b^\h\|_{H^2}^2\|b^\h\|_{H^2}+\|\nablah u^\h\|_{H^2}^\frac12\|\partial_1b^\h\|_{H^2}^\frac12\|\partial_2b^\h\|_{H^1}\|b^\h\|_{H^2}.
	\end{align}
	In view of \eqref{W2} and with the help of \eqref{bbbb}, we further split $S^{(i,j,k)}_{21}$ as
	\begin{align}\label{S21deco}
		S^{(i,j,k)}_{21}&=\int_{\R^3} b^ib\cdot\nabla (\partial_3^2b^k\partial_3^2u^j)-b^i\partial_3^2u^jb\cdot\nabla \partial_3^2b^k+ b^ib\cdot\nabla (\partial_3^2b^j\partial_3^2u^k)-b^i\partial_3^2u^kb\cdot\nabla \partial_3^2b^j\dx\nonumber\\
		&=\int_{\R^3}-b^i\partial_3^2u^j\partial_3^2(b\cdot\nabla b^k)+b^i\partial_3^2u^j\partial_3^2b\cdot\nabla b^k+2b^i\partial_3^2u^j\partial_3b\cdot\nabla \partial_3b^k\dx\nonumber\\
		&\quad+\int_{\R^3}-b^i\partial_3^2u^k\partial_3^2(b\cdot\nabla b^j)+b^i\partial_3^2u^k\partial_3^2b\cdot\nabla b^j+2b^i\partial_3^2u^k\partial_3b\cdot\nabla \partial_3b^j\dx+S^{(i,j,k)}_{211}\nonumber\\
		&=-\int_{\R^3}b^i\partial_3^2u^j\partial_3^2(\partial_tu^k+u\cdot \nabla u^k-\Delta_\h u^k+\partial_kp-\partial_2b^k)\dx\nonumber\\
		&\quad-\int_{\R^3}b^i\partial_3^2u^k\partial_3^2(\partial_tu^j+u\cdot \nabla u^j-\Delta_\h u^j+\partial_jp-\partial_2b^j)\dx+S^{(i,j,k)}_{211}+S^{(i,j,k)}_{212}\nonumber\\
		&=-\int_{\R^3}b^i\partial_3^2u^j\partial_3^2\partial_tu^k+b^i\partial_3^2u^k\partial_3^2\partial_tu^j+b^i\partial_3^2u^j u\cdot \nabla \partial_3^2u^k+b^i\partial_3^2u^k u\cdot \nabla \partial_3^2u^j\dx\nonumber\\
		&\quad+S^{(i,j,k)}_{211}+S^{(i,j,k)}_{212}+S^{(i,j,k)}_{213}\nonumber\\
		&=\int_{\R^3}-\frac{\mathrm{d}}{\dt}(b^i\partial_3^2u^j\partial_3^2u^k)+\partial_tb^i\partial_3^2u^k\partial_3^2u^j-b^i u\cdot \nabla (\partial_3^2u^k\partial_3^2u^j)\dx\nonumber\\
		&\quad+S^{(i,j,k)}_{211}+S^{(i,j,k)}_{212}+S^{(i,j,k)}_{213}\nonumber\\
		& =-\frac{\mathrm{d}}{\dt}\int_{\R^3}b^i\partial_3^2u^j\partial_3^2u^k\dx+\int_{\R^3}(-u\cdot\nabla b^i+\partial_1^2b^i+b\cdot\nabla u^i+\partial_2u^i)	\partial_3^2u^k\partial_3^2u^j\dx\nonumber\\
		&\quad -\int_{\R^3}b^i u\cdot \nabla (\partial_3^2u^k\partial_3^2u^j)\dx+S^{(i,j,k)}_{211}+S^{(i,j,k)}_{212}+S^{(i,j,k)}_{213}\nonumber\\
		& =-\frac{\mathrm{d}}{\dt}\int_{\R^3}b^i\partial_3^2u^j\partial_3^2u^k\dx+S^{(i,j,k)}_{211}+S^{(i,j,k)}_{212}+S^{(i,j,k)}_{213}+S^{(i,j,k)}_{214},
	\end{align}
	where
	\begin{align*}
		S^{(i,j,k)}_{211}&=\int_{\R^3} b^ib\cdot\nabla (\partial_3^2b^k\partial_3^2u^j+\partial_3^2b^j\partial_3^2u^k)\dx,\\
		S^{(i,j,k)}_{212}&=\int_{\R^3}b^i(\partial_3^2u^j\partial_3^2b\cdot\nabla b^k+\partial_3^2u^k\partial_3^2b\cdot\nabla b^j)+2b^i(\partial_3^2u^j\partial_3b\cdot\nabla \partial_3b^k+\partial_3^2u^k\partial_3b\cdot\nabla \partial_3b^j)\dx,\\
		S^{(i,j,k)}_{213}&=\int_{\R^3}b^i\partial_3^2u^j\partial_3^2(\Delta_\h u^k-\partial_kp+\partial_2b^k)+b^i\partial_3^2u^k\partial_3^2(\Delta_\h u^j-\partial_jp+\partial_2b^j)\dx,\\
		S^{(i,j,k)}_{214}&=\int_{\R^3}(\partial_1^2b^i+b\cdot\nabla u^i+\partial_2u^i)\partial_3^2u^k\partial_3^2u^j\dx.
	\end{align*}
	To $S^{(i,j,k)}_{211}$, by making use of \eqref{san} with $f=b,g=\nabla b^i$, $h=\partial_3^2b^k$ (resp.$\partial_3^2b^j$), $v=\partial_3^2u^j$ (resp. $\partial_3^2u^k$), $(m,n)=(1,2)$ and integration by parts, it yields
	\begin{align}\label{S211}
		S^{(i,j,k)}_{211}=&-\int_{\R^3} b\cdot\nabla b^i (\partial_3^2b^k\partial_3^2u^j+\partial_3^2b^j\partial_3^2u^k)\dx\nonumber\\
		\lesssim&\|(\partial_2b^\h,\partial_2b^3)\|_{H^1}^\frac12\|(b^\h,b^3)\|_{H^1}^\frac12\|\partial_1b^\h\|_{H^2}\|b^\h\|_{H^2}\|\partial_2u^\h\|_{H^2}^\frac12\|u^\h\|_{H^2}^\frac12\nonumber\\
		\lesssim&\|(\partial_2b^\h,\partial_2b^3)\|_{H^1}^\frac12\|\partial_1b^\h\|_{H^2}\|\nablah u^\h\|_{H^2}^\frac12\|(b^\h,b^3)\|_{H^2}^\frac12\|b^\h\|_{H^2}\|u^\h\|_{H^2}^\frac12.
	\end{align}
	For $S^{(i,j,k)}_{212}$, by taking advantage of \eqref{san} with $f=b^i$, $g=\nabla b^k$ (resp. $\nabla b^j,\partial_3b$), $h=\partial_3^2b$ (resp. $\nabla \partial_3b^k,\nabla \partial_3b^j$), $v=\partial_3^2u^j$ (resp. $\partial_3^2u^k,\partial_3^2u^j,\partial_3^2u^k$), $(m,n)=(1,2)$, it leads
	\begin{align}\label{S212}
		S^{(i,j,k)}_{212}	\lesssim&\|\partial_2b^\h\|_{H^1}^\frac12\|b^\h\|_{H^1}^\frac12\|\partial_1b^\h\|_{H^2}^\frac12\|b^\h\|_{H^2}^\frac12\|(\partial_1b^\h,\partial_1b^3)\|_{H^2}^\frac12\|(b^\h,b^3)\|_{H^2}^\frac12\|\partial_2u^\h\|_{H^2}^\frac12\|u^\h\|_{H^2}^\frac12\nonumber\\
		\lesssim&\|\partial_2b^\h\|_{H^1}^\frac12\|\partial_1b^\h\|_{H^2}^\frac12\|(\partial_1b^\h,\partial_1b^3)\|_{H^2}^\frac12\|\nablah u^\h\|_{H^2}^\frac12\|(b^\h,b^3)\|_{H^2}^\frac12\|b^\h\|_{H^2}\|u^\h\|_{H^2}^\frac12.
	\end{align}
	By integration by parts, we divide $S^{(i,j,k)}_{213}$ into
	\begin{align*}
		S^{(i,j,k)}_{213}=&\sum_{l=1}^{2}\int_{\R^3}(\partial_lb^i\partial_3^2u^j+b^i\partial_l\partial_3^2u^j)\partial_3^2\partial_l u^k+(\partial_lb^i\partial_3^2u^k+b^i\partial_l\partial_3^2u^k)\partial_3^2\partial_l u^j\dx\\
		&+\int_{\R^3}(\partial_kb^i\partial_3^2u^j+b^i\partial_k\partial_3^2u^j)\partial_3^2p+(\partial_jb^i\partial_3^2u^k+b^i\partial_j\partial_3^2u^k)\partial_3^2p\dx\\
		&+\int_{\R^3}(\partial_2b^i\partial_3^2u^j+b^i\partial_2\partial_3^2u^j)\partial_3^2b^k+(\partial_2b^i\partial_3^2u^k+b^i\partial_2\partial_3^2u^k)\partial_3^2b^j\dx\\
		\eqdefa& S^{(i,j,k)}_{2131}+S^{(i,j,k)}_{2132}+S^{(i,j,k)}_{2133}.
	\end{align*}
	In view of $S^{(i,j,k)}_{22}$, by H\"older's inequality, Sobolev embedding theorem and \eqref{L^infty_1D}, we get
	\begin{align}
		S^{(i,j,k)}_{2131}\lesssim&\|\partial_l b^\h\|_{L^\infty_{x_1}L^2_{x_2}L^\infty_{x_3}}\|\partial_3^2u^\h\|_{L^2_{x_1}L^\infty_{x_2}L^2_{x_3}}\|\partial_3^2\partial_l u^\h\|_{L^2}+\|b^\h\|_{L^\infty}\|\partial_3^2\partial_l u^\h\|_{L^2}^2\nonumber\\
		\lesssim&\|\partial_1b^\h\|_{H^2}^\frac12\|\nablah u^\h\|_{H^2}^\frac32\| b^\h\|_{H^2}^\frac12\| u^\h\|_{H^2}^\frac12+\|\nablah u^\h\|_{H^2}^2\|b^\h\|_{H^2},\nonumber\\\label{S2132.0}
		S^{(i,j,k)}_{2132}\lesssim&\left(\|\nablah b^\h\|_{L^\infty_{x_1}L^2_{x_2}L^\infty_{x_3}}\|\partial_3^2u^\h\|_{L^2_{x_1}L^\infty_{x_2}L^2_{x_3}}+\|b^\h\|_{L^\infty}\|\partial_3^2\nablah u^\h\|_{L^2}\right)\|\partial_3^2 p\|_{L^2}\nonumber\\
		\lesssim&\left(\|\partial_1b^\h\|_{H^2}^\frac12\|\nablah u^\h\|_{H^2}^\frac12\| b^\h\|_{H^2}^\frac12\| u^\h\|_{H^2}^\frac12+\|\nablah u^\h\|_{H^2}\|b^\h\|_{H^2}\right)\|\partial_3^2 p\|_{L^2},\\
		S^{(i,j,k)}_{2133}\lesssim&\left(\|\partial_2 b^\h\|_{L^2_{x_1}L^2_{x_2}L^\infty_{x_3}}\|\partial_3^2u^\h\|_{L^2_{x_1}L^\infty_{x_2}L^2_{x_3}}+\|b^\h\|_{L^2_{x_1}L^\infty_{x_2}L^\infty_{x_3}}\|\partial_2\partial_3^2 u^\h\|_{L^2}\right)\|\partial_3^2 b^\h\|_{L^\infty_{x_1}L^2_{x_2}L^2_{x_3}}\nonumber\\
		\lesssim&\left(\|\partial_2b^\h\|_{H^1}\|\nablah u^\h\|_{H^2}^\frac12\| u^\h\|_{H^2}^\frac12+\|\nablah u^\h\|_{H^2}\|\partial_2b^\h\|_{H^1}^\frac12\|b^\h\|_{H^2}^\frac12 \right) \|\partial_1b^\h\|_{H^2}^\frac12\| b^\h\|_{H^2}^\frac12.\nonumber
	\end{align}
	Inserting \eqref{P_H2} into \eqref{S2132.0}, it makes the following estimate hold
	\begin{align*}
		S^{(i,j,k)}_{2132}
		\lesssim&\left(\|\partial_1b^\h\|_{H^2}^\frac12\|\nablah u^\h\|_{H^2}^\frac12\| b^\h\|_{H^2}^\frac12\| u^\h\|_{H^2}^\frac12+\|\nablah u^\h\|_{H^2}\|b^\h\|_{H^2}\right)\nonumber\\
		&\times\bigg(\|\nablah u^\h\|_{H^2}\|u^\h\|_{H^2}+ \|\nablah u^\h\|_{H^2}^\frac12\|\nablah u^3\|_{H^2}^\frac12\|u^\h\|_{H^2}^\frac12\|u^3\|_{H^2}^\frac12 \nonumber\\
		&+\|\partial_1b^\h\|_{H^2}^\frac12\|\partial_2b^\h\|_{H^1}^\frac12\|b^\h\|_{H^2}+\|\partial_1b^\h\|_{H^2}^\frac12\|\partial_2b^3\|_{H^1}^\frac12\|b^\h\|_{H^2}^\frac12\|b^3\|_{H^2}^\frac12\bigg)\nonumber\\
		\lesssim&\|\nablah u^\h\|_{H^2}^2\|(u^\h,b^\h)\|_{H^2}^2+\|\partial_1b^\h\|_{H^2}\|b^\h\|_{H^2}\left(\|\partial_2b^\h\|_{H^1}\|b^\h\|_{H^2}+\|\partial_2b^3\|_{H^1}\|b^3\|_{H^2}\right)\nonumber\\
		&+ \|\nablah u^\h\|_{H^2}\|u^\h\|_{H^2}\left(\|\nablah u^3\|_{H^2}\|u^3\|_{H^2}+\|\partial_1b^\h\|_{H^2}\| b^\h\|_{H^2}\right).
	\end{align*}
	As a result, combining above inequalities, we get
	\begin{align}\label{S213}
		S^{(i,j,k)}_{213}\lesssim&\|\nablah u^\h\|_{H^2}^2\|(u^\h,b^\h)\|_{H^2}^2+\|\partial_1b^\h\|_{H^2}^\frac12\|\partial_2b^\h\|_{H^1}^\frac12\|\nablah u^\h\|_{H^2}\|b^\h\|_{H^2}\nonumber\\
		&+\left(\|\nablah u^\h\|_{H^2}+\|\partial_2b^\h\|_{H^1}\right)\|\nablah u^\h\|_{H^2}^\frac12\|\partial_1b^\h\|_{H^2}^\frac12\| u^\h\|_{H^2}^\frac12  \| b^\h\|_{H^2}^\frac12\nonumber\\
		&+\|\partial_1b^\h\|_{H^2}\|b^\h\|_{H^2}\left(\|\partial_2b^\h\|_{H^1}\|b^\h\|_{H^2}+\|\partial_2b^3\|_{H^1}\|b^3\|_{H^2}\right)\nonumber\\
		&+ \|\nablah u^\h\|_{H^2}\|u^\h\|_{H^2}\left(\|\nablah u^3\|_{H^2}\|u^3\|_{H^2}+\|\partial_1b^\h\|_{H^2}\| b^\h\|_{H^2}\right).
	\end{align}
	As to  $S^{(i,j,k)}_{214}$,  along the same lines of  $S^{(i,j,k)}_{211}$ and $S^{(i,j,k)}_{212}$, by applying \eqref{L^infty_1D}, Sobolev embedding theorem and \eqref{san} with $f=b,g=\nabla u^i,h=\partial_3^2u^k,v=\partial_3^2u^j,(m,n)=(2,1)$ again, we obtain
	\begin{align}\label{S214}
		S^{(i,j,k)}_{214}=&\int_{\R^3}-\partial_1b^i(\partial_1\partial_3^2u^k\partial_3^2u^j+\partial_3^2u^k\partial_1\partial_3^2u^j)+(b\cdot\nabla u^i+\partial_2u^i)\partial_3^2u^k\partial_3^2u^j\dx\nonumber\\
		\lesssim&\|\partial_1b^\h\|_{L^\infty}\|\partial_1\partial_3^2u^\h\|_{L^2}\|\partial_3^2u^\h\|_{L^2}+\|\partial_2u^\h\|_{L^2_{x_1}L^2_{x_2}L^\infty_{x_3}}\|\partial_3^2u^\h\|_{L^\infty_{x_1}L^2_{x_2}L^2_{x_3}}\|\partial_3^2u^\h\|_{L^2_{x_1}L^\infty_{x_2}L^2_{x_3}}\nonumber\\
		&+\|(\partial_2b^\h,\partial_2b^3)\|_{H^1}^\frac12\|(b^\h,b^3)\|_{H^1}^\frac12\|\partial_2u^\h\|_{H^2}^\frac12\|u^\h\|_{H^2}^\frac12\|\partial_1u^\h\|_{H^2}\|u^\h\|_{H^2}\nonumber\\
		\lesssim&\left(\|\partial_1b^\h\|_{H^2}+\|\nablah u^\h\|_{H^2}\right)\|\nablah u^\h\|_{H^2}\|u^\h\|_{H^2}\nonumber\\
		&+\|(\partial_2b^\h,\partial_2b^3)\|_{H^1}^\frac12\|\nablah u^\h\|_{H^2}^\frac32\|(b^\h,b^3)\|_{H^1}^\frac12\|u^\h\|_{H^2}^\frac32.
	\end{align}
	Substituting \eqref{S211}, \eqref{S212}, \eqref{S213} and \eqref{S214} into \eqref{S21deco} gives the estimate of $S^{(i,j,k)}_{21}$; together with \eqref{S22}, \eqref{S23}, this finally leads to
	\begin{align*}
		S^{(i,j,k)}_{2}\leqq&\frac{\mathrm{d}}{\dt}\int_{\R^3} b^i\partial_3^2b^j\partial_3^2b^k-b^i\partial_3^2u^j\partial_3^2u^k\dx+C\|\nablah u^\h\|_{H^2}^2\|(u^\h,b^\h)\|_{H^2}\left(1+\|(u^\h,b^\h)\|_{H^2}\right)\nonumber\\
		&+C\|\partial_1b^\h\|_{H^2}^\frac12\|\partial_2b^\h\|_{H^1}^\frac12\|\nablah u^\h\|_{H^2}\|b^\h\|_{H^2}+C\|\partial_1b^\h\|_{H^2}\|\nablah u^\h\|_{H^2}\|u^\h\|_{H^2}\nonumber\\
		&+C\left(\|\nablah u^\h\|_{H^2}+\|\partial_2b^\h\|_{H^1}\right)\|\nablah u^\h\|_{H^2}^\frac12\|\partial_1b^\h\|_{H^2}^\frac12\| u^\h\|_{H^2}^\frac12  \| b^\h\|_{H^2}^\frac12\nonumber\\
		&+C\|(\partial_2b^\h,\partial_2b^3)\|_{H^1}^\frac12\|\nablah u^\h\|_{H^2}^\frac32\|(b^\h,b^3)\|_{H^1}^\frac12\|u^\h\|_{H^2}^\frac32\nonumber\\
		&+C\|(\partial_2b^\h,\partial_2b^3)\|_{H^1}^\frac12\|\nablah u^\h\|_{H^2}^\frac12\|\partial_1b^\h\|_{H^2}\|(b^\h,b^3)\|_{H^2}^\frac12\|b^\h\|_{H^2}\|u^\h\|_{H^2}^\frac12\nonumber\\
		&+C\|\partial_2b^\h\|_{H^1}^\frac12\|\partial_1b^\h\|_{H^2}^\frac12\|(\partial_1b^\h,\partial_1b^3)\|_{H^2}^\frac12\|\nablah u^\h\|_{H^2}^\frac12\|(b^\h,b^3)\|_{H^2}^\frac12\|b^\h\|_{H^2}\|u^\h\|_{H^2}^\frac12\nonumber\\
		&+C\|\partial_1b^\h\|_{H^2}\|b^\h\|_{H^2}\left(\|\partial_2b^\h\|_{H^1}\|b^\h\|_{H^2}+\|\partial_2b^3\|_{H^1}\|b^3\|_{H^2}\right)\nonumber\\
		&+ C\|\nablah u^\h\|_{H^2}\|u^\h\|_{H^2}\left(\|\nablah u^3\|_{H^2}\|u^3\|_{H^2}+\|\partial_1b^\h\|_{H^2}\| b^\h\|_{H^2}\right).
	\end{align*}
	Applying Young's inequality, it brings out
	\begin{align}\label{S2}
		S^{(i,j,k)}_{2}\leqq &\frac{\mathrm{d}}{\dt}\int_{\R^3} b^i\partial_3^2b^j\partial_3^2b^k-b^i\partial_3^2u^j\partial_3^2u^k\dx+\frac{1}{100}\|(\nablah u^\h,\partial_1b^\h)\|_{H^2}^2\nonumber\\
		&+C\left(\|(\nablah u^\h,\partial_1b^\h)\|_{H^2}^2+\|(\partial_2b^\h,\partial_2b^3)\|_{H^1}^2\right)\|(u^\h,b^\h)\|_{H^2}^2\nonumber\\
		&+C\|\nablah u^3\|_{H^2}^2\|u^3\|_{H^2}^2 \|u^\h\|_{H^2}^2+C\|\partial_2b^3\|_{H^1}^2\|b^3\|_{H^2}^2\|b^\h\|_{H^2}^2\nonumber\\
		&+C\|(\partial_2b^\h,\partial_2b^3)\|_{H^1}^2\|(b^\h,b^3)\|_{H^2}^2\|u^\h\|_{H^2}^2.
	\end{align}
	For the last term of $S^{(i,j,k)}_{3}$, we rewrite it as
	\begin{align*}
		S^{(i,j,k)}_{3}=&\int_{\R^3} b^i (\partial_3^2b^k\partial_3^2u\cdot\nabla{b}^j+\partial_3^2b^j\partial_3^2u\cdot\nabla{b}^k)\dx\\
		&+2\int_{\R^3}b^i(\partial_3^2b^k\partial_3u\cdot\nabla\partial_3{b}^j+\partial_3^2b^j\partial_3u\cdot\nabla\partial_3{b}^k)\dx
		\eqdefa S^{(i,j,k)}_{31}+S^{(i,j,k)}_{32}.
	\end{align*}
	In the same method, applying \eqref{san} with $f=b^i$, $g=\nabla b^j$ (resp. $\nabla b^k$), $h=\partial_3^2b^k$ (resp. $\partial_3^2b^j$), $v=\partial_3^2u$, $(m,n)=(1,2)$, it holds that
	\begin{align*}
		S^{(i,j,k)}_{31}\lesssim&\|\partial_2b^\h\|_{H^1}^\frac12\|b^\h\|_{H^1}^\frac12\|\partial_1b^\h\|_{H^2}\|b^\h\|_{H^2}\|(\partial_2u^\h,\partial_2u^3)\|_{H^2}^\frac12\|(u^\h,u^3)\|_{H^2}^\frac12\\
		\lesssim&\|\partial_2b^\h\|_{H^1}^\frac12\|\partial_1b^\h\|_{H^2}\|(\nablah u^\h,\nablah u^3)\|_{H^2}^\frac12\|b^\h\|_{H^2}^\frac32\|(u^\h,u^3)\|_{H^2}^\frac12,
	\end{align*}
	and using \eqref{san} with $f=b^i$, $g=\partial_3u$, $h=\partial_3^2b^k$ (resp. $\partial_3^2b^j$), $v=\nabla\partial_3{b}^j$ (resp. $\nabla\partial_3{b}^k$), $(m,n)=(2,1)$ again, we have
	\begin{align*}
		S^{(i,j,k)}_{32}\lesssim&\|\partial_2b^\h\|_{H^1}^\frac12\|b^\h\|_{H^1}^\frac12\|(\partial_2u^\h,\partial_2u^3)\|_{H^2}^\frac12\|(u^\h,u^3)\|_{H^2}^\frac12\|\partial_1b^\h\|_{H^2}\|b^\h\|_{H^2}\\
		\lesssim&\|\partial_2b^\h\|_{H^1}^\frac12\|\partial_1b^\h\|_{H^2}\|(\nablah u^\h,\nablah u^3)\|_{H^2}^\frac12\|b^\h\|_{H^2}^\frac32\|(u^\h,u^3)\|_{H^2}^\frac12.
	\end{align*}
	Therefore, by summing up the above estimates, we conclude
	\begin{align}\label{S3}
		S^{(i,j,k)}_{3}\lesssim&\|\partial_2b^\h\|_{H^1}^\frac12\|\partial_1b^\h\|_{H^2}\|(\nablah u^\h,\nablah u^3)\|_{H^2}^\frac12\|b^\h\|_{H^2}^\frac32\|(u^\h,u^3)\|_{H^2}^\frac12\nonumber\\
		\leqq &\frac{1}{100}\|\partial_1b^\h\|_{H^2}^2+C\left(\|\partial_1b^\h\|_{H^2}^2+\|\partial_2b^\h\|_{H^1}^2\right)\|b^\h\|_{H^2}^2\nonumber\\
		&+C\|(\nablah u^\h,\nablah u^3)\|_{H^2}^2\|(u^\h,u^3)\|_{H^2}^2\|b^\h\|_{H^2}^2.
	\end{align}
	As a consequence, combining \eqref{S1}, \eqref{S2}, \eqref{S3} and Young's inequality, we obtain
	\begin{align*}
		\int_{\R^3}\partial_2u^i\partial_3^2b^j\partial_3^2b^k\dx \leqq&\frac{\mathrm{d}}{\dt}\int_{\R^3} b^i\partial_3^2b^j\partial_3^2b^k-b^i\partial_3^2u^j\partial_3^2u^k\dx+\frac{3}{100}\|(\nablah u^\h,\partial_1b^\h)\|_{H^2}^2\nonumber\\
		&+C\left(\|(\nablah u^\h,\partial_1b^\h)\|_{H^2}^2+\|(\partial_2b^\h,\partial_2b^3)\|_{H^1}^2\right)\|(u^\h,b^\h)\|_{H^2}^2\nonumber\\
		&+C\left(\|\nablah u^3\|_{H^2}^2\|u^3\|_{H^2}^2+ \|(\partial_2b^\h,\partial_2b^3)\|_{H^1}^2\|(b^\h,b^3)\|_{H^2}^2\right)\|u^\h\|_{H^2}^2\\	&+C\left(\|\partial_2b^3\|_{H^1}^2\|b^3\|_{H^2}^2+\|(\nablah u^\h,\nablah u^3)\|_{H^2}^2\|(u^\h,u^3)\|_{H^2}^2\right)\|b^\h\|_{H^2}^2.
	\end{align*}
	This completes the proof of Lemma \ref{wild}.
\end{proof}

\section*{Acknowledgement}
This work is partially supported by the Zhejiang Provincial Natural Science Foundation of China LZ25A010003.

\bibliographystyle{elsarticle-num}

\noindent {\bf Acknowledgments.}
This work is partially supported by the National Natural Science Foundation of China  11931010, and Zhejiang Provincial Natural Science Foundation of China LDQ23A010001.

\medskip
\end{document}